\RequirePackage{fix-cm}
\documentclass[a4paper]{elsarticle}

\makeatletter
\def\ps@pprintTitle{%
 \let\@oddhead\@empty
 \let\@evenhead\@empty
 \def\@oddfoot{}%
 \let\@evenfoot\@oddfoot}
\makeatother

\usepackage[english]{babel}
\usepackage[margin=2cm]{geometry}
\usepackage{amssymb, amsthm, amsmath, amsfonts}
\usepackage{xfrac}
\usepackage{bm}
\usepackage{enumitem}
\usepackage[usenames,dvipsnames]{xcolor}
\usepackage[titletoc]{appendix}
\usepackage[hang, flushmargin]{footmisc}

\usepackage{thmtools}

\theoremstyle{plain}
\declaretheorem[name=Proposition, numberwithin=section]{proposition}
\declaretheorem[name=Theorem, sibling=proposition]{theorem}
\declaretheorem[name=Lemma, sibling=proposition]{lemma}
\declaretheorem[name=Corollary, sibling=proposition]{corollary}

\theoremstyle{definition}
\declaretheorem[name=Remark, qed={$\triangle$}, sibling=proposition]{remark}
\declaretheorem[name=Example, qed = {$\bigcirc$}, sibling=proposition]{example}
\declaretheorem[name=Definition, sibling=proposition]{definition}

\declaretheorem[name=Notation, numbered=no]{notation}
\declaretheorem[name=Convention, numbered=no]{convention}

\usepackage[colorlinks=true]{hyperref}
\AtBeginDocument{}
\AtBeginDocument{}
\AtBeginDocument{}
\AtBeginDocument{}

\usepackage{mathrsfs}
\DeclareFontFamily{U}{rsfs}{\skewchar\font127}
\DeclareFontShape{U}{rsfs}{m}{n}{<-6> rsfs5 <6-8> rsfs7 <8-> rsfs10}{}

\renewenvironment{thebibliography}[1]{
	\begin{oldthebibliography}{#1}
	\setlength{\itemsep}{0em}
	\setlength{\parskip}{0em}
}
{
	\end{oldthebibliography}
}

\newcommand{\vertiii}[1]{{\left\vert\kern-0.25ex\left\vert\kern-0.25ex\left\vert #1 \right\vert\kern-0.25ex\right\vert\kern-0.25ex\right\vert}}
\newcommand{\numberthis}{\addtocounter{equation}{1}\tag{\theequation}}
\newcommand{\modre}[1]{#1}

\DeclareMathOperator{\R}{\mathbb{R}}
\DeclareMathOperator*{\E}{\mathbb{E}}
\renewcommand*{\d}{\mathrm{d}}
\DeclareMathOperator*{\Dom}{\mathrm{Dom}\,}

\begin{document}

\begin{frontmatter}
	\title{Stochastic integration with respect to fractional processes in Banach spaces}

	\author[MFF]{Petr \v{C}oupek}
	\ead{coupek@karlin.mff.cuni.cz}

	\author[MFF]{Bohdan Maslowski\corref{cor}}
	\ead{maslow@karlin.mff.cuni.cz}

	\author[UTIA]{Martin Ondrej\'at}
	\ead{ondrejat@utia.cas.cz}

	\cortext[cor]{Corresponding author}
	\address[MFF]{Charles University, Faculty of Mathematics and Physics, \\Sokolovsk\'a 83, Prague 8, 186 75, Czech Republic}
	\address[UTIA]{Czech Academy of Sciences, Institute of Information Theory and Automation, \\Pod Vod\'arenskou V\v{e}\v{z}\'i 4, Prague 8, 182~08, Czech Republic}


	\begin{abstract}
	In the article, integration of temporal functions in (possibly non-UMD) Banach spaces with respect to (possibly non-Gaussian) fractional processes from a finite sum of Wiener chaoses is treated. The family of fractional processes that is considered includes, for example, fractional Brownian motions of any Hurst parameter or, more generally, fractionally filtered generalized Hermite processes. The class of Banach spaces that is considered includes a large variety of the most commonly used function spaces such as the Lebesgue spaces, Sobolev spaces, or, more generally, the Besov and Lizorkin-Triebel spaces. In the article, a characterization of the domains of the Wiener integrals on both bounded and unbounded intervals is given for both scalar and cylindrical fractional processes. In general, the integrand takes values in the space of $\gamma$-radonifying operators from a certain homogeneous Sobolev-Slobodeckii space into the considered Banach space. Moreover, an equivalent characterization in terms of a pointwise kernel of the integrand is also given if the considered Banach space is isomorphic with a subspace of a cartesian product of mixed Lebesgue spaces. The results are subsequently applied to stochastic convolution for which both necessary and sufficient conditions for measurability and sufficient conditions for continuity are found. As an application, space-time continuity of the solution to a parabolic equation of order $2m$ with distributed noise of low time regularity is shown as well as measurability of the solution to the heat equation with Neumann boundary noise of higher regularity.
	\end{abstract}
\end{frontmatter}

\section{Introduction}

In the present article, integration in Banach spaces with respect to fractional processes from a finite sum of Wiener chaoses is treated. The class of fractional processes that is considered consists of stochastic processes $(z_t)_{t\in\mathfrak{T}}$ ($\mathfrak{T}\subseteq\R$ an interval) that are centered second-order stochastic processes for which the equality 
	\begin{equation}
	\label{eq:intro_R_H}
		\E z_sz_t =  \frac{\sigma^2}{2}\left(|s|^{2H} + |t|^{2H}- |t-s|^{2H}\right), \quad s,t\in\mathfrak{T},
	\end{equation}
holds with some $\sigma>0$ and $H\in (0,1)$. There are many processes that satisfy this requirement. In fact, the covariance function of any second-order $H$-self-similar process with stationary increments ($H$-sssi processes for short; see, e.g., \cite{SamTaq94, Tud13}) must necessarily be of the form \eqref{eq:intro_R_H} and in this spirit, the family of fractional Brownian motions (see, e.g., \cite{DecrUstu99,MvN68}) provides a prototypical example. On the other hand, there are many other non-Gaussian processes that can be considered and we name, for example, the family of Rosenblatt processes (see, e.g., \cite{Taqqu11,Tud08}) or the more general family of fractionally filtered generalized Hermite processes (see, e.g., \cite{BaiTaq14}).

The integral of real-valued deterministic functions on an interval $T\subseteq\mathfrak{T}$ with respect to real-valued fractional processes can be defined in a natural way and this construction is standard for concrete cases of the driving process; see, e.g., \cite{DecrUstu99, PipTaqq00, PipTaqq01} for the case of fractional Brownian motions, \cite{Tud08, MaeTud07} for the case of Rosenblatt and Hermite processes, and \cite{AlosMazNua01, BonTud11, CouMas17} for the case of Volterra processes. 
In particular, the integral is initially defined as the map $i: \bm{1}_{[0,t)}\mapsto z_t$ and extended by linearity to step functions. Subsequently, it is shown that there is an It\^o-type isometry for such Stieltjes-type sums and that this isometry can be used for the extension of the integral from step functions to a large abstract (Hilbert) space of admissible integrands $\mathscr{D}^H(T)$. It is here where the covariance structure of the driving process plays a fundamental role. Finally, one usually chooses some subspace of the abstract space $\mathscr{D}^H(T)$ that is suitable for the problem at hand and the Wiener integral is restricted to this space. For example, if the fractional Brownian motion with Hurst parameter $H$ is the integrator and $T$ is a bounded interval, the usual choice is the Lebesgue space $L^\frac{1}{H}(T)$ if $H\geq \sfrac{1}{2}$ or the space of H\"older continuous functions $\mathscr{C}^\theta(T)$ for $\theta>\sfrac{1}{2}-H$ if $H<\sfrac{1}{2}$; see, e.g., \cite{AlosNua03}. 

To aid in this choice, it is desirable to understand the structure of the abstract space of admissible integrands $\mathscr{D}^H(T)$. In this direction, the first main result of the present article is a complete characterization of this space for both bounded an unbounded intervals $T$; namely, it is shown in \autoref{prop:characterisation_of_D} that the space $\mathscr{D}^H(T)$ coincides with the homogeneous Sobolev-Slobodeckii space $\dot{W}^{\frac{1}{2}-H,2}(T)$. Thus, the dependence of the space $\mathscr{D}^H(T)$ on the parameter $H$ is made explicit and it is readily seen that the value $H=\sfrac{1}{2}$ is critical in the following sense: If $H\leq \sfrac{1}{2}$, the space of admissible integrands contains only functions (and the characterization says which functions precisely) while if $H>\sfrac{1}{2}$, then this space is much larger and it contains distributions as well (and again, the characterization says which distributions precisely). We note that while this characterization is proved for the Wiener integral with respect to fractional Brownian motions on bounded intervals in \cite[Theorem 3.3]{Jol07} and the key identity for unbounded intervals is given in \cite[formula (3.4)]{PipTaqq00}, our result applies to a much broader class of integrators and its proof is completely different from the proofs in these two papers.

As far as integration with respect to infinite-dimensional (or, more precisely, $U$-cylindrical) Wiener processes in Banach spaces is concerned, the topic has been the subject to extensive research in the last couple of decades and the reader can find an excellent overview with many references to the most prominent results in the survey article \cite{NeeVerWei15}. See also, e.g., \cite{CioCoxVer18,VerYar16} for some more recent results. On the other hand, while the literature on integration in Banach spaces for infinite-dimensional fractional processes is much more scarce, we refer to the papers \cite{BrzNeeSal12} and \cite{IssRie14} for some results in this direction. In the present article, however, different additional assumptions are put on the driving noise and on the considered Banach space than those usually considered.

In particular, the fractional process is additionally assumed to live in a finite sum of Wiener chaoses (of the largest order $n$); that is, it is assumed that the random variable $z_t$ can be written as a finite sum of multiple Wiener-It\^o integrals (of orders up to $n$) with respect to some isonormal Gaussian process. This requirement allows to use the second main result of the present paper; namely, that of the equivalence of the $p$-th and $q$-th moments of linear combinations from a finite Wiener chaos with coefficients from a Banach space for every $p,q\geq 0$; see \autoref{prop:hypercontractivity}. This result is a generalization of the Kahane-Khinchine inequality, see, e.g., \cite[Theorem 1.3.1]{PG}, and it was announced in \cite[Proposition 2.1]{CouMasOnd18} where it is stated for $p,q\geq 1$. We give the full proof of the general claim in the present article. 

The assumption on the fractional process is then complemented with an assumption on the considered Banach space $X$. In particular, the Banach space $X$ is assumed to have the property that the second moment of a linear combination of elements from $X$ with coefficients from the finite sum of Wiener chaoses in which the noise lives is equivalent to the second moment of a linear combination of the same elements but with different coefficients taken from the Wiener chaos. In particular, this property allows to pass from non-Gaussian coefficients to Gaussian ones. Since it seems that this notion has not been studied in the literature so far, we simply say that the Banach space is $n$-good.

Under these two assumptions, Wiener integration in the Banach space $X$ is treated and the third main result of the paper is given. In particular, it is shown in \autoref{prop:stoch_int_1} that there is a sufficient condition for integrability in $X$ that is formulated in terms of a $\gamma$-radonifying norm of an operator associated with the integrand. If, additionally, the Banach space $X$ has the approximation property, then it is shown in \autoref{prop:char_stoch_int_vector} that this sufficient condition is also necessary. 

While it is clear that all the above mentioned examples of fractional processes live in a finite sum of Winer chaoses, it is a priori not at all clear that there some useful examples of Banach spaces that satisfy the above assumptions. In this direction, we prove that if the Banach space $X$ is isomorphic with a subspace of a product of mixed Lebesgue spaces, then $X$ is both $n$-good for every $n\in\mathbb{N}$ and has the approximation property; see \autoref{prop:Y_is_good_and_approximable}. As a consequence, one can consider many of the commonly used function spaces within our framework and we name, for example, the Lebesgue spaces $L^p(D)$, Sobolev spaces $W^{r,p}(D)$, Besov spaces $B_{p,q}^s(D)$, and the Lizorkin-Triebel spaces $F_{p,q}^s(D)$ (here, $D$ is either $\R^d$, $\R_+^d$, or a bounded $\mathscr{C}^\infty$-domain in $\R^d$; and $p,q\geq 1$, $s\in\R$, and $r>0$); see \autoref{cor:function_spaces} for the precise statement. Note that the above mentioned examples also include some obviously non-UMD spaces such as the space $L^1(\R)$; see, e.g., \cite{Bur01} and the references therein. Moreover, to further aid in applications, we show that if $X$ has the above described structure (as an isomorphic space with a subspace of a product of mixed Lebesgue spaces), then the sufficient and necessary condition for integrability can be formulated in terms of pointwise (Green) kernels; see \autoref{prop:kernel}.

The abstract integration results are then applied to the stochastic convolution integral for which sufficient and necessary conditions for existence are found; see \autoref{prop:stoch_convolution_existence}. The focus is then on the special case in which the integrand takes the form $S(\cdot)\varPhi$ where $\varPhi\in\mathscr{L}(U,X)$ and where $S: [0,\infty)\rightarrow\mathscr{L}(X)$ is a strongly continuous semigroup. In this case, the necessary and sufficient condition for the existence of the convolution integral can be simplified, see \autoref{prop:existence_Yt}, and it is also shown that it already implies the existence of a measurable version of the convolution integral, see \autoref{cor:NSC_for_measurability}. If, additionally, the semigroup is analytic, sufficient conditions for existence of a continuous version of the convolution integral are also given, see \autoref{prop:continuity}. 

Finally, two examples are given. The first example concerns a stochastic parabolic partial differential equation of order $2m$, $m\in\mathbb{N}$, on a bounded $\mathscr{C}^\infty$-domain $D\subseteq\R^d$ with distributed space-time noise. Such equation is treated in \cite{CouMasOnd18} where the case of a regular noise $H>\sfrac{1}{2}$ is considered and where it is shown that if $H>\sfrac{d}{4m}$, then the solution is a space-time continuous random field. In here, we use a similar method to treat the case of a singular noise $H<\sfrac{1}{2}$ and we show that even in this case, space-time continuity of the solution occurs if $H>\sfrac{d}{4m}$. More precisely, the formal parabolic equation is formulated as a stochastic Cauchy problem in the space $L^p(D)$ for a suitable choice of $p\geq 2$ and its solution is sought in the mild form, i.e. as a stochastic convolution integral. Subsequently, time continuity of the convolution integral in the domain of a fractional power of the differential operator is shown and since this space is a subspace of a certain Bessel potential space, the Sobolev embedding is used to prove continuity in the spatial variable; see \autoref{prop:space-time_continuity} for the precise statement of the result. We note that, roughly speaking, the larger the parameter $p$ is considered in this example, the better regularity of the solution is obtained. For some other related results on stochastic partial differential equations with distributed noise, we refer, for example, to \cite{BonTud11,BrzNee03,BrzNeeSal12, CouMas17, CouMasOnd18, DunMasDun02,PorVer19,TinTudVie03} and the references therein.

The second example concerns the stochastic heat equation on a bounded $\mathscr{C}^\infty$-domain $D\subseteq\R^d$ of finite surface measure that is perturbed by space-time noise through the boundary of the domain. In this example, it is assumed that $H>\sfrac{1}{2}$. The equation is again treated as a stochastic Cauchy problem in a suitable $L^p(D)$ space but in this example, the situation is in a certain sense reversed when compared with the first example - in here, it is desirable to choose the parameter $p$ as small as possible. This becomes apparent if one tries to apply \cite[Corollary 3.1]{CouMasOnd18} to this situation - while similar conditions are obtained, there is the barrier $pH\geq 1$ which comes from the use of Minkowski's inequality. To overcome this difficulty, we find a Green kernel for the integrand and appeal to \autoref{prop:kernel} instead. Consequently, we show that if $H>\sfrac{d}{2}-\sfrac{1}{2}$, then it is possible to choose the parameter $p\in (1,2]$ so that measurability of the solution in the space $L^p(D)$ occurs; see \autoref{prop:heat_bdry_noise_measurability} for the precise statement of the result. For some related results on stochastic partial differential equations with boundary noise, we refer, for example, to \cite{AlosBon02,BrzGolPesRus15,DunMasDun02,FabGol09,LinVer20,Mas95,SchVer11} and the references therein.

\textit{Organization of the article.} In \autoref{sec:preliminaries}, we recall some notions from Gaussian analysis and prove the hypercontractivity result in \autoref{prop:hypercontractivity}. Subsequently, the notion of fractional processes is formalized and some examples given. Wiener integration for these processes is treated here as well and the characteriztion of the abstract space of admissible integrands is proved in \autoref{prop:characterisation_of_D}. In \autoref{sec:Wiener_integration_cylindrical_processes}, Wiener integration for cylindrical fractional processes is analysed and the section is split into two subsections. In \autoref{sec:Wiener_integra_scalar_case}, Wiener integrability for cylindrical processes in the scalar case is characterized and in \autoref{sec:Wiener_integration_vector_case}, Wiener integrability for cylindrical process in the vector case is treated. In particular, while weak integrability is defined and characterized in \autoref{sec:weak_integrability}, strong integrability is defined and characterized in \autoref{sec:strong_integrability}. It is also in the last mentioned \autoref{sec:strong_integrability}, where $n$-good Banach spaces that have the approximation property are considered. \autoref{sec:stoch_conv} is devoted to stochastic convolution and \autoref{sec:examples} contains the two examples of stochastic partial differential equations to which our results are applied. The results on homogeneous fractional Sobolev spaces needed for our analysis are collected in \autoref{sec:appendix}.

\section{Preliminaries}
\label{sec:preliminaries}

In this section, some preliminaries from Gaussian analysis and fractional processes are given. Throughout the paper, the following notation is used. 

\begin{notation}
The notation $A\lesssim B$ means that there is a finite positive constant $c$ such that $A\leq cB$. Similarly, $A\eqsim B$ means that there are two finite positive constants $c_1$ and $c_2$ such that $c_1A\leq B\leq c_2A$ and $A\propto B$ means that there is a finite positive constant $c$ such that $A=cB$. This notation is used whenever the precise values of the constants are not important.
\end{notation}

\subsection{Equivalence of moments on a finite Wiener chaos}

We begin with a general setting. Let $V$ be a real separable Hilbert space and assume that $(\Omega,\mathscr{F},\mathbb{P})$ is a probability space with a $V$-isonormal Gaussian process $(W(v))_{v\in V}$ defined on it. It is assumed that the sigma field $\mathscr{F}$ is generated by this isonormal Gaussian process and augmented by $\mathbb{P}$-zero sets. Denote by $H_n$ the $n^\mathrm{th}$ \textit{Hermite polynomial} that is defined by 
	\begin{equation*}
		H_n(x) := \frac{(-1)^n}{n!}\mathrm{e}^{\frac{x^2}{2}}\frac{\d^n}{\d x^n}\left(\mathrm{e}^{-\frac{x^2}{2}}\right), \quad x\in\R.
	\end{equation*}
The $n^\mathrm{th}$ \textit{Wiener chaos}, denoted here by $\mathscr{H}_n$, is the closed linear subspace of the space $L^2(\Omega)$ generated by the linear span $\{H_n(W(v))\,|\,v\in V, \|v\|_V=1\}$. For a thorough analysis of these notions, we refer, for example, to the monograph \cite{Nua06} and the references therein. An important feature of the spaces $\mathscr{H}_n$ is that their elements have equivalent moments. A generalization of this property to random variables with values in normed linear spaces is central to the present paper and we it is stated precisely in \autoref{prop:hypercontractivity}. The result improves that of \cite[Proposition 2.1]{CouMasOnd18}.

\begin{definition}
\label{def:finite_Wiener_chaos}
 In this paper, by a \textit{finite Wiener chaos} we mean the space $\mathscr{H}^{\oplus n}:=\bigoplus_{i=0}^n\mathscr{H}_i$ for some $n\in\mathbb{N}$. We say that a stochastic process $z: \mathfrak{T}\rightarrow L^2(\Omega)$, where $\mathfrak{T}\subseteq\R$ is an interval, \textit{lives in a finite Wiener chaos} if there exists a non-negative integer $n$, such that $z_t\in \mathscr{H}^{\oplus n}$ for every $t\in \mathfrak{T}$. 
\end{definition}

\begin{proposition}
\label{prop:hypercontractivity}
Let $p,q\in (0,\infty)$ and $n\in\mathbb{N}_0$. Then there exists a finite positive constant $C_{p,q,n}$ such that the inequality 
	\begin{equation*}
		\E\left(\left\|\sum_{j=1}^m\xi_jx_j\right\|_{\mathcal{B}}^{q}\right)^\frac{1}{q} \leq C_{p,q,n} \left(\E\left\|\sum_{j=1}^m\xi_jx_j\right\|_{\mathcal{B}}^p\right)^\frac{1}{p}
	\end{equation*} 
holds for every normed linear space $\mathcal{B}$, $m\in\mathbb{N}$, $\{x_j\}_{j\leq m}\subset \mathcal{B}$, and every $\{\xi_j\}_{j\leq m} \subset \mathscr{H}^{\oplus n}$.
\end{proposition}

\begin{proof}
According to \cite[Theorem 3.2.10 (i)]{PG} (in particular the reference to \cite[Theorem 3.2.5]{PG}), there exist a finite positive constant $C_{p,q,n}$ such that the inequality
\begin{equation}
	\label{prop:hypercontractivity_1}
	\left(\E\|X\|_\mathcal{B}^{q}\right)^{\frac{1}{q}}\leq C_{p,q,n}\left(\E\|X\|_\mathcal{B}^{p}\right)^{\frac{1}{p}}
\end{equation}
is satisfied for every Banach space $\mathcal{B}$, every $1<p<q<\infty$, and for every random variable $X$ from the linear span of
	\begin{equation*}
			\left\{x\prod_{j=1}^\infty W(e_j)^{\alpha_j}\,\Bigg|\,\,x\in \mathcal{B},\,\,\alpha\in\mathbb{N}_0^{\mathbb{N}},\,\|\alpha\|_{\ell^1}\leq n\right\}
	\end{equation*}
where $\{e_j\}_{j\in\mathbb{N}}$ is an orthonormal basis of the Hilbert space $V$. By linearity and Fernique's theorem, inequality \eqref{prop:hypercontractivity_1} holds also for every random variable $X$ from the linear span of
	\begin{equation*}
		\left\{x\prod_{j=1}^\infty W(\xi_j)^{\alpha_j}\,\Bigg|\,\,x\in \mathcal{B},\,\alpha\in\mathbb{N}_0^{\mathbb{N}},\,\|\alpha\|_{\ell^1}\le n,\,\xi\in V^{\mathbb{N}}\right\}.
	\end{equation*}
This means that the $L^r(\Omega)$-norms for $1<r<\infty$ are equivalent on the linear span of the set 
	\begin{equation*}
\left\{\prod_{j=1}^\infty W(\xi_j)^{\alpha_j}\,\Bigg|\,\,x\in \mathcal{B},\,\alpha\in\mathbb{N}_0^{\mathbb{N}},\,\|\alpha\|_{\ell^1}\leq n,\,\xi\in V^{\mathbb{N}}\right\}
	\end{equation*}
and therefore, the closure of this set in the space $L^r(\Omega)$ for any $1<r<\infty$ coincides with $\mathscr{H}^{\oplus n}$. This follows, for example, by the remark on page 6 after \cite[Theorem 1.1.1]{Nua06}. We have thus proved that inequality \eqref{prop:hypercontractivity_1} holds for every random variable $X$ from the linear span of
	\begin{equation*}
		\left\{x\eta\,|\,x\in \mathcal{B},\,\eta\in\mathscr{H}^{\oplus n}\right\}.
	\end{equation*}
To prove the claim for $0<p<q<\infty$, $q\in(1,\infty)$, the trick in the remark that follows \cite[Theorem 3.2.2]{PG} is used. That is, the Cauchy-Schwarz inequality is applied to $\|X\|^{\alpha q}_\mathcal{B}\|X\|^{(1-\alpha)q}_\mathcal{B}$ to obtain the inequality
	\begin{equation*}
		\|X\|_{L^q(\Omega;\mathcal{B})}\le\|X\|_{L^p(\Omega;\mathcal{B})}^\alpha\|X\|^{1-\alpha}_{L^{2q-p}(\Omega;\mathcal{B})}\le C^{1-\alpha}_{q,2q-p,n}\|X\|_{L^p(\Omega;\mathcal{B})}^\alpha\|X\|^{1-\alpha}_{L^q(\Omega;\mathcal{B})}
	\end{equation*}
where $\alpha$ is defined by $2\alpha q=p$, from which the inequality
	\begin{equation*}
		\|X\|_{L^q(\Omega;\mathcal{B})}\le C^{-1+\alpha^{-1}}_{q,2q-p,n}\|X\|_{L^p(\Omega;\mathcal{B})}
	\end{equation*}
follows. The case $0<p<q\le 1$ then follows by combining Jensen's inequality and the previous step as
	\begin{equation*}
		\|X\|_{L^q(\Omega;\mathcal{B})}\le\|X\|_{L^2(\Omega;\mathcal{B})}\leq C_{2,4-p,n}^{-1+\frac{4}{p}}\|X\|_{L^p(\Omega;\mathcal{B})}.
	\end{equation*}
\end{proof}

\begin{remark}
\label{rem:hypercontractivity_p=0}
We note that \autoref{prop:hypercontractivity} also holds for $0=p<q<\infty$ in the sense the linear span of $\{x\eta\,|\, x\in\mathcal{B},\eta\in\mathscr{H}^{\oplus n}\}$ when equipped with the topology of $L^0(\Omega;\mathcal{B})$ is uniformly continuously embedded in itself when equipped with the topology of $L^q(\Omega;\mathcal{B})$. Indeed, it follows by using the Cauchy-Schwarz inequality and \autoref{prop:hypercontractivity} that there exists a finite positive constant $C_{q,n}$ such that the inequality
	\begin{equation*}
		\E\|X\|_{\mathcal{B}}^q \leq \left[\mathbb{P}(\|X\|_{\mathcal{B}}>R)\right]^{\frac{1}{2}}\|X\|_{L^{2q}(\Omega;\mathcal{B})}^q + \varepsilon^q \leq C_{q,n}\left[\mathbb{P}(\|X\|_{\mathcal{B}}>R)\right]^{\frac{1}{2}}\E\|X\|^q_{\mathcal{B}} + R^q
	\end{equation*}
is satisfied for every $R>0$. Consequently, there is the following implication:
	\begin{equation*}
		\mathbb{P}(\|X\|_{\mathcal{B}}>R) \leq \frac{1}{4C_{q,n}^2} \quad \mbox{for some } R>0 \quad \implies \quad \|X\|_{L^q(\Omega;\mathcal{B})} \leq 2^\frac{1}{q} R.
	\end{equation*}
\end{remark}

\subsection{Fractional processes}

Let $\mathfrak{T}$ be the real axis $\R$, the interval $[0,\infty)$, or the interval $[0,\tau]$ for some $\tau>0$. Let $z=(z_t)_{t\in \mathfrak{T}}$ be a stochastic process on $(\Omega,\mathscr{F},\mathbb{P})$ that is centered, has finite variance, and such that there is $H\in (0,1)$, the so-called \textit{Hurst index}, for which the equality
	\begin{equation}
	\label{eq:covariance_of_z}
		\E z_sz_t = \sigma^2R_H(s,t)
	\end{equation}
is satisfied for every $s,t\in \mathfrak{T}$ with the function $R_H(s,t)$ defined by
	\begin{equation}
	\label{eq:covariance_function}
		R_H(s,t) := \frac{1}{2}\left(|s|^{2H} + |t|^{2H} - |t-s|^{2H}\right)
	\end{equation}
and $\sigma>0$ a constant. A process $z$ that satisfies the above conditions is called an \textit{$H$-fractional process} in this article.

\begin{remark}
By formulas \eqref{eq:covariance_of_z} and \eqref{eq:covariance_function}, we have that for an $H$-fractional process $z$ the equality
	\begin{equation*}
		\E(z_t-z_s)^2 = |t-s|^{2H}\sigma^2
	\end{equation*}
holds for every $s,t\in \mathfrak{T}$. Hence, without any further assumptions, the process $z$ admits a measurable version by \cite[Theorem 2.6]{Doob90} because it is continuous in mean square and therefore also in probability. If, moreover, $H\in (\sfrac{1}{2},1)$, the process $z$ also admits a version with H\"older continuous sample paths up to the order $H-\sfrac{1}{2}$ by Kolmogorov's continuity criterion; see, e.g.,  \cite[Theorem 39.3]{Bau96}. However, if it is assumed that the process $z$ lives in a finite Wiener chaos, then it has a continuous version even in the singular case $H\in (0,\sfrac{1}{2}]$. More precisely, if $z$ lives in a finite Wiener chaos $\mathscr{H}^{\oplus n}$, we have that the inequality
	\begin{equation*}
		\E(z_t-z_s)^{q} \lesssim \left[\E(z_t-z_s)^2\right]^q = |t-s|^{Hq} \sigma^q
	\end{equation*}
holds for every $q>0$ and $s,t\in \mathfrak{T}$ by \autoref{prop:hypercontractivity} and it follows by Kolmogorov's criterion that $z$ has a version with H\"older continuous sample paths of every order smaller that $H$, cf. \cite[Remark 2.1]{CouMasOnd18}.
\end{remark}

\begin{example}
The class of fractional processes from a finite Wiener chaos that satisfy the above conditions for the process $z$ is quite rich and some examples are given here. Let us first specify the general setting. We assume that $(\Omega,\mathscr{A},\mathbb{P})$ is a probability space with a Wiener process $(W_t)_{t\in\R}$ defined on it. Constructed in the usual manner, the first order Wiener-It\^o integral is an $L^2(\R)$-isonormal process defined on this probability space and it is assumed that this process generates the $\sigma$-field $\mathscr{A}$. 

Note first that every $H$-self-similar process with stationary increments ($H$-sssi processes for short) from a finite Wiener chaos can be considered. This follows by, for example, Lemma 7.2.1 of \cite{SamTaq94}. Clearly, the main examples are \textit{fractional Brownian motions}; however there are many other $H$-sssi processes from a finite Wiener chaos that have been considered in the literature. 

We can mention, for example, the family of the \textit{fractionally filtered generalized Hermite processes}, that is introduced and analysed in \cite{BaiTaq14}. This family includes, among others, the processes $z_k^{\alpha,\beta}$ that are defined by 
	\begin{equation*}
		z_k^{\alpha,\beta} (t) := C_{\alpha,\beta,k}\int_{\R^k}^\prime \left\{\int_{\R}k_t^\beta(u)\prod_{j=1}^k(u-y_j)_+^{\frac{\alpha}{k}}\d{u}\right\}\d{W}_{\bm{y}}^{\otimes k}, \quad t\geq 0,
	\end{equation*}
where the kernel $k_t^\beta$ is given by
	\begin{equation*}
		k_t^\beta (u) := \begin{cases}
										\frac{1}{\beta}\left[(t-u)_+^\beta-(-u)_+^\beta\right],& \quad \beta\neq 0,\\
										\bm{1}_{(0,t]}(u), & \quad \beta = 0,
								   \end{cases}
	\end{equation*}
and where the parameters $\alpha$ and $\beta$ satisfy
	\begin{equation*}
		-1 <-\alpha - \frac{k}{2}-1<\beta<-\alpha - \frac{k}{2}<\frac{1}{2}.
	\end{equation*}
The constant $C_{\alpha,\beta,k}$ is a normalizing constant that ensures that $\E [z_k^{\alpha,\beta}(1)]^2=1$ and the integral $\int_{\R^k}^\prime (\ldots)\d{W}_{\bm{y}}^{\otimes k}$ is the Wiener-It\^o multiple integral of order $k$; see, e.g.,  \cite{Nua06} or \cite{NouPec12}. It follows by \cite[Theorem 3.27]{BaiTaq14} that the process $z_k^{\alpha,\beta}$ is an $H$-sssi process with the parameter $H$ that is given by 
	\begin{equation*}
		H=\alpha+\beta+\frac{k}{2}+1
	\end{equation*}
and that belongs to the interval $(0,1)$. Moreover, by its construction as a Wiener-It\^o multiple integral of a deterministic function, the process $z_k^{\alpha,\beta}$ lives in the $k$\textsuperscript{th} Wiener chaos. 

It should be noted that the above class includes some well-known stochastic processes. In particular, the process $z_1^{\alpha,\beta}$ is the fractional Brownian motion with the Hurst parameter $H=\alpha + \beta+\sfrac{3}{2}\in (0,1)$ for every $\alpha$ and $\beta$ that satisfy 
	\begin{equation*}
		-1 < -\alpha -\frac{3}{2} < \beta < - \alpha -\frac{1}{2}<\frac{1}{2}.
	\end{equation*}
This is because the fractional Brownian motion is the only $H$-sssi process in the first Wiener chaos in the sense of finite-dimensional distributions, cf. \cite[Proposition 1.1]{Tud13}. But already in the second Wiener chaos, there are infinitely many $H$-sssi processes; see the paper \cite{MaeTud12} for the discussion of this phenomenon.

Another particular case of the above class is the family of the processes $z_k^{\alpha, 0}$. These processes are the much studied \textit{Hermite processes} of order $k$ with the Hurst parameter $H=\alpha + \frac{k}{2}+1\in (\sfrac{1}{2},1)$ for every $\alpha$ that satisfies
	\begin{equation*}
		-\frac{k}{2}-\frac{1}{2} <\alpha < -\frac{k}{2}.
	\end{equation*}
Note that the family of Hermite processes also includes the family of \textit{Rosenblatt processes} that has received considerable attention in the last couple of years. See, for example, \cite[section 3.1]{Tud13} and the references contained therein for further properties of Hermite processes. See also the seminal paper \cite{Tud08} for stochastic analysis of the Rosenblatt process. 
\end{example}

\subsection{Wiener integration for scalar fractional processes}
\label{sec:Wiener_integration_one-dim}

The integral of deterministic functions with respect to a $H$-fractional process $(z_t)_{t\in\mathfrak{T}}$ is defined in the sequel. The construction of the integral follows the approach used in the case of fractional Brownian motions (and, more generally, Volterra processes, cf. e.g. \cite{AlosMazNua01, CouMas17, BonTud11}) and somewhat embodies the idea that only the covariance structure of the driving process $z$ is needed since the integral is constructed as the limit of Stieltjes-type sums in the space of square integrable random variables. 

Let $T\subseteq\mathfrak{T}$ and denote by $\mathscr{E}(T)$ the linear space of deterministic step functions whose support is contained in the interval $T$; that is, a function $f\in \mathscr{E}(T)$ satisfies the equality
	\begin{equation}
	\label{eq:step_function}
		f = \sum_{j=1}^{n} f_j\bm{1}_{[t_{j-1}, t_{j})}
	\end{equation}
with some $n\in\mathbb{N}$, some set $\{t_j\}_{j\leq n}\subset T$ such that $t_0<t_1<\ldots <t_n$, and a set $\{f_j\}_{j\leq n}\subset \R$. Note that such a function can be extended by zero outside of the interval $T$ and this is done without an explicit comment whenever needed. Now, for a step function $f\in\mathscr{E}(T)$ that is given by formula \eqref{eq:step_function} set
	\begin{equation}
	\label{eq:step_integral}
		i_T(f):=\sum_{j=1}^nf_j(z_{t_{j}}-z_{t_{j-1}})
	\end{equation}
and consider the operator $\mathscr{K}^*_H: \mathscr{E}(T)\rightarrow L^2(\R)$ that is defined by $(\mathscr{K}_{\sfrac{1}{2}}^*f)(r) := f(r)$ and by 
	\begin{equation*}
		(\mathscr{K}_H^* f)(r) := \begin{cases}
													\displaystyle \frac{1}{c_H}\int_r^\infty \left[f(u)-f(r)\right] (u-r)^{H-\frac{3}{2}}\d{u}, & \quad H\in (0,\sfrac{1}{2}),\\[10pt]
													\displaystyle \frac{1}{c_H}\int_r^\infty f(u)(u-r)^{H-\frac{3}{2}}\d{u}, & \quad H\in (\sfrac{1}{2},1).
												\end{cases}
	\end{equation*}
The constant $c_H$ in the above definition is given by
	\begin{equation*}
		c_H^2 := \int_0^\infty \left[(1+s)^{H-\frac{1}{2}} - s^{H-\frac{1}{2}}\right]^2\d{s} + \frac{1}{2H}.
	\end{equation*}
The reason for considering the operator $\mathscr{K}^*_H$ is that the equality
	\begin{equation*}
		R_H(s,t) = \langle \mathscr{K}_H^*\bm{1}_{[0,s)},\mathscr{K}_H^* \bm{1}_{[0,t)}\rangle_{L^2(\R)}
	\end{equation*}
is satisfied for every $s,t\in T$. Here, the indicator function $\bm{1}_{[0,\tau)}$ is interpreted as $-\bm{1}_{[\tau,0)}$ if $\tau <0$. As a consequence, it follows that the equality
	\begin{equation}
	\label{eq:Ito_isometry}
		\|i_T(f)\|_{L^2(\Omega)}^2 = \sigma^2\|\mathscr{K}_H^* f\|_{L^2(\R)}^2
	\end{equation}
is satisfied for every step function $f\in\mathscr{E}(T)$. Moreover, since the operator $\mathscr{K}^*_H$ is injective, the bilinear form defined for $f,g\in\mathscr{E}(T)$ by
	\begin{equation*}
		\langle f,g\rangle_{\mathscr{D}^H(T)}:= \sigma^2\langle\mathscr{K}^*_Hf,\mathscr{K}^*_Hg\rangle_{L^2(\R)}
	\end{equation*}
is an inner product. By equality \eqref{eq:Ito_isometry}, the linear map $i_T$ is an isometry between the space $\mathscr{E}(T)$ endowed with the norm $\|\,\cdot\,\|_{\mathscr{D}(T)}$ and the linear span $\{i_T(f), f\in\mathscr{E}(T)\}$ that is endowed with the norm $\|\cdot\|_{L^2(\Omega)}$. This isometry is now extended to a linear isometry between the completion of $\mathscr{E}(T)$ with respect to $\|\cdot\|_{\mathscr{D}^H(T)}$, that is denoted by $\mathscr{D}^H(T)$ in this paper, and the closure of $\mathrm{span}\,\{i_T(f), f\in\mathscr{E}(T)\}$ in the norm $\|\cdot\|_{L^2(\Omega)}$. The extension is again denoted by $i_T$ and it satisfies formula \eqref{eq:Ito_isometry} for every $f\in\mathscr{D}^H(T)$. Note that this relationship reduces to the classical It\^o isometry for the Wiener integral if $z$ is a Wiener process. The space $\mathscr{D}^H(T)$ is called the \textit{space of admissible integrands} for the process $z$. An element $f$ of $\mathscr{D}^H(T)$ is said to be \textit{Wiener integrable on $T$ with respect to the process $z$} and the square integrable random variable $i_T(f)$ is called its \textit{Wiener integral with respect to $z$}.

In the following result, a complete characterization of the space $\mathscr{D}^H(T)$ is given. This result in the case when $T$ is a bounded interval and $z$ is a fractional Brownian motion is proved in \cite[Theorem 3.3]{Jol07} and the key identity for the case when $T$ is an unbounded interval is also given in \cite[formula (4.3)]{PipTaqq00}. Here we prove the result in this general setting for any fractional process $z$ (in particular, without the assumption of Gaussianity) and on both bounded and unbounded intervals. Moreover, the method of proof given here is different from the one given in \cite{Jol07}. 

\begin{proposition}
\label{prop:characterisation_of_D}
Let $H\in (0,1)$. Then there is the equality
	\begin{equation*}
		\mathscr{D}^H(T) = \dot{W}^{\frac{1}{2}-H,2}(T)
	\end{equation*}
\modre{with the following equality between the norms
	\begin{equation*}
		\|\,\cdot\,\|_{\mathscr{D}^H(T)} = C_{\sigma,H} \|\,\cdot\,\|_{\dot{W}^{\frac{1}{2}-H,2}(T)}
	\end{equation*}
where $C_{\sigma,H}$ is a finite positive constant.} Here, $\dot{W}^{\frac{1}{2}-H,2}(T)$ is the homogeneous fractional Sobolev space (see \autoref{app:Ws_R} for the case when $T$ is the real axis and see \autoref{app:Ws_I} for the case when $T$ is the positive real axis or a bounded interval).
\end{proposition}

\begin{proof}
For $c\in (-1,1)$, define 
	\begin{equation*}
		h_c(x) := |x|^{-c}\bm{1}_{(-\infty,0)}(x)
	\end{equation*}
which will be understood as a distribution below. Note that if $c\in (0,1)$, the Fourier transform $\hat{h}_c$ of $h_c$ is given by
	\begin{equation}
	\label{eq:fourier_h}
		\hat{h}_c(x) = A_c|x|^{c-1} + \mathrm{i} B_c\mathrm{sgn}(x)|x|^{c-1}
	\end{equation}
where the constants $A_c$ and $B_c$ are given by
	\begin{equation*}
		A_c := \frac{1}{\sqrt{2\pi}}\sin\left(\frac{\pi c}{2}\right)\Gamma(1-c), \quad B_c := \frac{1}{\sqrt{2\pi}}\cos\left(\frac{\pi c}{2}\right)\Gamma(1-c).
	\end{equation*}
If $H=\sfrac{1}{2}$, the claim is clear. Assume therefore that $H\neq \sfrac{1}{2}$ and let $f\in\mathscr{E}(T)$. If $H\in (0,\sfrac{1}{2})$, set $a:=\sfrac{1}{2}-H$ and note in this case the operator $\mathscr{K}_H^*$ can be expressed as
	\begin{equation*}
		c_H\mathscr{K}_H^*f = \frac{1}{a}h_a*f'
	\end{equation*}
where $*$ denotes convolution. Moreover, in this case, the chain of equalities
	\begin{equation*}
			\frac{1}{a}\|h_a* f^\prime\|_{L^2(\R)} = \frac{\sqrt{2\pi}}{a} \|\hat{h}_a\hat{f'}\|_{L^2(\R;\mathbb{C})} = \frac{\sqrt{2\pi}}{a}\|x\hat{h}_a\hat{f}\|_{L^2(\R;\mathbb{C})} = \frac{\Gamma\left(\frac{1}{2}+H\right)}{\frac{1}{2}-H}\||x|^{\frac{1}{2}-H}\hat{f}\|_{L^2(\R;\mathbb{C})}
	\end{equation*}
holds by using Plancheler's theorem and equality \eqref{eq:fourier_h}. On the other hand, if $H\in (\sfrac{1}{2},1)$, set $b:=\sfrac{3}{2}-H$ and note that in this case the operator $\mathscr{K}_H^*$ can be expressed as
	\begin{equation*}
		c_H\mathscr{K}_H^*f = h_b* f.
	\end{equation*}
Moreover, in this case, the chain of equalities   
	\begin{equation*}
		\|h_b*f\|_{L^2(\R)} = \sqrt{2\pi}\|\hat{h}_b\hat{f}\|_{L^2(\R;\mathbb{C})} = \Gamma\left(H-\frac{1}{2}\right) \||x|^{\frac{1}{2}-H}\hat{f}\|_{L^2(\R;\mathbb{C})}
	\end{equation*}
holds by similar arguments as above. Consequently, it follows for any $H\in (0,1)$, $H\neq \sfrac{1}{2}$, that if $\modre{C}_{\sigma,H}$ is the constant defined by $$\modre{C}_{\sigma,H}:= \frac{\sigma}{c_H}\left|\Gamma\left(H-\frac{1}{2}\right)\right|$$ where the standard convention $\Gamma(z):=\Gamma(z+1)/z$ for $z<0$, $z\not\in\mathbb{Z}$, is used, there is the chain of equalities
	\begin{equation*}
		\|f\|_{\mathscr{D}^H(T)} = \sigma \|\mathscr{K}_H^* f\|_{L^2(\R)} = \modre{C}_{\sigma,H} \||x|^{\frac{1}{2}-H}\hat{f}\|_{L^2(\R;\mathbb{C})} =  \modre{C}_{\sigma,H}\|f\|_{\dot{W}^{\frac{1}{2}-H,2}(\R)} =  \modre{C}_{\sigma,H}\|f\|_{\dot{W}^{\frac{1}{2}-H,2}(T)}.
	\end{equation*}
Since the set of step functions $\mathscr{E}(T)$ is dense in $\dot{W}^{\frac{1}{2}-H,2}(T)$, the claim follows. 
\end{proof}

\section{Wiener integration for cylindrical fractional processes}
\label{sec:Wiener_integration_cylindrical_processes} 

In this section, Wiener integration with respect to possibly infinite-dimensional fractional processes is treated. Initially, the notion of a cylindrical fractional process is defined.

\begin{definition}
Let $(\Omega,\mathscr{F},\mathbb{P})$ be a probability space. Let $H\in (0,1)$ and let $\mathcal{U}$ be a separable Hilbert space. A \textit{$\mathcal{U}$-cylindrical $H$-fractional process} is a collection $(Z_t)_{t\in\mathfrak{T} }$ of bounded linear operators $Z_t: \mathcal{U}\rightarrow L^2(\Omega)$ such that for every $u\in \mathcal{U}$, $(Z_t(u))_{t\in{\mathfrak{T}}}$ is a one-dimensional fractional process defined on the probability space $(\Omega,\mathscr{F},\mathbb{P})$ and for which the equality
	\begin{equation*}
		\E Z_s(u)Z_t(v) = R_H(s,t)\langle u,v\rangle_\mathcal{U}
	\end{equation*}
is satisfied for every $s,t\in{\mathfrak{T}}$ and $u,v\in \mathcal{U}$. Here, $R_H$ is the covariance function given by formula \eqref{eq:covariance_function}.
\end{definition}

The following definition is a generalization of \autoref{def:finite_Wiener_chaos} to cylindrical processes.

\begin{definition}
Let $H\in (0,1)$, and let $\mathcal{U}$ be a separable Hilbert space. Let $(Z_t)_{t\in \mathfrak{T}}$ be a $\mathcal{U}$-cylindrical $H$-fractional process. If there exists $n\in\mathbb{N}$ such that $Z_t(u)$ belongs a finite Wiener chaos $\mathscr{H}^{\oplus n}$ for every $u\in\mathcal{U}$ and every $t\in \mathfrak{T}$, then the process $Z$ is said to \textit{live in a finite Wiener chaos}.
\end{definition}

The following lemma is proved by a standard approximation argument and it will be useful in the sequel.

\begin{lemma}
\label{lem:Ito_isometry_cylindrical}
Let $H\in (0,1)$ and let $\mathcal{U}$ be a separable Hilbert space. Let $(Z_t)_{t\in \mathfrak{T}}$ be a $\mathcal{U}$-cylindrical $H$-fractional process. Furthermore, let $T\subseteq\mathfrak{T}$ be an interval. Then the equality
	\begin{equation*}
		\E \left[\int_Tg_1\d{Z}(u_1)\right]\left[\int_Tg_2\d{Z}(u_2)\right] = \langle g_1,g_2\rangle_{\mathscr{D}^H(T)} \langle u_1,u_2\rangle_\mathcal{U}
	\end{equation*}
is satisfied for every $g_1,g_2\in\mathscr{D}^H(T)$ and every $u_1,u_2\in \mathcal{U}$. 
\end{lemma}

Let us fix $H\in (0,1)$, a separable Hilbert space $U$, an interval $T\subseteq\mathfrak{T}$, and a $U$-cylindrical $H$-fractional process $(Z_t)_{t\in T}$ for the remainder of this section. Let us also fix the following notation:

\begin{notation}
For two Banach spaces $\mathcal{X}$ and $\mathcal{Y}$, we denote by $\mathscr{L}(\mathcal{X};\mathcal{Y})$ the space of bounded linear operators $\mathcal{X}\rightarrow\mathcal{Y}$ and by $\|\cdot\|_{\mathscr{L}(X;Y)}$ the operator norm. For two Hilbert spaces $\mathcal{U}$ and $\mathcal{V}$, we denote by $\mathscr{L}_2(\mathcal{U};\mathcal{V})$ the space of Hilbert-Schmidt operators $\mathcal{U}\rightarrow\mathcal{V}$ and by $\|\cdot\|_{\mathscr{L}_2(\mathcal{U};\mathcal{V})}$ the Hilbert-Schmidt norm. For a Hilbert space $\mathcal{U}$ and a Banach space $\mathcal{X}$, we denote by $\gamma(\mathcal{U};\mathcal{X})$ the space of $\gamma$-radonifying operators $\mathcal{U}\rightarrow\mathcal{X}$ and by $\|\cdot\|_{\gamma(\mathcal{U};\mathcal{X})}$ the $\gamma$-radonifying norm.
\end{notation}

\begin{notation}
The space $\mathscr{D}^H(T)\otimes_2 U$ is denoted by $\mathscr{D}^H(T;U)$ in the rest of the paper.
\end{notation}

\subsection{The scalar case}
\label{sec:Wiener_integra_scalar_case}

We begin with Wiener integration in the case when the target space is one-dimensional. Denote by $I_T$ the unique isometry from the space $\mathscr{D}^H(T;U)$ to the space $L^2(\Omega)$ that satisfies the equality
	\begin{equation*}
		I_{T}(g\otimes u) = \int_T g\,\d{Z}(u)
	\end{equation*}
for every $g\in\mathscr{D}^H(T)$ and every $u\in U$ where the integral on the right is the integral of $g$ with respect to the one-dimensional $H$-fractional process $Z(u)$ as defined in \autoref{sec:Wiener_integration_one-dim}.

\begin{definition}
\label{def:stoch_int_scalar}
A bounded linear operator $A: U\rightarrow \mathscr{D}^H(T)$ is said to be \textit{(Wiener) integrable with respect to the process $Z$} if there exists a random variable $\xi\in L^2(\Omega)$ such that the equality
	\begin{equation*}
		\E I_{T}(g\otimes u)\xi = \langle g,Au\rangle_{\mathscr{D}^H(T)}
	\end{equation*}
is satisfied for every $g\in\mathscr{D}^H(T)$ and $u\in U$. 
\end{definition}

There is the following characterization of integrability in the scalar case.

\begin{proposition}
\label{prop:char_stoch_int_scalar}
Let $A\in\mathscr{L}(U;\mathscr{D}^H(T))$. The operator $A$ is integrable with respect to the process $Z$ if and only if $A$ is Hilbert-Schmidt. In that case, the random variable $\xi$ from \autoref{def:stoch_int_scalar} is unique; the equality
	\begin{equation*}
		\xi = \sum_k\int_T Ae_k\,\d{Z}(e_k)
	\end{equation*}
is satisfied for any orthonormal basis $\{e_k\}_k$ of the Hilbert space $U$; and, moreover, there is the equality 
	\begin{equation*}
		\E \xi^2 = \|A\|_{\mathscr{L}_2(U;\mathscr{D}^H(T))}.
	\end{equation*}
\end{proposition}

\begin{proof}
Assume that $\mathrm{dim}\, U =\infty$, the case $\mathrm{dim}\,U<\infty$ is clear. Assume also that the operator $A\in \mathscr{L}(U;\mathscr{D}^H(T))$ is integrable with respect to $Z$. Let $N\in\mathbb{N}$ and let $\{e_k\}_k$ be an orthonormal basis of the Hilbert space $U$. Then by a straightforward computation the following equality is obtained:
	\begin{equation*}
		\E \xi^2 = \E\left(\xi-\sum_{k=1}^N\int_T Ae_k\d{Z}(e_k)\right)^2 + \sum_{k=1}^N \|Ae_k\|_{\mathscr{D}^H(T)}^2.
	\end{equation*}
By letting $N\rightarrow\infty$, the operator $A$ is shown to be Hilbert-Schmidt. Conversely, if $A\in\mathscr{L}_2(U;\mathscr{D}^H(T))$, let $\{e_k\}_k$ be an orthonormal basis of $U$ and define 
	\begin{equation*}
		\xi_N:= \sum_{k=1}^N \int_T Ae_k\d{Z}(e_k)
	\end{equation*}
for $N\in\mathbb{N}$. Then it holds for $N,M\in\mathbb{N}$ that
	\begin{equation*}
		\E|\xi_N-\xi_M|^2 = \sum_{k=M+1}^N \|Ae_k\|_{\mathscr{D}^H(T)}^2
	\end{equation*}
by using the It\^o-type isometry \eqref{eq:Ito_isometry} and the fact that if two $H$-fractional processes are uncorrelated, the same holds for their Wiener integrals. Since $A$ is Hilbert-Schmidt, the last equality shows that the sequence $\{\xi_N\}_{N\in\mathbb{N}}$ is Cauchy in $L^2(\Omega)$ and therefore, it has a limit, denoted by $\xi$, there. Now, if $g\in\mathscr{D}^H(T)$ and $u\in U$ are arbitrary, it follows by using \autoref{lem:Ito_isometry_cylindrical} that the estimate
	\begin{align*}
		\left|\E I_T(g\otimes u)\xi- \langle g,Au\rangle_{\mathscr{D}^H(T)}\right| & \\
		& \hspace{-4cm} \leq \left|\E I_T(g\otimes u)\left(\xi-\xi_N\right)\right| + \left|\E I_T(g\otimes u) \xi_N - \langle g,Au\rangle_{\mathscr{D}^H(T)}\right|\\
		& \hspace{-4cm} \leq \left\|I_T(g\otimes u)\right\|_{L^2(\Omega)} \|\xi-\xi_N\|_{L^2(\Omega)} + \left|\sum_{k=1}^N \langle A^*g,e_k\rangle_U\langle u,e_k\rangle_U - \langle g, Au\rangle_{\mathscr{D}^H(T)}\right|
	\end{align*}
is satisfied for $N\in\mathbb{N}$ and by letting $N\rightarrow\infty$, it is shown that the operator $A$ is integrable with respect to the $U$-cylindrical $H$-fractional process $Z$.
\end{proof}

\begin{notation}
The symbol $\int_TA\d{Z}$ is used for the random variable $\xi$ from \autoref{prop:char_stoch_int_scalar}.
\end{notation}

\subsection{The vector case}
\label{sec:Wiener_integration_vector_case}
In what follows, we treat Wiener integration in the case when the target space is possibly infinite-dimensional. Fix a Banach space $X$ for the remainder of this section.

\subsubsection{Weak integrability}
\label{sec:weak_integrability}

Initially, the notion of weak integrability is defined. 

\begin{definition}
A bilinear mapping $G: X^*\times U\rightarrow\mathscr{D}^H(T)$ is called \textit{weakly (Wiener) integrable} with respect to the process $Z$ if there exists a constant $C>0$ such that the inequality
	\begin{equation*}
		\|G(\varphi,\cdot)\|_{\mathscr{L}_2(U;\mathscr{D}^H(T))} \leq C\|\varphi \|_{X^*}
	\end{equation*} 
is satisfied for every $\varphi\in X^*$. 
\end{definition}

There is the following characterization of weak integrability in $X$.

\begin{proposition}
\label{lem:char_weak_stoch_integrability}
A bilinear mapping $G:X^*\times U\rightarrow \mathscr{D}^H(T)$ is weakly integrable with respect to the process $Z$ if and only if there exists a bounded linear operator $\overline{G}: X^*\rightarrow \mathscr{D}^H(T;U)$ for which the equality
	\begin{equation}
	\label{eq:Gbar}
		\langle \overline{G}\varphi,g\otimes u\rangle_{\mathscr{D}^H(T;U)} = \langle G(\varphi,u),g\rangle_{\mathscr{D}^H(T)}
	\end{equation}
holds for every $\varphi\in X^*$, $g\in\mathscr{D}^H(T)$, and every $u\in U$. In this case, the equality 
	\begin{equation*}
		\overline{G} = \sum_k G(\cdot , e_k)\otimes e_k
	\end{equation*}
is satisfied for any orthonormal basis $\{e_k\}_k$ of the Hilbert space $U$; and, moreover, the equality
	\begin{equation*}
		\|\overline{G}\varphi\|_{\mathscr{D}^H(T;U)} = \|G(\varphi,\cdot)\|_{\mathscr{L}_2(U;\mathscr{D}^H(T))}
	\end{equation*} 
is satisfied for every $\varphi\in X^*$. 
\end{proposition}

\begin{proof}
Assume that $\mathrm{dim}\,U = \infty$, the case $\dim U<\infty$ is clear. Let $G:X^*\times U\rightarrow \mathscr{D}^H(T)$ be a bilinear mapping. Assume first that $G$ is weakly integrable with respect to the process $Z$. Let $\{e_k\}_{k}$ be an orthonormal basis of $U$ and define a sequence of bounded linear operators $\{\overline{G}_N\}_{N\in\mathbb{N}}$ where for $N\in\mathbb{N}$, $\overline{G}_N$ is the operator $\overline{G}_N: X^*\rightarrow \mathscr{D}^H(T;U)$ defined by
	\begin{equation}
	\label{eq:G_N}
		\overline{G}_N\varphi := \sum_{k=1}^N G(\varphi,e_k)\otimes e_k, \quad \varphi\in X^*. 
	\end{equation}
Then it holds for every $\varphi\in X^*$ and $N,M\in\mathbb{N}$ that
	\begin{equation*}
		\|\overline{G}_N\varphi-\overline{G}_M\varphi\|_{\mathscr{D}^H(T;U)}^2 = \sum_{k={M+1}}^N \|G(\varphi,e_k)\|_{\mathscr{D}^H(T)}^2
	\end{equation*}
and because $G$ is weakly integrable, $G(\varphi,\cdot): U\rightarrow \mathscr{D}^H(T)$ is Hilbert-Schmidt. Thus it follows from the above equality that $\{G_N\varphi\}_{N\in\mathbb{N}}$ is Cauchy in $\mathscr{D}^H(T;U)$ by letting $N,M\rightarrow\infty$. Consequently, for every $\varphi\in X^*$, the sequence $\{\overline{G}_N\varphi\}_{N\in\mathbb{N}}$ converges in $\mathscr{D}^H(T;U)$ and it follows that the operator $\overline{G}: X^*\rightarrow \mathscr{D}^H(T;U)$ defined by $\overline{G}\varphi := \lim_{N\rightarrow \infty}\overline{G}_N\varphi$ is linear and, by the uniform boundedness principle, bounded. Now, if $\varphi\in X^*$, $g\in\mathscr{D}^H(T)$, and $u\in U$ are arbitrary, it follows that the estimate 
	\begin{align*}
		\left|\langle\overline{G}\varphi, g\otimes u\rangle_{\mathscr{D}^H(T;U)} - \langle G(\varphi,u),g\rangle_{\mathscr{D}^H(T)}\right| & \\
			& \hspace{-6cm} \leq \left|\langle \overline{G}\varphi, g\otimes u\rangle_{\mathscr{D}^H(T;U)} -  \langle \overline{G}_N\varphi,g\otimes u\rangle_{\mathscr{D}^H(T;U)}\right| + \left|\langle \overline{G}_N\varphi,g\otimes u\rangle_{\mathscr{D}^H(T;U)} - \langle G(\varphi,u),g\rangle_{\mathscr{D}^H(T)}\right|\\
			& \hspace{-6cm} \leq \left|\langle \overline{G}\varphi, g\otimes u\rangle_{\mathscr{D}^H(T;U)} -  \langle \overline{G}_N\varphi,g\otimes u\rangle_{\mathscr{D}^H(T;U)}\right| + \left|\sum_{k=1}^N \langle e_k, G(\varphi,\cdot)^*g\rangle_{U}\langle e_k,u\rangle_U - \langle G(\varphi,u),g\rangle_{\mathscr{D}^H(T)}\right|
	\end{align*}
is satisfied for $N\in\mathbb{N}$. By letting $N\rightarrow\infty$, it is shown that the equality \eqref{eq:Gbar} is satisfied. 
Conversely, let $\overline{G}: X^*\rightarrow\mathscr{D}^H(T;U)$ be a bounded linear operator such that equality \eqref{eq:Gbar} is satisfied for every $\varphi\in X^*$, $g\in\mathscr{D}^H(T)$, and $u\in U$. Let $\varphi\in X^*$ and let $\{e_k\}_{k}$ be an orthonormal basis of $U$ and $\{g_j\}_j$ be an orthonormal basis of $\mathscr{D}^H(T)$. Then it holds that 
	\begin{align*}
		\|G(\varphi,\cdot)\|_{\mathscr{L}_2(U;\mathscr{D}^H(T))}^2 & = \sum_k \|G(\varphi,e_k)\|_{\mathscr{D}^H(T)}^2\\
		& = \sum_{k,j} \langle G(\varphi,e_k),g_j\rangle_{\mathscr{D}^H(T)}^2\\
		& = \sum_{k,j} \langle \overline{G}\varphi, g_j\otimes e_k\rangle_{\mathscr{D}^H(T;U)}^2\\
		& = \|\overline{G}\varphi\|_{\mathscr{D}^H(T;U)}^2
	\end{align*}
because $\{g_j\otimes e_k\}_{j,k}$ is an orthonormal basis of the space $\mathscr{D}^H(T;U)$. Therefore, $G$ is weakly integrable.
\end{proof}

The following lemma is proved in a similar manner as the second part of the proof of \autoref{lem:char_weak_stoch_integrability}.

\begin{lemma}
Let $B:X^*\rightarrow\mathscr{D}^H(T;U)$ be a bounded linear operator. Define
	\begin{equation*}
		G(\varphi,u) := \sum_j g_j\langle B\varphi, g_j\otimes u\rangle_{\mathscr{D}^H(T;U)}, \quad \varphi\in X^*, \quad u\in U,
	\end{equation*}
for an orthonormal basis $\{g_j\}_j$ in $\mathscr{D}^H(T)$. Then $G$ is weakly integrable with respect to the process $Z$ and it holds for the operator $\overline{G}$ from \autoref{lem:char_weak_stoch_integrability} that $\overline{G}=B$. 
\end{lemma}

The following lemma connects the integrals with respect to scalar and cylindrical fractional processes.

\begin{lemma}
\label{lem:link_between_scalar_and_vector_integral}
Let $G: X^*\times U\rightarrow \mathscr{D}^H(T)$ be weakly  integrable with respect to the process $Z$ and let $\overline{G}$ be the corresponding operator from \autoref{lem:char_weak_stoch_integrability}. Then the equality
	\begin{equation*}
		\int_TG(\varphi,\cdot)\d{Z} = I_{T}(\overline{G}\varphi)
	\end{equation*}
is satisfied almost surely for every $\varphi\in X^*$. Here, $\overline{G}$ is the operator that corresponds to the mapping $G$ from \autoref{lem:char_weak_stoch_integrability}.
\end{lemma}

\begin{proof}
It holds for $N\in\mathbb{N}$, $\varphi\in X^*$, and an orthonormal basis $\{e_k\}_k$ of the Hilbert space $U$ that 
	\begin{align*}
		\left\|\int_TG(\varphi,\cdot)\d{Z} - I_{T}(\overline{G}\varphi)\right\|_{L^2(\Omega)} & \leq  \left\|\int_TG(\varphi,\cdot)\d{Z} - \sum_{k=1}^N \int_TG(\varphi,e_k)\d{Z}(e_k)\right\|_{L^2(\Omega)}  \\
					&  \hspace{2cm}+ \left\|\sum_{k=1}^N\int_TG(\varphi, e_k)\d{Z}(e_k) - I(\overline{G}\varphi)\right\|_{L^2(\Omega)} \\
					& =  \left\|\int_TG(\varphi,\cdot)\d{Z} - \sum_{k=1}^N \int_TG(\varphi,e_k)\d{Z}(e_k)\right\|_{L^2(\Omega)}  \\
					& \hspace{2cm} + \left\|I_{T}\left(\sum_{k=1}^N G(\varphi,e_k)\otimes e_k\right) - I_T(\overline{G}\varphi)\right\|_{L^2(\Omega)}.\numberthis\label{eq:integral_1}
	\end{align*}
Since $G$ is weakly integrable with respect to the process $Z$, it follows by \autoref{prop:char_stoch_int_scalar} that $G(\varphi,\cdot)$ is integrable in the sense of \autoref{def:stoch_int_scalar} and that the first term in \eqref{eq:integral_1} tends to zero as $N\rightarrow\infty$. For the second term, we have by appealing to the It\^o-type isometry \eqref{eq:Ito_isometry} that 
	\begin{equation*}
		\left\|I_{T}\left(\sum_{k=1}^N G(\varphi,e_k)\otimes e_k\right) - I_{T}(\overline{G}\varphi)\right\|_{L^2(\Omega)} = \|I_{T}[(\overline{G}_N-\overline{G})\varphi]\|_{L^2(\Omega)} = \|(\overline{G}_N-\overline{G})\varphi\|_{\mathscr{D}^H(T)},
	\end{equation*}
where $\overline{G}_N$ is the operator defined by formula \eqref{eq:G_N}, holds. By the proof of \autoref{lem:char_weak_stoch_integrability}, it follows that the right-hand side of the above equality tends to zero as $N\rightarrow\infty$ which concludes the proof. 
\end{proof}

\subsubsection{Strong integrability}
\label{sec:strong_integrability}

In this section, a stronger notion of Wiener integrability is treated. Initially, this notion is defined.

\begin{definition}
\label{def:stoch_int_vector}
Let $G: X^*\times U\rightarrow \mathscr{D}^H(T)$ be weakly integrable with respect to the process $Z$ and let $\overline{G}$ be the corresponding operator from \autoref{lem:char_weak_stoch_integrability}. The mapping $G$ is said to be \textit{(Wiener) integrable with respect to the process $Z$} if there exists a random variable $ \xi\in L^0(\Omega;X)$ such that the equality
	\begin{equation*}
		\langle \varphi,\xi\rangle = I_{T}(\overline{G}\varphi)
	\end{equation*}
is satisfied almost surely for every $\varphi\in X^*$. In this case, the random variable $\xi$ will be called the \textit{(Wiener) integral} of $G$ with respect to the process $Z$.
\end{definition}

If $G$ is weakly integrable with respect to the process $Z$ and $\overline{G}$ is the corresponding operator from \autoref{prop:char_stoch_int_scalar}, its adjoint $\overline{G}^*$ maps the space $\mathscr{D}^H(T;U)$ into $X^{**}$. However, if $G$ is integrable with respect to the process $Z$ and if its integral is a square-integrable random variable, the adjoint $\overline{G}^*$ takes values in the space $X$ as shown in the following result.

\begin{proposition}
Let $G: X^*\times U\rightarrow \mathscr{D}^H(T)$ be integrable with respect to the process $Z$, let $\overline{G}$ be the corresponding operator from \autoref{lem:char_weak_stoch_integrability}, and let $\xi$ be its integral. If $\xi\in L^2(\Omega;X)$, then the operator $\overline{G}^*$ maps the space $\mathscr{D}^H(T;U)$ into the space $X$ and it is given by 
	\begin{equation*}
		\overline{G}^*(S) = \E I_{T}(S)\xi, \quad S\in\mathscr{D}^H(T;U).
	\end{equation*}
\end{proposition}

\begin{proof}
Let $g\in\mathscr{D}^H(T)$, $u\in U$. Then we have that 
	\begin{equation*}
		\langle \overline{G}\varphi,g\otimes u\rangle_{\mathscr{D}^H(T;U)} = \E I_{T}(\overline{G}\varphi)I_{T}(g\otimes u) = \E\langle \varphi,\xi\rangle I_{T}(g\otimes u) = \langle \varphi, \E\xi I_{T}(g\otimes u)\rangle.
	\end{equation*}
where \autoref{lem:Ito_isometry_cylindrical} was used to obtain the first equality while integrability of $G$ was used in the second equality. Thus it follows that $\overline{G}^*(S) = \E I_{T}(S)\xi$ for every $S\in\mathscr{D}^H(T;U)$.
\end{proof}

In what follows, two properties of Banach spaces that are useful for the subsequent analysis are given. The first notion seems not to have been explicitly considered and studied in the literature so far.

\begin{definition}
Let $n\in\mathbb{N}$. A Banach space $X$ is said to be \textit{$n$-good} if the equivalence
		\begin{equation*}
			\left\|\sum_{k=1}^m\varepsilon^{(1)}_kx_k\right\|_{L^2(\Omega;X)} \eqsim \left\|\sum_{k=1}^m\varepsilon^{(2)}_kx_k\right\|_{L^2(\Omega;X)}
		\end{equation*}
holds for every two orthonormal sets $\{\varepsilon^{(1)}_k\}_{k=1}^m\subseteq\mathscr H^{\oplus n}$ and $\{\varepsilon^{(2)}_k\}_{k=1}^m\subseteq\mathscr H^{\oplus n}$ and every $\{x_k\}_{k=1}^m\subseteq X$. 
\end{definition}

We also recall the notion of the approximation property of Banach spaces; see, e.g. \cite[Chapter 1]{LinTza77}.

\begin{definition} A Banach space $X$ is said to have \textit{the approximation property} if, for every $\varepsilon>0$ and every compact $C\subseteq X$, there exists a finite rank operator $T\in\mathscr L(X)$ such that
	\begin{equation*}
		\sup_{x\in C}\,\|x-Tx\|_X\le\varepsilon.
	\end{equation*}
If there exists $\lambda\geq 0$ for which every such $T$ satisfies $\|T\|_{\mathscr L(X)}\leq\lambda$, then $X$ is said to have \textit{the $\lambda$-bounded approximation property}.
\end{definition}

It is shown now that many commonly used function spaces are both $n$-good (for every $n\in\mathbb{N}$) and have the approximation property. Initially, it is useful to establish the following notation.

\begin{notation} For $p_1,p_2\in[1,\infty)$ and for two $\sigma$-finite separable measure spaces $(D_1,\mathcal{D}_1,\mu_1)$ and $(D_2,\mathcal{D}_2,\mu_2)$, we denote by $E^{p_1,p_2}(D_1\times D_2)$ the space of jointly measurable functions $f: D_1\times D_2\rightarrow \R$ for which the following finiteness condition is satisfied:
	\begin{equation*}
		\|f\|_{E^{p_1,p_2}(D_1\times D_2)}:= \left(\int_{D_1}\left(\int_{D_2}|f(x,y)|^{p_2}\,\mu_2(\d{y})\right)^\frac{p_1}{p_2}\mu_1(\d{x})\right)^\frac{1}{p_1}<\infty.
	\end{equation*}
Under the assumptions above, the space $E^{p_1,p_2}(D_1\times D_2)$ is a separable Banach space (it is, in fact, a so-called Lebesgue space with mixed norm, cf., e.g., \cite{BenPan61}, or, more explicitly, the space $L^{p_1}(D_1;L^{p_2}(D_2)$). 
\end{notation}

\begin{proposition} 
\label{prop:Y_is_good_and_approximable}
Let $N\in\mathbb{N}$. Let $\{p_{i,1}\}_{i=1}^N$ and $\{p_{i,2}\}_{i=1}^N$ be two subsets of the interval $[1,\infty)$ and let $\{(D_{i,1},\mathcal{D}_{i,1},\mu_{i,1})\}_{i=1}^N$ and $\{(D_{i,2},\mathcal{D}_{i,2},\mu_{i,2})\}_{i=1}^N$ be two sets of $\sigma$-finite separable measure spaces. Set
	\begin{equation*}
		Y:=\prod_{i=1}^NE^{p_{i,1},p_{i,2}}(D_{i,1}\times D_{i,2}).
	\end{equation*}
If $X$ is a normed linear space that is isomorphic with a subspace of $Y$ then $X$ is $n$-good for every $n\in\mathbb{N}$. If $X$ is a retraction\footnote{A normed linear space $X$ said to be a \textit{retraction of $Y$} if there exist $R\in\mathscr L(Y,X)$ and $S\in\mathscr L(X,Y)$ such that $RS=I_X$.} of $Y$, then there exists $\lambda \geq 0$ such that $X$ has the $\lambda$-bounded approximation property.
\end{proposition}

\begin{proof}
Let $Q:X\rightarrow \tilde{Y}$ be the linear injection of $X$ onto the subspace $\tilde{Y}$ of $Y$. Let $n\in\mathbb{N}$ be arbitrary, let $\{\varepsilon_k\}_{k=1}^m$ be an orthonormal set in $\mathscr{H}^{\oplus n}$, and let $\{x_k\}_{j=1}^m$ be a subset of $X$. Define $f^i_k:=[Qx_k]_i$, $i=1,2,\ldots, N$; $k=1,2,\ldots, m$. Then there is the equivalence
	\begin{equation}
	\label{eq:n_good}
		\left\|\sum_{k=1}^m\varepsilon_kx_k\right\|_{L^1(\Omega;X)}\eqsim\sum_{i=1}^N\left\{\int_{D_{i,1}}\left[\int_{D_{i,2}}\left(\sum_{k=1}^m|f^i_k|^2\right)^\frac{p_{i,2}}{2}\d\mu_{i,2}\right]^\frac{p_{i,1}}{p_{i,2}}\d\mu_{i,1}\right\}^\frac{1}{p_{i,1}}
	\end{equation}
and, consequently, the space $X$ is $n$-good. Indeed, we have that
	\begin{equation*}
		\left\|\sum_{k=1}^m\varepsilon_kx_k\right\|_{L^1(\Omega;X)} \eqsim \sum_{i=1}^N\E\left\|\sum_{k=1}^m\varepsilon_kf_k^i\right\|_{E^{p_{i,1},p_{i,2}}(D_{i,1}\times D_{i,2})} \eqsim \sum_{i=1}^N \left(\E\left\|\sum_{k=1}^m\varepsilon_kf_k^i\right\|_{E^{p_{i,1},p_{i,2}}(D_{i,1}\times D_{i,2})}^{p_{i,1}}\right)^\frac{1}{p_{i,1}}
	\end{equation*}
holds by using \autoref{prop:hypercontractivity} for each term in the sum with $\mathcal{B}=E_i$, $p=1$, and $q=p_{i,1}$. Now, for $i\in\{1,2,\ldots,N\}$, it follows that
	\begin{align*}
		\E\left\|\sum_{k=1}^m\varepsilon_kf_k^i\right\|_{E_i}^{p_{i,1}} & = \E \int_{D_{i,1}}\left(\int_{D_{i,2}}\left|\sum_{k=1}^m\varepsilon_k f_k^i(x,y)\right|^{p_{i,2}}\mu_{i,2}(\d{y})\right)^\frac{p_{i,1}}{p_{i,2}}\mu_{i,1}(\d{x})\\
		& = \int_{D_{i,1}}\E\left\|\sum_{k=1}^m\varepsilon_kf_k^i(x,\cdot)\right\|_{L^{p_{i,2}}(D_{i,2})}^{p_{i,1}}\mu_{i,1}(\d{x})\\
		& \eqsim \int_{D_{i,1}}\left(\E\left\|\sum_{k=1}^m \varepsilon_kf_k^i(x,\cdot)\right\|_{L^{p_{i,2}}(D_{i,2})}^{p_{i,2}}\right)^{\frac{p_{i,1}}{p_{i,2}}}\mu_{i,1}(\d{x})
	\intertext{where \autoref{prop:hypercontractivity} is used with $\mathcal{B}=L^{p_{i,2}}(D_{i,2})$, $p=p_{i,1}$, and $q=p_{i,2}$. The chain of equivalences continues as}
		& = \int_{D_{i,1}}\left(\E\int_{D_{i,2}}\left|\sum_{k=1}^m\varepsilon_kf_k^i(x,y)\right|^{p_{i,2}}\mu_{i,2}(\d{y})\right)^{\frac{p_{i,1}}{p_{i,2}}}\mu_{i,1}(\d{x})\\
		& = \int_{D_{i,1}}\left(\int_{D_{i,2}}\E\left|\sum_{k=1}^m\varepsilon_kf_k^i(x,y)\right|^{p_{i,2}}\mu_{i,2}(\d{y})\right)^{\frac{p_{i,1}}{p_{i,2}}}\mu_{i,1}(\d{x})\\
		& \eqsim \int_{D_{i,1}}\left[\int_{D_{i,2}} \left( \E\left|\sum_{k=1}^m \varepsilon_k f_{k}^i(x,y)\right|^2\right)^{\frac{p_{i,2}}{2}}\mu_{i,2}(\d{y})\right]^{\frac{p_{i,1}}{p_{i,2}}}\mu_{i,1}(\d{x})
	\end{align*}
where again \autoref{prop:hypercontractivity} is used with $\mathcal{B}=\R$, $p=p_{i,2}$, and $q=2$. Now, since $\{\varepsilon_k\}_{k=1}^m$ is an orthonormal set in $\mathscr{H}^{\oplus n}$, it follows that the equality
	\begin{equation*}
		\E\left|\sum_{k=1}^m \varepsilon_k f_{k}^i(x,y)\right|^2 = \sum_{k=1}^m |f_k^i(x,y)|^2
	\end{equation*}
is satisfied for $\mu_{i,1}\otimes\mu_{i,2}$-almost every $(x,y)\in D_{i,1}\times D_{i,2}$ which proves formula \eqref{eq:n_good}.

For a space $E^{p_1,p_2}(D_1\times D_2)$ where $(D_1,\mathcal{D}_1,\mu_1)$ and $(D_2,\mathcal{D}_2,\mu_2)$ are $\sigma$-finite separable measure spaces and $p_1, p_2\in [1,\infty)$; and for some $\sigma$-finite partitions $\{A_i\}$ and $\{B_j\}$ of $D_1$ and $D_2$, respectively, define $P_{A,B}: E^{p_1,p_2}(D_1\times D_2)\rightarrow E^{p_1,p_2}(D_1\times D_2)$ by
	\begin{equation*}
P_{A,B}f:=\sum_{i,j}\frac{1}{\mu_1(A_i)\mu_2(B_j)}\left(\int_{A_i\times B_j}f\,\d\mu_1\otimes\mu_2\right)\bm{1}_{A_i\times B_j}
	\end{equation*}
with the convention $\sfrac{0}{0}:=0$. Then 
	\begin{equation*}
		\|P_{A,B}f\|_{E^{p_1,p_2}(D_1\times D_2)}\leq\|f\|_{E^{p_1,p_2}(D_1\times D_2)}, \quad f\in E^{p_1,p_2}(D_1\times D_2).
	\end{equation*}
Now, if $\{A^k_i\}$ and $\{B^k_j\}$ are sequences of partitions of $D_1$ and $D_2$, respectively, such that $\{A^{k+1}_i\}$ refines $\{A^k_i\}$ and $\{B^{k+1}_i\}$ refines $\{B^k_i\}$; and if $\sigma(\{A^k_i\})$ and $\sigma(\{B^k_i\})$ generate the $\sigma$-algebras $\mathcal{D}_1$ and $\mathcal{D}_2$, respectively, then
	\begin{equation*}
		\lim_{k\rightarrow\infty}\|P_{A^k,B^k}f-f\|_{E^{p_1,p_2}(D_1\times D_2)} = 0, \quad f\in E^{p_1,p_2}(D_1\times D_2). 
	\end{equation*}
Hence, $f\mapsto P_{ K_m^n,\tilde{K}_m^n}(f\mathbf 1_{K_m\times\tilde K_m})$ are the sought projections that make $E^{p_1,p_2}(D_1\times D_2)$ a Banach space with the $1$-bounded approximation property. Clearly, the space $\prod_{i=1}^NE^{p_{i,1},p_{i,2}}(D_{i,1}\times D_{i,2})$ also has the $1$-bounded approximation property as well as the $\|R\|_{\mathscr{L}(Y;X)}\|S\|_{\mathscr{L}(X;Y)}$-bounded approximation property.
\end{proof}

\begin{corollary}
\label{cor:function_spaces}
Let $D$ be $\mathbb{R}^d$, $\mathbb{R}^d_+$, or a bounded $\mathscr{C}^\infty$-domain in $\mathbb{R}^d$; and let $p,q\in [1,\infty)$, $s\in\R$, and $r\in (0,\infty)$. Then the spaces
	\begin{equation*}
		L^p(D), \quad W^{r,p}(D), \quad B_{p,q}^s(D), \quad F_{p,q}^s(D)
	\end{equation*}
are $n$-good for every $n\in\mathbb{N}$ and have the approximation property. Here, the spaces $L^p(D)$ are the standard Lebesgue spaces, $W^{r,p}(D)$ are the (fractional) Sobolev spaces, $B_{p,q}^s(D)$ are the Besov spaces, and $F_{p,q}(D)$ are the Lizorkin-Triebel spaces; see, e.g.,  \cite{Trie83}.
\end{corollary}

\begin{proof}
The first assertion of the corollary follows immediately for the spaces $L^p(D)$, $W^{r,p}(D)$, $B_{p,q}^s(\R^d)$, and $F_{p,q}^s(\R^d)$ by \autoref{prop:Y_is_good_and_approximable} because these spaces are isomorphic with a subspace of a product of some mixed Lebesgue spaces. The first assertion also holds for $B^s_{p,q}(D)$ and $F_{p,q}^s(D)$ in the case when $D$ is either $\mathbb{R}^d_+$ or a bounded $\mathscr{C}^\infty$-domain in $\R^d$ because these spaces are retractions of $B_{p,q}^s(\R^d)$ and $F_{p,q}^s(\R^d)$, respectively; see \cite[Theorem 2.9.4 and Theorem 3.3.4]{Trie83}.

The focus is on the second assertion of the corollary now. The spaces $B_{p,q}^s(\R^d)$ and $F_{p,q}^s(\R^d)$ have the approximation property since they are retractions (the space $F_{p,q}^s(\R^d)$ for $p,q\in (1,\infty)$) of some mixed Lebesgue spaces; see, e.g.,  \cite[2.3.2/12 and 2.3.2/13]{Trie78}. Therefore, it remains to prove that $W^{r,1}(\R^d)$ and $F_{p,q}^s(\R^d)$ with $p$ or $q$ being equal to one, have the approximation property. 

As far as the spaces $W^{r,1}(\R^d)$ are concerned, note that whenever $j\in\{0,1\}$, the sequence of operators $\{\pi_{k}^j\}_{k\in\mathbb{N}}$ defined by $\pi_k^0(f):= \psi(\sfrac{\cdot}{k})f$, where $\psi\in\mathscr{C}_c^\infty(\R^d)$ is such that $\psi(0)=1$, and $\pi_k^1(f) := f*\omega_k$, where $\omega_k$ are the standard smooth compactly supported mollifiers, has the following three properties:
	\begin{itemize}
	\itemsep0em
		\item There is a constant $C_{r,d}>0$ such that $\|\pi_k^j\|_{\mathscr{L}(W^{r,1}(\R^d))}\leq C_{r,d}$ holds for every $k\in\mathbb{N}$.
		\item For every $a>0$ and $k\in\mathbb{N}$, there is a constant $C_{a,k,d}>0$ such that $\|\pi_k^{1}\|_{\mathscr{L}(L^1(\R^d);W^{a,2}(\R^d))} \leq C_{a,d,k}$.
		\item For every $f\in W^{r,1}(\R^d)$, there is the convergence $\lim_{k\rightarrow\infty}\|f-\pi_k^j(f)\|_{W^{r,1}(\R^d)}=0$. 
	\end{itemize}
Hence, it follows that the space $W^{r,1}(\R^d)$ has the approximation property.

Now, the focus is on the spaces $F_{p,q}^s(\R^d)$ with $p$ or $q$ being equal to one. Consider the standard decomposition of unity $\{\phi_j\}_{ j\in\mathbb{N}_0}$ given, for example, in \cite[2.3.1]{Trie83}. Then the sequence of operators $\{\pi_k\}_{k\in\mathbb{N}}$ defined by 
	\begin{equation*}
		\pi_k(T) := \sum_{j=0}^k\mathscr{F}^{-1}\{\phi_j\mathscr{F}T\}, \quad T\in\mathscr{S}'(\R^d),
	\end{equation*}
where $\mathscr{F}$ denotes the Fourier transform, has the following three properties:
	\begin{itemize}
	\itemsep0em
		\item There exists a constant $C_d>0$ such that $\|\pi_k\|_{\mathscr{L}(F_{p,q}^s(\R^d))}\leq C_d$ holds for every $k\in\mathbb{N}$. 
		\item For every $k\in\mathbb{N}$ there exists a constant $C_k>0$ such that $\|\pi_k\|_{\mathscr{L}(F_{p,q}^s(\R^d);B_{p,1}^s(\R^d))}\leq C_k$. 
		\item For every $T\in F_{p,q}^s(\R^d)$, there is the convergence $\lim_{k\rightarrow\infty}\|T-\pi_k(T)\|_{F_{p,q}^s(\R^d)}=0$. 
	\end{itemize}
Hence, it follows that the space $F_{p,q}^s(\R^d)$ has the approximation property. Finally, it follows from the above that the spaces $W^{r,1}(D)$ and $F_{p,q}^s(D)$ (with $p$ or $q$ being equal to one) have the approximation property in the case when $D$ is $\R^d_+$ or a bounded $\mathscr{C}^\infty$-domain since these spaces are retractions of $W^{1,r}(\R^d)$ and $F_{p,q}^s(\R^d)$, respectively.
\end{proof}

\begin{remark}
Note that \autoref{cor:function_spaces} covers even some non-UMD Banach spaces; e.g., the space $L^1(\R^d)$ is a typical example since it is not reflexive.
\end{remark}

In what follows, we return to integrability with respect to the process $Z$. In the first result, it is shown that if $Z$ lives in a finite Wiener chaos and if the Banach space $X$ is $n$-good, then there is a sufficient condition for integrability with respect to $Z$. Let us first prove a lemma which will be crucial in the proof of the forthcoming \autoref{prop:stoch_int_1}. The lemma is a slight modification of a claim that appeared in the proof of \cite[Proposition 3.1]{CouMasOnd18} and it shows that Wiener integrals with respect to a fractional process from a finite Wiener chaos also belong to this finite Wiener chaos. As an immediate consequence, all moments of these integrals are equivalent by \autoref{prop:hypercontractivity}.

\begin{lemma}
\label{lem:Wiener_integral_belongs_to_H}
If the process $z$ lives in a finite Wiener chaos $\mathscr{H}^{\oplus n}$, then the Wiener integral $i_{ T}(f)$ belongs to $\mathscr{H}^{\oplus n}$ for every $f\in\mathscr{D}^{ H}(T)$.
\end{lemma}

\begin{proof}
Let $f\in\mathscr{E}(T)$ take the form \eqref{eq:step_function}. Since the process $z$ lives in a finite Wiener chaos $\mathscr{H}^{\oplus n}$, we have that the elementary integral of $f$ that takes the form \eqref{eq:step_integral} also belongs to $\mathscr{H}^{\oplus n}$ because this is a linear space. If $f\in\mathscr{D}^{H}(T)$ and $\{f^k\}_{k\in\mathbb{N}}$ is a sequence of step functions that converges to $f$ in $\mathscr{D}^{ H}(T)$, then $i_{T}(f^k)$ converges to $i_{T}(f)$ in $L^2(\Omega)$. However, each integral $i_{T}(f^k)$ belongs to $\mathscr{H}^{\oplus n}$ and since this is a closed subspace of $L^2(\Omega)$, it follows that $i_{T}(f)\in\mathscr{H}^{\oplus n}$. 
\end{proof}

A main result of the present section follows.

\begin{theorem}
\label{prop:stoch_int_1}
Assume that the process $Z$ lives in the finite Wiener chaos $\mathscr{H}^{\oplus n}$ for $n\in\mathbb{N}$. Assume moreover that the Banach space $X$ is $n$-good. Let $G:X^*\times U\rightarrow \mathscr{D}^H(T)$ be weakly integrable with respect to the process $Z$ and let $\overline{G}$ be the corresponding operator from \autoref{lem:char_weak_stoch_integrability}. If $\overline{G}^*\in\gamma(\mathscr{D}^H(T;U);X)$, then $G$ is integrable with respect to the process $Z$ and it holds for every $r>0$ that
	\begin{equation}
	\label{eq:equivalence_of_norm_of_integral}
		\|\xi\|_{L^r(\Omega;X)} \eqsim \|\overline{G}^*\|_{\gamma(\mathscr{D}^H(T;U);X)}
	\end{equation}
where $\xi$ is the integral of $G$ with respect to the process $Z$. 
\end{theorem}

\begin{proof}
\textit{Step 1.} Assume first that $G$ is such that the operator $\overline{G}^*$ has finite range. This means that $\overline{G}^*$ can be expressed as
	\begin{equation}
	\label{eq:finite_rank_G}
		\overline{G}^*S = \sum_{k=1}^m\langle S,e_k\rangle_{\mathscr{D}^{H}(T;U)} x_k, \quad S\in\mathscr{D}^{H}(T;U),
	\end{equation}
for some $m\in\mathbb{N}$, some orthonormal set $\{e_k\}_{k=1}^m$ in $\mathscr{D}^H(T;U)$ and some set $\{x_k\}_{k=1}^N\subseteq X$. Define
	\begin{equation}
	\label{eq:elem_integral}
		\xi:= \sum_{k=1}^m I_{T}(e_k)x_k .
	\end{equation}
Note that $\{I_{T}(e_k)\}_{k=1}^m\subseteq \mathscr{H}^{\oplus n}$ by \autoref{lem:Wiener_integral_belongs_to_H}. Consequently, since $X$ is $n$-good, there is the equivalence
	\begin{equation}
	\label{eq:integra_and_gamma_norm}
		\E\left\|\xi\right\|_{X}^2 = \E\left\|\sum_{k=1}^NI_{T}(e_k)x_k\right\|_{X}^2 \eqsim \E\left\|\sum_{k=1}^N\varepsilon_kx_k\right\|_{X}^2 = \E\left\|\sum_{k=1}^m \varepsilon_k\overline{G}^*e_k\right\|_{X}^2 = \|\overline{G}^*\|_{\gamma(\mathscr{D}^H(T;U);X)}^2
	\end{equation}
where $\{\varepsilon_k\}_{k=1}^N$ is a collection of independent standard Gaussian random variables is satisfied.

\textit{Step 2.} Now if $G$ is such that $\overline{G}^*\in\gamma(\mathscr{D}^H(T;U);X)$, then there exists a sequence $\{V_N\}_{N\in\mathbb{N}}$ of finite rank operators $V_N:\mathscr{D}^H(T;U)\rightarrow X$ of the form \eqref{eq:finite_rank_G} such that $\lim_{N\rightarrow\infty}\|V_N-\overline{G}^*\|_{\gamma(\mathscr{D}^H(T;U);X)}=0$. Let $\{\xi_N\}_{N\in\mathbb{N}}$ be the sequence of random variables $\xi_N\in L^2(\Omega;X)$ of the form \eqref{eq:elem_integral} such that $\xi_N$ corresponds to the operator $V_N$ as is \textit{Step 1} of the proof. It follows by equivalence \eqref{eq:integra_and_gamma_norm} that the equivalence
	\begin{equation*}
		\|\xi_N-\xi_M\|_{L^2(\Omega;X)} \eqsim \|V_N - V_M\|_{\gamma(\mathscr{D}^H(T;U);X)}
	\end{equation*}
is satisfied for $N,M\in\mathbb{N}$. Because $\overline{G}^*\in\gamma(\mathscr{D}^H(T;U);X)$, by letting $N,M\rightarrow\infty$, we have that the sequence $\{\xi_N\}_{N\in\mathbb{N}}$ is Cauchy in $L^2(\Omega;X)$ and thus, convergent there. Denote the limit by $\xi$. Let $\{\overline{G}_N\}_{N\in\mathbb{N}}$ be the sequence of operator $\overline{G}_N: X^*\rightarrow \mathscr{D}^H(T;U)$ defined by $\overline{G}^*_N:=V_N$. It follows immediately from the convergence of $\{V_N\}_{N\in\mathbb{N}}$ that $\lim_{N\rightarrow\infty}\|\overline{G}_N-\overline{G}\|_{\mathscr{L}(X^*;\mathscr{D}^H(T;U))}=0$ and hence, it follows that \[I_{T}(\overline{G}\varphi) = \langle \varphi,\xi\rangle\] holds $\mathbb{P}$-almost surely for every $\varphi\in X^*$. Finally, equivalence \eqref{eq:equivalence_of_norm_of_integral} follows by \autoref{prop:hypercontractivity}.
\end{proof}

As shown in the following result, if, additionally, the Banach space $X$ has the approximation property, the sufficient condition from \autoref{prop:stoch_int_1} is also necessary.

\begin{theorem}
\label{prop:char_stoch_int_vector}
Assume that the process $Z$ lives in a finite Wiener chaos $\mathscr{H}^{\oplus n}$ for $n\in\mathbb{N}$ and assume that the Banach space $X$ is $n$-good and has the approximation property. Let $G:X^*\times U\rightarrow \mathscr{D}^H(T)$ be weakly integrable with respect to the process $Z$ and let $\overline{G}$ be the corresponding operator from \autoref{lem:char_weak_stoch_integrability}. The mapping $G$ is integrable with respect to $Z$ if and only if $\overline{G}^*\in\gamma(\mathscr{D}^H(T;U);X)$. In that case, it holds for every $r>0$ that
	\begin{equation}
	\label{eq:equivalence_of_norm_of_integral_2}
		\|\xi\|_{L^r(\Omega;X)} \eqsim \|\overline{G}^*\|_{\gamma(\mathscr{D}^H(T;U);X)}
	\end{equation}
where $\xi$ is the integral of $G$ with respect to the process $Z$.
\end{theorem}

\begin{proof}
It is only shown that if $G$ is integrable with respect to the process $Z$, then $\overline{G}^*\in\gamma(\mathscr{D}^H(T;U);X)$ as the converse is proved in \autoref{prop:stoch_int_1}. Let $\xi\in L^2(\Omega;X)$ be the Wiener integral of $G$ and assume that $\{C_m\}_{m\in\mathbb{N}}$ is a sequence of compact subsets of the Banach space $X$ such that $\mathbb{P}(\xi\not\in C_m)<2^{-m}$ holds for every $m\in\mathbb{N}$. Since $X$ has the approximation property, then for every $m\in\mathbb{N}$, there exists an operator $T_m\in\mathscr{L}(X)$ of finite rank such that 
	\begin{equation*}
		\sup_{x\in C_m}\left\|x-T_mx\right\|_X \leq 2^{-m}.
	\end{equation*}
Consequently, $\xi_m:=T_m\xi\rightarrow \xi$ as $m\rightarrow\infty$ almost surely, therefore in probability, and by \autoref{rem:hypercontractivity_p=0} also in $L^2(\Omega;X)$. Now, since $T_m$ is a finite-rank operator, it can be expressed as
	\begin{equation*}
		T_m x = \sum_{k=1}^{N_m}\varphi_k^m(x)x_k^m, \quad x\in X,
	\end{equation*}
for some $N_m\in\mathbb{N}$, $\{\varphi_k^m\}_{k=1}^{N_m}\subseteq X^*$, and $\{x_k^m\}_{k=1}^{N_m}\subseteq X$. Define the sequence $\{\overline{G}_m\}_{m\in\mathbb{N}}$ of operators $\overline{G}_m: X^*\rightarrow \mathscr{D}^H(T;U)$ by  
	\begin{equation*}
		\overline{G}_m \varphi := \sum_{k=1}^{N_m} \overline{G}\varphi_j^m\varphi(x_j^m), \quad \varphi\in X^*. 
	\end{equation*}
Then, as in \textit{Step 2} of the proof of \autoref{prop:stoch_int_1}, we have that the equivalence
	\begin{equation*}
		\|\xi_N-\xi_M\|_{L^2(\Omega;X)} \eqsim \|\overline{G}^*_N-\overline{G}^*_M\|_{\gamma(\mathscr{D}^H(T;U);X)}
	\end{equation*}
is satisfied for $N,M\in\mathbb{N}$, and because the sequence $\{\xi_m\}_{m\in\mathbb{N}}$ converges to $\xi$ in $L^2(\Omega;X)$, it follows that the sequence $\{\overline{G}^*_m\}_{m\in\mathbb{N}}$ is Cauchy in $\gamma(\mathscr{D}^H(T;U);X)$ and its limit is $\overline{G}^*$.
\end{proof}

\begin{remark}
Note that, as a consequence of \autoref{rem:hypercontractivity_p=0}, equivalences \eqref{eq:equivalence_of_norm_of_integral} and \eqref{eq:equivalence_of_norm_of_integral_2} also hold for $r=0$ in the sense that the convergence of elementary integrals $\{\xi_N\}_N$ in probability implies the convergence of the corresponding operators $\{\overline{G}^*_N\}_N$ in the space $\gamma(\mathscr{D}^H(T;U);X)$.
\end{remark}

\begin{notation} The random variable $\xi$ from \autoref{def:stoch_int_vector} is denoted by $\int_TG\d{Z}$.
\end{notation}

In practical applications, it is common for $X$ to be a function space (e.g. Lebesgue, Sobolev, or a Besov space). Such function spaces are isomorphic with a subspace $\tilde{Y}$ of the general product space $Y$ that is considered in \autoref{prop:Y_is_good_and_approximable}. In this case, $\gamma$-radonifying operators with values in such spaces can be represented by pointwise kernels. Although the following proposition is stated with a general separable Hilbert space $\mathcal{V}$, the choice $\mathcal{V}=\mathscr{D}^H(J;U)$ for an interval $J\subseteq\R$ is particularly useful for our purposes.

\begin{proposition}
\label{prop:kernel} Assume that $X$ is isomorphic with a subspace $\tilde{Y}$ of the product space $Y$ that is defined in \autoref{prop:Y_is_good_and_approximable} and let $Q:X\rightarrow \tilde{Y}$ be the linear injection. Let $\mathcal{V}$ be a separable Hilbert space. A bounded linear operator $A\in\mathscr{L}(\mathcal{V};X)$ is $\gamma$-radonifying if and only if there exists a kernel $a=(a_i)_{i=1}^N$ such that the following two conditions are satisfied:
	\begin{itemize}
		\item For every $i=1,2,\ldots,N$, $a_i: D_{i,1}\times D_{i,2}\rightarrow\mathcal{V}$ is a measurable function that satisfies
			\begin{equation*}
				[QAv]_i(x,y) = \langle a_i(x,y),v\rangle_{\mathcal{V}}, \quad v\in\mathcal{V},
			\end{equation*}
			for $\mu_{i,1}\otimes \mu_{i,2}$-almost every $(x,y)\in D_{i,1}\times D_{i,2}$.
		\item The following finiteness condition is satisfied:
				\begin{equation*}
				 	\|a\|_{\tilde{Y}}:=\sum_{i=1}^N \left\{\int_{D_{i,1}}\left[\int_{D_{i,2}}\|a_i(x,y)\|_{\mathcal{V}}^{p_{i,2}}\mu_{i,2}(\d{y})\right]^{\frac{p_{i,1}}{p_{i,2}}}\mu_{i,1}(\d{x})\right\}^\frac{1}{p_{i,1}} <\infty.
				\end{equation*}
	\end{itemize}
In this case, it holds that 
	\begin{equation*}
		\|A\|_{\gamma(\mathcal{V};X)} \eqsim \|a\|_{\tilde{Y}}.
	\end{equation*}
\end{proposition}

\begin{proof}
The proposition is proved similarly as \cite[Theorem 2.3]{BrzNee03}.
\end{proof}

\section{Stochastic convolution}
\label{sec:stoch_conv}

In this section, stochastic convolution with respect to fractional processes is treated.

Assume that $V$ is a separable Hilbert space and $(\Omega,\mathcal{F},\mathbb{P})$ is a probability space with a $V$-isonormal Gaussian process $(W(v))_{v\in V}$ defined on it. Assume that the $\sigma$-field $\mathscr{F}$ is generated by the process $W$ and augmented with $\mathbb{P}$-zero sets. Let $H\in (0,1)$ and let $U$ be a separable Hilbert space. Let $(Z_t)_{t\geq 0}$ be a $U$-cylindrical $H$-fractional process that lives in the finite Wiener chaos $\mathscr{H}^{\oplus n}$ of the isonormal process $W$ for some $n\in\mathbb{N}$. Finally, let $X$ be an $n$-good Banach space that has the approximation property. 

Initially, a necessary and sufficient condition for the existence of a stochastic convolution integral is given. 

\begin{proposition}
\label{prop:stoch_convolution_existence}
Let $0\leq s<t$ and let $\tilde{G}: (0,t-s)\rightarrow \mathscr{L}(U;X)$ be a function. Define the map $G$ by 
	\begin{equation*}
		G(\varphi,u)(r):= \langle \varphi,\tilde{G}(r)u\rangle, \quad r\in (0,t-s), \quad \varphi\in X^*, \quad u\in U.
	\end{equation*}
Then the convolution integral 
	\begin{equation}
	\label{eq:xi_st}
		\xi_{s,t}:=\int_s^t G(t-\cdot)\d{Z}
	\end{equation}
is well defined if and only if $\overline{G}^*\in\gamma(\mathscr{D}^H(0,t-s;U);X)$. In that case, it holds for every $r>0$ that 
	\begin{equation}
	\label{eq:equiv_moments}
		\|\xi_{s,t}\|_{L^r(\Omega;X)} \eqsim \|\overline{G}^*\|_{\gamma(\mathscr{D}^H(0,t-s;U);X)}.
	\end{equation}
\end{proposition}
\begin{proof}
By \autoref{prop:char_stoch_int_vector}, $\xi_{s,t}$ is a well defined random variable that belongs to the space $L^2(\Omega;X)$ if and only if the finiteness condition $$\|\overline{G(t-\cdot)}^*\|_{\gamma(\mathscr{D}^H(s,t;U);X)}<\infty$$ is satisfied. On the other hand, the equality 
	\begin{equation}
	\label{eq:conv_prop_1_1}
		 \E\left\|\sum_{n,m} \delta_{n,m} \overline{G(t-\cdot)}^*(g_m\otimes e_n)\right\|_{X}^2 = \E\left\|\sum_{n,m}\delta_{n,m} \overline{G}^*\big((g_m)_{-1,t}\otimes e_n\big)\right\|_X^2
	\end{equation}
holds for every orthonormal basis $\{g_m\}_m$ of the space $\mathscr{D}^H(s,t)$, every orthonormal basis $\{e_n\}_n$ of the space $U$, and every array of independent standard Gaussian random variables $\{\delta_{n,m}\}_{n,m}$. Here, the symbol $(\,\cdot\,)_{-1,t}$ denotes the operator of affine transformation $(\,\cdot\,)_{a,b}$ defined in \autoref{sec:additional_lemmas}. Indeed, for every $\varphi\in X^*$, there is the chain of equalities
	\begin{align*}
		\left\langle\varphi, \overline{G(t-\cdot)}^*(g_m\otimes e_n)\right\rangle & = \left\langle \overline{G(t-\cdot)}\varphi, g_m\otimes e_n\right\rangle_{\mathscr{D}^H(s,t;U)}\\
		& = \left\langle \sum_k [G(t-\cdot)](\varphi, e_k)\otimes e_k, g_m\otimes e_n\right\rangle_{\mathscr{D}^H(s,t)}\\
		& = \Big\langle [G(t-\cdot)](\varphi,e_n),g_m\Big\rangle_{\mathscr{D}^H(s,t)}\\
		& = \Big\langle G(\varphi,e_n), (g_m)_{-1,t}\Big\rangle_{\mathscr{D}^H(0,t-s)}
	\end{align*}
where duality, the representation of $\overline{G(t-\cdot)}$ from the proof of \autoref{lem:char_weak_stoch_integrability}, and \autoref{lem:affine_transform} are used. By a computation similar to the computation above, the equality 
	\begin{equation*}
		\Big\langle G(\varphi,e_n), (g_m)_{-1,t}\Big\rangle_{\mathscr{D}^H(0,t-s)} = \Big\langle\varphi, \overline{G}^*\big((g_m)_{-1,t}\otimes e_n\big)\Big\rangle
	\end{equation*}
is obtained and thus, equality \eqref{eq:conv_prop_1_1} is proved.
\end{proof}

\begin{proposition}
\label{prop:if_p_greater_2}
Assume that $X$ is isomorphic with a subspace $\tilde{Y}$ of the product space $Y$ that is defined in \autoref{prop:Y_is_good_and_approximable}. Let $0\leq s<t$ and let $\tilde{G}: (0,t-s)\rightarrow {\gamma}(U;X)$ be a function.  Assume also that one of the following conditions is satisfied:
\begin{enumerate}
\itemsep0em
\item The parameter $H$ belongs to the interval $(0,\sfrac{1}{2})$; the space $\tilde{Y}$ is such that $p_{i,j}\geq 2$ holds for every $i=1,2,\ldots, N$ and $j=1,2$; and $\tilde{G}$ satisfies the following finiteness condition:
	\[\int_0^{t-s} \|\tilde{G}(u)\|_{\gamma(U;X)}^2\d{u} + \int_0^{t-s}\int_0^{t-s} \|\tilde{G}(u)-\tilde{G}(v)\|_{\gamma(U;X)}^2|u-v|^{2H-2}\d{u}\d{v}<\infty.\]
\item The parameter $H$ belongs to the interval $[\sfrac{1}{2},1)$; the space $\tilde{Y}$ is such that $p_{i,j}H\geq 1$ holds for every $i=1,2,\ldots, N$ and $j=1,2$; and $\tilde{G}$ satisfies the following finiteness condition:
	\[\int_0^{t-s}\|\tilde{G}(u)\|_{\gamma(U;X)}^\frac{1}{H}\d{u}<\infty.\]
\end{enumerate}
Then $\overline{G}^*\in\gamma(\mathscr{D}^H(0,t-s;U);X)$ and the convolution integral $\xi_{s,t}$ given by \eqref{eq:xi_st} is well defined.
\end{proposition}

\begin{proof}
\autoref{prop:kernel}, \autoref{prop:characterisation_of_D} together with either \autoref{prop:equiv_norm_1} (in the case when condition 1. is satisfied) or \autoref{rem:equiv_norm} combined with the Hardy-Littlewood-Sobolev inequality (in the case when condition 2. is satisfied), and twice Minkowski inequality are used successively. 
\end{proof}

The general result in \autoref{prop:stoch_convolution_existence} is now applied to the case when the integrand has an additional algebraic structure. Let $(S(t),t\geq 0)$ be a strongly continuous semigroup of bounded linear operators acting on the space $X$ and let $\varPhi\in\mathscr{L}(U;X)$.

\begin{convention}
 In what follows, the operator-valued function $S\varPhi$ is understood as the bilinear operator $S(\cdot)\varPhi: X^*\times U\rightarrow\mathscr{D}^H(0,\infty)$ that is defined by 
	\begin{equation*}
		[S(\cdot)\varPhi](\varphi,u)(r):= \langle \varphi, S(r)\varPhi u\rangle, \quad r\geq 0, \quad \varphi\in X^*, \quad u\in U.
	\end{equation*}
A similar convention is adopted for $S(t-\cdot)\varPhi$. 
\end{convention}

For $t>0$, denote 
	\begin{equation*}
		Y_t := \int_0^t S(t-\,\cdot\,)\varPhi\d{Z}
	\end{equation*}
whenever the integrand is integrable with respect to the process $Z$. The following corollary is a direct consequence of \autoref{prop:stoch_convolution_existence} and for its purposes, the following notation is introduced. 
	
\begin{notation} 
In what follows, we denote $R_{\tilde{T}}:=|_{\tilde{T}}\otimes \mathrm{I}_U$ where $|_{\tilde{T}}:\mathscr{D}^H(T)\rightarrow \mathscr{D}^H(\tilde{T})$ is the restriction of a distribution defined on the interval $T\subseteq\R$ to the interval $\tilde{T}\subseteq T$ and where $\mathrm{I}_U: U\rightarrow U$ is the identity operator. Note that by \autoref{lem:restriction}, the operator $R_{\tilde{T}}: \mathscr{D}^H(T;U)\rightarrow\mathscr{D}^H(\tilde{T};U)$ is continuous.
\end{notation}

\begin{corollary}
\label{prop:existence_Yt}
Let $t_0>0$ be fixed. Then $Y_{t_0}$ is well defined if and only if $[R_{[0,t_0]}\overline{S\varPhi}]^*\in\gamma(\mathscr{D}^H(0,t_0;U);X)$. In that case, it holds for every $r>0$ that
	\begin{equation*}
		\|Y_{t_0}\|_{L^r(\Omega;X)} \eqsim \|[R_{[0,t_0]}\overline{S\varPhi}]^*\|_{\gamma(\mathscr{D}^H(0,t_0;U);X)}.
	\end{equation*}
\end{corollary}

In the following two results it is shown that it is enough to verify the existence of $Y_{t_0}$ for some $t_0>0$ to obtain the existence and measurability of the whole process $(Y_t)_{t\geq 0}$. 

\begin{proposition}
\label{prop:existence_of_Yt_for_every_t}
The following statements are equivalent.
\vspace{-2.6mm}
	\begin{enumerate}[label=\arabic*)]
		\itemsep0em 
		\item\label{stat:1} There exists $t_0>0$ such that the random variable $Y_{t_0}$ is well defined.
		\item\label{stat:2} For every $t>0$, the random variable $Y_t$ is well defined.
	\end{enumerate}
\end{proposition}

\begin{proof}
Let statement \textit{\ref{stat:1}} be satisfied and assume first that $0<t<t_0$. By using \autoref{prop:existence_Yt} and the ideal property of $\gamma$-radonifying operators from \cite[Theorem 6.2]{Nee10} (together with \autoref{lem:restriction}) the following estimate is obtained:
	\begin{align*}
		\|Y_t\|_{L^2(\Omega;X)} & \eqsim \|[R_{[0,t]}\overline{S\varPhi}]^*\|_{\gamma(\mathscr{D}^H(0,t;U);X)}\\
		& = \|[R_{[0,t]}R_{[0,t_0]}\overline{S\varPhi}]^*\|_{\gamma(\mathscr{D}^H(0,t;U);X)} \\
		& \leq \|[R_{[0,t_0]}\overline{S\varPhi}]^*\|_{\gamma(\mathscr{D}^H(0,t_0;U);X)} \|R_{[0,t]}^*\|_{\mathscr{L}(\mathscr{D}^H(0,t;U);\mathscr{D}^H(0,t_0;U))}\\
		& = \|[R_{[0,t_0]}\overline{S\varPhi}]^*\|_{\gamma(\mathscr{D}^H(0,t_0;U);X)} \|R_{[0,t]}\|_{\mathscr{L}(\mathscr{D}^H(0,t_0;U);\mathscr{D}^H(0,t;U))}\\
		& \lesssim \|Y_{t_0}\|_{L^2(\Omega;X)}
	\end{align*}
which is finite by assumption. Assume now that $t_0<t\leq 2t_0$ and write
	\begin{equation*}
		Y_t = S(t-t_0)\int_0^{t_0}S(t_0-\cdot)\varPhi\d{Z} + \int_{t_0}^t S(t-\cdot)\varPhi\d{Z}
	\end{equation*}
which is possible by virtue of \autoref{lem:link_between_scalar_and_vector_integral} and \autoref{prop:char_stoch_int_scalar}. For the first term on the right-hand side of the above equality, we obtain the estimate
	\begin{equation*}
		\|S(t-t_0)Y_{t_0}\|_{L^2(\Omega;X)} \leq M_S\mathrm{e}^{\kappa_St_0}\|Y_{t_0}\|_{L^2(\Omega;X)}
	\end{equation*}
whose right-hand side is finite by assumption. Here and in the rest of the section, $M_S\geq 1$ and $\kappa_S>0$ are finite constants such that the inequality $\|S(r)\|_{\mathscr{L}(X)}\leq M_S\mathrm{e}^{\kappa_Sr}$ holds for every $r\geq 0$. For the second term, it follows by \autoref{prop:stoch_convolution_existence}, the ideal property of $\gamma$-radonifying operators, and \autoref{prop:existence_Yt} that
	\begin{align*}
		\left\|\int_{t_0}^t S(t-\cdot)\varPhi\d{Z}\right\|_{L^2(\Omega;X)} & \eqsim \|[R_{[0,t-t_0]}\overline{S\varPhi}]^*\|_{\gamma(\mathscr{D}^H(0,t-t_0;U);X)} \\
		& = \|[R_{[0,t-t_0]}R_{[0,t_0]}\overline{S\varPhi}]^*\|_{\gamma(\mathscr{D}^H(0,t-t_0;U);X)}\\
		& \lesssim \|[R_{[0,t_0]}\overline{S\varPhi}]^*\|_{\gamma(\mathscr{D}^H(0,t_0;U);X)} \\
		& \eqsim \|Y_{t_0}\|_{L^2(\Omega;X)}
	\end{align*}
where again the last expression is finite by assumption. Thus, we have that $\|Y_t\|_{L^2(\Omega;X)}<\infty$ holds for every $t\in (0,2t_0]$. Applying this result to $\tilde{t}_0=2t_0$ yields the claim for $t\in (0,4t_0]$ and by continuing in this manner, statement \textit{\ref{stat:2}} of the proposition is proved. 
\end{proof}

\begin{lemma}
\label{lem:gamma_norm_0}
Let $\mathcal{U}$ and $\{\mathcal{U}_n\}_{n\in\mathbb{N}}$ be Hilbert spaces and let $\mathcal{X}$ be a Banach space. Let $G\in\gamma(\mathcal{U};\mathcal{X})$ and let $\{R_n\}_{n\in\mathbb{N}}$ be a sequence of bounded linear operators $R_n: \mathcal{U}\rightarrow \mathcal{U}_n$ such that there is the convergence
	\begin{equation*}
		\lim_{n\rightarrow\infty} \|R_nu\|_{\mathcal{U}_n}= 0
	\end{equation*}
for every $u\in \mathcal{U}$. Then 
	\begin{equation}
	\label{eq:lim}
		\lim_{n\rightarrow\infty}\|GR_n^*\|_{\gamma(\mathcal{U}_n;\mathcal{X})} = 0.
	\end{equation}
\end{lemma}

\begin{proof}
If $G=x\otimes u$ for some $x\in \mathcal{X}$ and $u\in \mathcal{U}$. Then it follows that $GR_n^*=(R_nu)\otimes x$ and 
	\begin{equation*}
		\lim_{n\rightarrow\infty} \|GR_n^*\|_{\gamma(\mathcal{H}_n;\mathcal{X})} = \lim_{n\rightarrow\infty} \|(R_nu)\otimes x\|_{\gamma(\mathcal{H}_n;\mathcal{X})} = \lim_{n\rightarrow\infty} \|R_nu\|_{\mathcal{U}_n}\|x\|_{\mathcal{X}} = 0.
	\end{equation*}
Consequently, if $G$ is of finite rank, then convergence \eqref{eq:lim} is satisfied. Now, if $G$ is an operator from $\gamma(\mathcal{U};\mathcal{X})$ (not necessarily of finite rank), then there exists a sequence of finite rank operators $\{G_k\}_{k\in\mathbb{N}}\subset \mathscr{L}(\mathcal{U};\mathcal{X})$ such that 
	\begin{equation}
	\label{eq:G_k_converge_to_G}
		\lim_{k\rightarrow\infty} \|G_k-G\|_{\gamma(\mathcal{U};\mathcal{X})} = 0.
	\end{equation}
It follows that 
	\begin{equation*}
		\|GR_n^*\|_{\gamma(\mathcal{U}_n;\mathcal{X})} \leq \|G-G_k\|_{\gamma(\mathcal{U};\mathcal{X})} \sup_{m\in\mathbb{N}} \|R_m\|_{\mathscr{L}(\mathcal{U};\mathcal{U}_n)} + \|G_kR_n^*\|_{\gamma(\mathscr{U};\mathcal{X})}
	\end{equation*}
by the ideal property of $\gamma$-radonifying operators. Consequently, there is the estimate
	\begin{equation*}
		\limsup_{n\rightarrow\infty} \|GR_n^*\|_{\gamma(\mathcal{U}_n;\mathcal{X})} \leq \|G-G_k\|_{\gamma(\mathcal{U};\mathcal{X})} \sup_{m\in\mathbb{N}} \|R_m\|_{\mathscr{L}(\mathcal{U};\mathcal{U}_m)}
	\end{equation*}
by the above reasoning for finite rank operators; and the claim of the lemma follows by convergence \eqref{eq:G_k_converge_to_G} and by the fact that $\sup_{m\in\mathbb{N}}\|R_m\|_{\mathscr{L}(\mathcal{U};\mathcal{U}_n)}$ is finite by the uniform boundedness principle.
\end{proof}

\begin{proposition}
\label{prop:Y_t_mean_square_continuity}
If there exists $t_0>0$ such that $Y_{t_0}$ is well defined, then the process $(Y_t)_{t\geq 0}$ is mean-square right continuous and, in particular, it admits a measurable version.
\end{proposition}

\begin{proof}
By \autoref{prop:existence_of_Yt_for_every_t}, the convolution integral $Y_t$ is a well defined element of $L^2(\Omega;X)$ for every $t\geq 0$. It will be shown that the process $(Y_t)_{t\geq 0}$ is mean-square continuous from the right. To this end, let $0\leq s<t\leq \tau$ be fixed and note that by virtue of \autoref{lem:link_between_scalar_and_vector_integral} and \autoref{prop:char_stoch_int_scalar}, we can write
	\begin{equation*}
		Y_t-Y_s = [S(t-s)-\mathrm{I}_X]\int_0^sS(s-\cdot)\varPhi\d{Z} + \int_s^tS(t-\cdot)\varPhi\d{Z}
	\end{equation*}
where $\mathrm{I}_X$ is the identity operator acting on the space $X$. Consequently, there is the following estimate:
	\begin{equation*}
		\|Y_t-Y_s\|_{L^2(\Omega;X)} \leq I_1(s,t) + I_2(s,t)
	\end{equation*}
where $I_1(s,t)$ and $I_2(s,t)$ are defined by
	\begin{align*}
		I_1(s,t) &:= \left\|[S(t-s)-\mathrm{I}_X]\int_0^sS(s-\cdot)\varPhi\d{Z}\right\|_{L^2(\Omega;X)} \quad\mbox{and}\quad I_2(s,t) := \left\|\int_s^t S(t-\cdot)\varPhi\d{Z}\right\|_{L^2(\Omega;X)}.
	\end{align*}
We have that for $\mathbb{P}$-almost every $\omega\in\Omega$, the value $Y_s(\omega)$ belongs to $X$. Therefore, for $\mathbb{P}$-almost every $\omega\in\Omega$, we have by the strong continuity of the semigroup $S$ that $[S(t-s)-\mathrm{I}_X]Y_s(\omega)$ tends to the zero element in $X$ as $t\rightarrow s+$. Moreover, since the estimate
	\begin{equation*}
		\|[S(t-s)-\mathrm{I}_X]Y_s(\omega)\|_X \leq \|S(t-s)-\mathrm{I}_X\|_{\mathscr{L}(X)}\|Y_s(\omega)\|_X \leq (M_S\mathrm{e}^{\kappa_S(\tau-s)}+1)\|Y_s(\omega)\|_X
	\end{equation*}
holds for $\mathbb{P}$-almost every $\omega\in\Omega$ and since $\E\|Y_s\|_{X}^2 <\infty$, it follows by the dominated convergence theorem that $\lim_{t\rightarrow s+}I_1(s,t)=0$. Assuming that $0<t-s<t_0$, we have for $I_2(s,t)$ that
	\begin{equation*}
		I_2(s,t) \eqsim \|[R_{[0,t-s]}\overline{S\varPhi}]^*\|_{\gamma(\mathscr{D}^H(0,t-s;U);X)} = \|\overline{S\varPhi}^*R_{[0,t-s]}^*\|_{\gamma(\mathscr{D}^H(0,t-s;U);X)}
	\end{equation*}
Now, for every $g\in\mathscr{D}^H(0,t_0;U)$, we have that $\lim_{t\rightarrow s+}\|R_{[0,t-s]}g\|_{\mathscr{D}^H(0,t-s;U)}=0$ by \autoref{lem:norm_tends_to_zero}. Consequently, it follows by \autoref{lem:gamma_norm_0} that also $\lim_{t\rightarrow s+}I_2(s,t)=0$.
\end{proof}

For convenience, \autoref{prop:Y_t_mean_square_continuity} is restated by using the characterisation in \autoref{prop:stoch_convolution_existence}.

\begin{corollary}
\label{cor:NSC_for_measurability}
If there exists $t_0>0$ such that the finiteness condition 
	\begin{equation*}
		\|[R_{[0,t_0]}\overline{S\varPhi}]^*\|_{\gamma(\mathscr{D}^H(0,t_0;U);X)}<\infty
	\end{equation*}
is satisfied, then the process $(Y_t)_{t\geq 0}$ is mean-square right continuous and, in particular, it admits a measurable version.
\end{corollary}

If $S$ is an analytic semigroup, then there exists $\lambda >0$ such that $(\lambda\mathrm{I}_X-A)$ is a non-negative operator. Therefore, the fractional powers $(\lambda \mathrm{I}_X-A)^{\alpha}$, $\alpha\geq 0$, can be defined; see, e.g., \cite[section 2.6]{Pazy}. For $0< \alpha\leq 1$, denote by $X_\alpha$ the Banach space that is defined as the space $\Dom (\lambda\mathrm{I}_X-A)^\alpha$ endowed with the graph norm of $(\lambda\mathrm{I}_X-A)^\alpha$. For $t>0$ and $\alpha\in (0,1]$, denote also
	\begin{equation*}
		Y_t^\alpha := \int_0^t (\lambda \mathrm{I}_X-A)^\alpha S(t-\cdot)\varPhi \d{Z}
	\end{equation*}
whenever the integrand is strongly integrable with respect to the process $Z$. The definitions are extended to allow for $\alpha=0$ by setting $X_0:=X$ and $Y_t^0:=Y_t$. 

\begin{proposition}
\label{prop:Y_t_alpha_mean_square_continuity_X}
Assume that the semigroup $S$ is analytic. If there exists $t_0>0$ and $0<\alpha\leq 1$ such that the random variable $Y_{t_0}^\alpha$ is well defined, then the process $(Y_t^{\alpha})_{t\geq 0}$ is mean-square right continuous and, in particular, it admits a measurable version.
\end{proposition}

\begin{proof}
The claim follows by repeating the proofs of \autoref{prop:existence_of_Yt_for_every_t} and \autoref{prop:Y_t_mean_square_continuity} with $Y_t$ and $S\varPhi$ replaced by $Y_t^\alpha$ and $(\lambda\mathrm{I}_X-A)^\alpha S\varPhi$, respectively. 
\end{proof}

\begin{corollary}
\label{cor:Y_t_alpha-beta-beta}
Assume that the semigroup $S$ is analytic. If there exists $t_0>0$ and $0\leq\alpha\leq 1$ such that the random variable $Y_{t_0}^\alpha$ is well defined, then for every $\beta_1,\beta_2\geq 0$ such that $\beta_1+\beta_2\leq \alpha$, the process $(Y_t^{\beta_1})_{t\geq 0}$ is mean-square right continuous in the space $X_{\beta_2}$ and, in particular, it admits a measurable version there. 
\end{corollary}

\begin{proof}
Note first that if $\kappa>0$ and $\gamma_1, \gamma_2\geq 0$ are such that $\gamma_1+\gamma_2\leq \kappa$, the operator $(\lambda\mathrm{I}_X-A)^{\kappa-(\gamma_1+\gamma_2)}$ is invertible and there is the inequality
	\begin{align*}
		\|x\|_{X_{\gamma_2}}^2 & \leq \|(\lambda\mathrm{I}_X-A)^{\gamma_1+\gamma_2-\kappa}\|_{\mathscr{L}(X;X_{\kappa-\gamma_1-\gamma_2})}^2 \\
		& \hspace{1cm}\times \left[ \|(\lambda\mathrm{I}_X - A)^{\kappa-(\gamma_1+\gamma_2)}x\|_X^2 + \|(\lambda\mathrm{I}_X-A)^{\kappa-\gamma_1}x\|_X^2\right]\numberthis \label{eq:ineq_X}
	\end{align*}
for $x\in X_{\kappa-\gamma_1}$. Note also that if $t>0$ and $0<\delta\leq 1$ are such that $\|Y_t^\delta\|_{L^2(\Omega;X)}<\infty$, then $Y_t\in X_\delta$ and 
	\begin{equation}
	\label{eq:Yd=dY}
		Y_t^\delta = (\lambda\mathrm{I}_X-A)^\delta Y_t
	\end{equation}
hold $\mathbb{P}$-almost surely by \autoref{prop:char_stoch_int_scalar} and \autoref{lem:char_weak_stoch_integrability}. Now, if $0\leq \beta\leq \alpha$, then the above inequality \eqref{eq:ineq_X} (with $\gamma_1=\beta$, $\gamma_2=0$) and equality \eqref{eq:Yd=dY} (with $\delta = \beta$) yield
	\[ \E\| Y_t^\beta-Y_s^\beta\|_{X}^2 \lesssim \E\| (\lambda\mathrm{I}_X-A)^{\alpha-\beta}[Y_t^\beta-Y_s^\beta]\|_X^2 = \E\|Y_t^\alpha-Y_s^\alpha\|_X^2\]
for $0\leq s<t$. Consequently, for every $0\leq \beta\leq \alpha$, the process $(Y_t^\beta)_{t\geq 0}$ is mean-square right continuous in $X$ by \autoref{prop:Y_t_alpha_mean_square_continuity_X}. The general claim of the corollary then follows from the equality
	\begin{equation*}
		\|Y_t^{\beta_1}-Y_s^{\beta_1}\|_{L^2(\Omega;X_{\beta_2})}^2 = \E \|Y_t^{\beta_1}-Y_s^{\beta_1}\|_X^2 + \E \|Y_t^{\beta_1+\beta_2}-Y_s^{\beta_1+\beta_2}\|_X^2
	\end{equation*}
for $0\leq s<t$ by using the first part of the proof with $\beta=\beta_1$ and $\beta=\beta_1+\beta_2$. 
\end{proof}

For convenience, \autoref{cor:Y_t_alpha-beta-beta} is restated by using the characterisation in \autoref{prop:stoch_convolution_existence}.

\begin{corollary}
\label{rem:existence_Y_alpha}
Assume that the semigroup $S$ is analytic. If there exists $t_0>0$ and $0\leq\alpha\leq 1$ such that the finiteness condition
	\begin{equation*}
		\| [R_{[0,t_0]}\overline{(\lambda\mathrm{I}_X-A)^\alpha S\varPhi}^*\|_{\gamma(\mathscr{D}^H(0,t_0;U);X)}<\infty
	\end{equation*}
is satisfied, then for every $\beta_1,\beta_2\geq 0$ such that $\beta_1+\beta_2\leq \alpha$, the process $(Y_t^{\beta_1})_{t\geq 0}$ is mean-square right continuous in the space $X_{\beta_2}$ and, in particular, it admits a measurable version there. 
\end{corollary}

In what follows, sufficient conditions for continuity of the integral process $(Y_t)_{t\geq 0}$ are given. The proof of \autoref{prop:continuity} is based on the Kolmogorov-Chentsov theorem; however, in the literature on stochastic convolution in Banach spaces, the factorization method (see, e.g., \cite{CouMasOnd18,DaPKwaZab87}), $\gamma$-boundedness, or maximal inequalities (see, e.g., \cite{KwaVerWei16,NeeVerWei15b,NeeVer20,NeeVerWei12, NeeVerWei15c,NeeZhu11,VerWei11}) can be also used.

\begin{proposition}
\label{prop:continuity}
Assume that the semigroup $S$ is anaytic. Assume also that there exist constants $\beta_1>0$ and $\beta_2\geq 0$ such that $\beta_1+\beta_2\leq 1$ and such that the following two conditions are satisfied:
	\begin{enumerate}[label=(H\arabic*)]
	\itemsep0em
		\item\label{eq:finiteness_Y_delta} There exists $t_0>0$ such that the random variable $Y^{\beta_1+\beta_2}_{t_0}$ is well defined.
		\item\label{eq:bound} It holds that $\|Y_t^{\beta_2}\|_{L^2(\Omega;X)} \lesssim t^{\beta_1}$ for every $t\in [0,\tau]$, $\tau>0$.
	\end{enumerate}
Then the process $(Y_t)_{t\geq 0}$ has a version with paths in $\mathscr{C}^{\kappa}([0,\tau]; X_{\beta_2})$ for every $0\leq \kappa<\beta_1$ and $\tau>0$.
\end{proposition}

\begin{proof}
Condition \ref{eq:finiteness_Y_delta} implies by \autoref{cor:Y_t_alpha-beta-beta} that all the processes $(Y_t)_{t\geq 0}$, $(Y_t^{\beta_2})_{t\geq 0}$, and $(Y_t^{\beta_1+\beta_2})_{t\geq 0}$ are well defined and mean-square right continuous in $X$. Moreover, by the proofs of \autoref{prop:Y_t_mean_square_continuity} and \autoref{cor:Y_t_alpha-beta-beta}, we also have that for every $\tau>0$, it holds  
	\begin{equation}
	\label{eq:finite_sup_2}
		\sup_{0\leq t\leq \tau} \|Y_t\|_{L^2(\Omega;X)} \lesssim \sup_{0\leq t\leq \tau} \|Y_t^{\beta_i}\|_{L^2(\Omega;X)}  \lesssim \sup_{0\leq t\leq \tau} \|Y_t^{\beta_1+\beta_2}\|_{L^2(\Omega;X)} <\infty.
	\end{equation}
for both $i=1,2$. Let $0\leq s<t\leq \tau$ be fixed and write
	\begin{equation*}
		Y_t-Y_s = [S(t-s)-\mathrm{I}_X]\int_0^s S(s-\cdot)\varPhi\d{Z} + \int_s^t S(t-\cdot)\varPhi\d{Z}.
	\end{equation*}
Since $S$ is analytic, we have that $S(r): X\rightarrow \cap_{n=0}^\infty \Dom A^n\subset \Dom A$ for every $r>0$. Hence, by \cite[Theorem 1.2.4 (d)]{Pazy}, it holds that 
	\begin{equation*}
		\E\|[S(t-s)-\mathrm{I}_X]Y_s\|_{X_{\beta_2}}^2 = \E\left\|\lambda \int_0^{t-s}S(u)Y_s\d{u} - \int_0^{t-s}(\lambda \mathrm{I}_X-A)S(u)Y_s\d{u}\right\|_{X_{\beta_2}}^2.
	\end{equation*}
For the first integral, there is the estimate
	\begin{equation*}
		\E\left\|\lambda \int_0^{t-s}S(u)Y_s\d{u}\right\|_{X_{\beta_2}}^2 \leq \lambda^2 \left(\int_0^{t-s}\|S(u)\|_{\mathscr{L}(X)}\d{u}\right)^2 \E\|Y_s\|_{X_{\beta_2}}^2 \lesssim (t-s)^2 \E\|Y_s^{\beta_2}\|_X^2.
	\end{equation*}
For the second integral, note that the operator $A-\lambda\mathrm{I}_X$ generates the semigroup $(\mathrm{e}^{-\lambda r}S(r), r\geq 0)$. By \cite[Theorem 2.6.13 (c)]{Pazy}, there is the estimate
	\begin{align*}
		\E\left\|\int_0^{t-s} (\lambda\mathrm{I}_X-A)S(u)Y_s\d{u}\right\|_{X_{\beta_2}}^2 & \eqsim \E\left\|\int_0^{t-s} \mathrm{e}^{\lambda u}(\lambda\mathrm{I}_X-A)^{1-\beta_1}\mathrm{e}^{-\lambda u}S(u)(\lambda\mathrm{I}_X-A)^{\beta_1+\beta_2} Y_s\d{u}\right\|_X^2\\
		& \lesssim \E\left(\int_0^{t-s} u^{\beta_1 -1}\| Y_s^{\beta_1+\beta_2}\|_X\d{u}\right)^2\\
		& \leq (t-s)^{2\beta_1}\E\| Y_s^{\beta_1+\beta_2}\|_X^2.
	\end{align*}
Thus, by appealing to \eqref{eq:finite_sup_2}, the following estimate is obtained:
	\begin{equation*}
		\E\|[S(t-s)-\mathrm{I}_X]Y_s\|_{X_{\beta_2}}^2 \lesssim (t-s)^{2\beta_1}.
	\end{equation*}
Now, we also have by \autoref{prop:stoch_convolution_existence} and assumption \ref{eq:bound} that 	
	\begin{align*}
		\E\left\|\int_s^t S(t-\cdot)\varPhi\d{Z}\right\|_{X_{\beta_2}}^2 & = \E\left\|\int_s^tS(t-\cdot)\varPhi\d{Z}\right\|_X^2 + \E\left\|\int_s^t(\lambda\mathrm{I}_X-A)^{\beta_2}S(t-\cdot)\varPhi\d{Z}\right\|_X^2\\ 
		& \eqsim \|[R_{[0,t-s]}\overline{S\varPhi}]^*\|_{\gamma(\mathscr{D}^H(0,t-s;U);X)}^2 + \|[R_{[0,t-s]}\overline{(\lambda\mathrm{I}_X-A)^{\beta_2}S\varPhi}]^*\|_{\gamma(\mathscr{D}^H(0,t-s;U);X)}^2 \\
		& \eqsim \|Y_{t-s}\|_{L^2(\Omega;X_{\beta_2})}^2 + \|Y_{t-s}^{\beta_2}\|_{L^2(\Omega;X)}^2 \\
		& = \|Y_{t-s}\|_{X_{\beta_2}}^2 \\
		& \lesssim (t-s)^{2\beta_1}
	\end{align*}
and the proof is concluded by appealing to the fact that $Y_t-Y_s$ has equivalent moments (see equivalence \eqref{eq:equiv_moments} and \autoref{prop:hypercontractivity}) and to the Kolmogorov-Chentsov continuity criterion.
\end{proof}

\section{Applications}
\label{sec:examples}

In this section, two applications are presented. In the first application, a parabolic equation of order $2m$ with distributed space-time noise of low time regularity is considered. This application extends the results of \cite[section 5.2]{CouMasOnd18} to noise that is singular in time. In the second application, the heat equation with space-time noise of higher regularity in time in the Neumann boundary condition is studied. This application extends some results of \cite[section 5]{DunMasDun02} from the Hilbert-space setting (in \cite{DunMasDun02}, the space $L^2(D)$ is considered) to the Banach-space setting (here, the space $L^p(D)$ is considered) and \cite{SchVer11} to a noise that is more regular in time and that is possibly non-Gaussian (in \cite{SchVer11}, the driving noise is the Wiener process while here, for example, the Rosenblatt process can be considered).

\subsection{Parabolic equation of order \texorpdfstring{$2m$}{2m} with distributed singular noise in \texorpdfstring{$L^p(D)$}{Lp(D)}}

Let $D\subseteq \R^d$, $d\in\mathbb{N}$, be a bounded open domain with smooth boundary $\partial D$. Let $m\in\mathbb{N}$ and let $\tilde{L}_{2m}$ be the differential operator of order $2m$ given by 
	\[ \tilde{L}_{2m} = \sum_{|k|\leq 2m}a_k\partial^k\]
where $a_k\in \mathscr{C}_b^\infty(D)$ and assume that $\tilde{L}_{2m}$ is uniformly elliptic. Let also $\eta$ be space-time random noise that is fractional in time. Consider the parabolic equation on $D$ perturbed by noise $\eta$ at every point of $D$ that is formally given by 
	\begin{equation}
	\label{eq:heat_distributed_eq}
	(\partial_t u)(t,x) = (\tilde{L}_{2m} u)(t,x) + \eta_t(x), \quad (t,x)\in (0,\infty)\times D, 
	\end{equation}
and that is subject to the initial condition 
	\begin{equation}
	\label{eq:heat_distributed_IC}
	u(0,x) = u_0(x), \quad x\in D, 
	\end{equation}
and the Dirichlet boundary condition 
	\begin{equation}
	\label{eq:heat_distributed_BC}
	(\partial_\nu^ku)(t,x)=0, \quad (t,x)\in (0,\infty)\times \partial D, \quad k\in \{0,1,\ldots,m-1\}.
	\end{equation}
The problem \eqref{eq:heat_distributed_eq} - \eqref{eq:heat_distributed_BC} is given rigorous meaning as an abstract Cauchy problem for a linear stochastic evolution equation and the solution is sought in the mild form. 

In particular, let $p\geq 2$ and let $L_{2m}$ be the realisation of the operator $\tilde{L}_{2m}$ in the space $L^p(D)$ with the domain 
	\[\Dom L_{2m} = \left\{ f\in W^{2m,p}(D)\,\left|\, \partial_\nu^kf=0\mbox{ on } \partial D \mbox{ for } k\in \{0,1,\ldots, m-1\}\right.\right\}. \] 
It is well-known that the operator $L_{2m}$ is sectorial. Let $\lambda >0$ be such that $(\lambda\mathrm{I}-L_{2m})$ is positive and we denote by $(S_{2m}(t), t\geq 0)$ the analytic semigroup generated by $L_{2m}$. Let also $H\in (0,\sfrac{1}{2})$ and let $(Z_t)_{t\geq 0}$ be an $L^2(D)$-cylindrical $H$-fractional process that lives in a finite Wiener chaos (for example, the cylindrical fractional Brownian motion of $H\in (0,\sfrac{1}{2})$ can be considered). Let also $C\in\mathscr{L}(L^2(D))$\footnote{The noise $\eta$ in equation \eqref{eq:heat_distributed_eq} is modeled by $\eta_t = C\dot{Z}_t$ where the dot represents the formal time derivative.}. The solution to problem \eqref{eq:heat_distributed_eq} - \eqref{eq:heat_distributed_BC} is then interpreted as the $L^p(D)$-valued stochastic process $(Y_t^{2m})_{t\geq 0}$ where
	\begin{equation}
	\label{eq:heat_stochastic_convolution}
		Y_t^{2m} := \int_0^t S_{2m}(t-\cdot)C\d{Z}, \quad t\geq 0.
	\end{equation}
In what follows, sufficient conditions for the existence and space-time continuity of the above stochastic convolution integral are given. To this end, we begin with three simple lemmas.  

\begin{lemma}
\label{lem:heat_est}
For every $\alpha\geq 0$ and $t>0$, there is the following estimate:
	\[\|(\lambda\mathrm{I}-L_{2m})^\alpha S_{2m}(u)C\|_{\gamma(L^2(D);L^p(D))} \lesssim u^{-\frac{d}{4m}-\alpha}, \quad 0<u\leq t.\]
\end{lemma}

\begin{proof}
Note initially that there is the estimate
	\begin{equation}
	\label{eq:est_heat_gamma}
		 \|S_{2m}(u)C\|_{\gamma(L^2(D);L^p(D))} \lesssim u^{-\frac{d}{2m}}, \quad u>0.
	\end{equation}
Indeed, for $u>0$, there is the inequality 
	\[ \|S_{2m}(u)C\|_{\gamma(L^2(D);L^p(D))} \leq \|S_{2m}(u)\|_{\gamma(L^2(D);L^p(D))} \|C\|_{\mathscr{L}(L^2(D))}\]
and by appealing to \autoref{prop:kernel} and to the estimate in \cite[Theorem 1.1]{Eid70} for Green's kernel of the semigroup $S_{2m}$, we also have
	\begin{align*}
		\|S_{2m}(u)\|_{\gamma(L^2(D);L^p(D))}  \eqsim \left[\int_D\left(\int_D u^{-\frac{d}{m}}\mathrm{e}^{-c\left(\frac{|x-y|^{2m}}{u}\right)^{\frac{1}{2m-1}}}\d{y}\right)^\frac{p}{2}\d{x}\right]^{\frac{1}{p}} \lesssim u^{-\frac{d}{2m}} |D|^{\frac{1}{2}+\frac{1}{p}}
	\end{align*}
where $c$ is some finite positive constant. Now, for $0<u\leq t$, we have that
	\begin{align*}
		\|(\lambda\mathrm{I}-L_{2m})^\alpha S_{2m}(u)C\|_{\gamma(L^2(D);L^p(D))} & \\
		& \hspace{-2cm} =\left\|(\lambda\mathrm{I}-L_{2m})^\alpha S_{2m}\left(\frac{u}{2}\right)S_{2m}\left(\frac{u}{2}\right)C\right\|_{\gamma(L^2(D);L^p(D))}\\
			&\hspace{-2cm} \leq\mathrm{e}^{\frac{\lambda u}{2}}\left\|(\lambda\mathrm{I}-L_{2m})^\alpha \mathrm{e}^{-\frac{\lambda u}{2}}S_{2m}\left(\frac{u}{2}\right)\right\|_{\mathscr{L}(L^p(D))} \left\|S_{2m}\left(\frac{u}{2}\right)C\right\|_{\gamma(L^2(D);L^p(D))}
	\end{align*}
and the claim follows by using \cite[Theorem 2.6.13 c)]{Pazy} (since the operator $(L_{2m}-\lambda\mathrm{I})$ generates the semigroup $(\mathrm{e}^{-\lambda r}S_{2m}(r), r\geq 0)$) and estimate \eqref{eq:est_heat_gamma}.
\end{proof}

\begin{lemma}
\label{lem:heat_est_2}
For every $\alpha\geq 0$, $\beta\geq \alpha$, and $t>0$, there is the following estimate:
	\[\|(\lambda\mathrm{I}-L_{2m})^\alpha [S_{2m}(u)-S_{2m}(v)]C\|_{\gamma(L^2(D);L^p(D))}\lesssim (u-v)^{\beta - \alpha} v^{-\frac{d}{4m}-\beta}, \quad 0<u<v\leq t.\]
\end{lemma}

\begin{proof}
We have that 
	\begin{align*}
		\|(\lambda\mathrm{I}-L_{2m})^\alpha [S_{2m}(u)-S_{2m}(v)]C\|_{\gamma(L^2(D);L^p(D))} & \\
		& \hspace{-6cm} = \|(\lambda\mathrm{I}-L_{2m})^\alpha [S_{2m}(u-v)-\mathrm{I}]S_{2m}(v)C\|_{\gamma(L^2(D);L^p(D))}\\
		& \hspace{-6cm} \lesssim \left\|\int_0^{u-v} \mathrm{e}^{\lambda r} (\lambda\mathrm{I}-L_{2m})^{1+\alpha-\beta} \mathrm{e}^{-\lambda r} S_{2m}(r)\d{r}\right\|_{\mathscr{L}(L^p(D))} \|(\lambda\mathrm{I}-L_{2m})^\beta S_{2m}(v)C\|_{\gamma(L^2(D);L^p(D))}
	\end{align*}
and the claim follows by using \cite[Theorem 2.6.13 c)]{Pazy} and \autoref{lem:heat_est}.
\end{proof}

\begin{lemma}
\label{lem:Y_t^a_heat}
If $d< 2m$, $H\in (\sfrac{d}{4m}, \sfrac{1}{2})$, and $\alpha\in [0,H-\sfrac{d}{4m})$, then for every $t\geq 0$, the convolution integral
	\[Y_t^{\alpha,2m}:=\int_0^t (\lambda\mathrm{I}-L_{2m})^\alpha S_{2m}(t-\cdot)C\d{Z}\]
is well defined and satisfies
	\[\|Y_t^{\alpha,2m}\|_{L^2(\Omega;L^p(D))} \lesssim t^{H-\alpha-\frac{d}{4m}}.\] 
\end{lemma}

\begin{proof}
Let $t>0$. By \autoref{prop:stoch_convolution_existence}, \autoref{prop:if_p_greater_2}, and \autoref{prop:characterisation_of_D} there is the estimate 
	\[ \|Y_t^{\alpha,2m}\|_{L^2(\Omega;L^p(D))} \lesssim \|(\lambda I-L_{2m})^\alpha S_{2m}(\cdot)C\|_{\dot{W}^{\frac{1}{2}-H,2}(0,t;\gamma(L^2(D);L^p(D)))}.
	\]
By \autoref{lem:heat_est}, there is the estimate
	\begin{equation*}
		\int_0^{t} \|(\lambda\mathrm{I}-L_{2m})^\alpha S_{2m}(u)C\|_{\gamma(L^2(D);L^p(D))}^2\d{u}  \lesssim \int_0^{t}u^{-2\alpha-\frac{d}{2m}}\d{u} \propto t^{1-2\alpha-\frac{d}{2m}}.
	\end{equation*}
Choose $\beta\in (1-2H+2\alpha, 1-\sfrac{d}{2m})$. Such a choice is possible since $2\alpha<2H-\sfrac{d}{2m}$. By \autoref{lem:heat_est_2}, there is the estimate 
	\begin{align*}
		\int_0^{t}\int_0^{t} \|(\lambda\mathrm{I}-L_{2m})^\alpha [S_{2m}(u)-S_{2m}(v)]C\|_{\gamma(L^2(D);L^p(D))}^2|u-v|^{2H-2}\d{u}\d{v} & \\
			&\hspace{-6cm} \lesssim \int_0^{t}\int_0^u (u-v)^{\beta-2\alpha+2H-2}v^{-\frac{d}{2m}-\beta}\d{v}\d{u}\\
			& \hspace{-6cm} = \mathrm{B}\left(\beta-1+2H-2\alpha,1-\frac{d}{2m}-\beta\right) \int_0^{t} u^{-2\alpha+2H-1-\frac{d}{2m}}\d{u} \\
			& \hspace{-6cm} \propto t^{2H-2\alpha-\frac{d}{2m}}
	\end{align*}
and the claim follows. 
\end{proof}
	
The previous lemmas are now used to find sufficient conditions for the stochastic convolution \eqref{eq:heat_stochastic_convolution} to be a process that is continuous in time and takes values in certain Bessel potential spaces.

\begin{proposition}
\label{prop:continuity_in_domains}
Assume that $d<2m$ and $H\in (\sfrac{d}{4m},\sfrac{1}{2})$. Then:
	\begin{enumerate}
	\itemsep0em
	\item The convolution integral $(Y_t^{2m})_{t\geq 0}$ admits a version with paths in $\mathscr{C}^{\kappa}([0,\tau];H_{\partial_\tau}^{2m\alpha,p}(D))$ for every $\alpha \in [0,H-\sfrac{d}{4m})$ such that $2m\alpha-\sfrac{1}{p}\not\in \{0,1,\ldots, m-1\}$, every $\kappa\in [0,H-\sfrac{d}{4m}-\alpha)$, and every $\tau>0$. Here, the space $H_{\partial_\tau}^{2m\alpha,p}(D)$ is defined by
		\[H^{2m\alpha,p}_{\partial_\tau}(D):=\left\{f\in H^{2m,p}(D)\,\left|\, \partial_\nu^kf=0\mbox{ on }\partial D \mbox{ for } k\in\{0,1,\ldots,m-1\}:\, k<2m\alpha-\frac{1}{p}\right.\right\}.\]
	\item The convolution integral $(Y_t^{2m})_{t\geq 0}$ admits a version with paths in $\mathscr{C}^{\kappa}([0,\tau];\tilde{H}_{\partial_\tau}^{2m\alpha,p}(D))$ for every $\alpha\in [0,H-\sfrac{d}{4m})$ such that $2m\alpha-\sfrac{1}{p}=:l\in\{0,1,\ldots, m-1\}$, every $\kappa\in [0,H-\sfrac{d}{4m}-\alpha)$, and every $\tau>0$. Here, the space $\tilde{H}_{\partial_\tau}^{2m\alpha,p}(D)$ is defined by 
		\[\tilde{H}^{2m\alpha,p}_{\partial_\tau}(D):= \{f\in H^{2m\alpha,p}_{\partial_\tau}(D)\,|\, \partial_\nu^lf\in \tilde{H}^{\frac{1}{p},p}(D)\}\]
with 
		\[\tilde{H}^{\frac{1}{p},p}(D):= \{f\in H^{\frac{1}{p},p}(\R^d)\,|\,\, \mathrm{supp}\, f\subset\overline{D}\}.\]
	\end{enumerate}
\end{proposition}

\begin{proof}
It follows from \autoref{lem:Y_t^a_heat} by appealing to \autoref{prop:continuity} with $\beta_1= H-\sfrac{d}{4m} -\alpha$ and $\beta_2=\alpha$ for $\alpha\in [0,H-\sfrac{d}{4m})$ that $(Y^{2m}_t)_{t\geq 0}$ has paths in $\mathscr{C}^\kappa([0,\tau];\Dom (\lambda\mathrm{I}-L_{2m})^\alpha)$ for every $\tau>0$. Since the operator $(\lambda\mathrm{I}-L_{2m})$ has bounded imaginary powers by \cite[Theorem 1.1]{See71}, it holds that 
	\[\Dom(\lambda\mathrm{I}-L_{2m})^\alpha = [L^p(D);\Dom (\lambda\mathrm{I}-L_{2m})]_{\alpha},\]
where $[\cdot;\cdot]_{\theta}$ denotes the complex interpolation functor, by \cite[Theorem 1.15.3]{Trie78}. Finally, the claim of the proposition follows from \cite[Theorem 4.3.3]{Trie78}.
\end{proof}

Finally, sufficient conditions for space-time continuity of the stochastic convolution \eqref{eq:heat_stochastic_convolution} are given.

\begin{proposition}
\label{prop:space-time_continuity}
If $d<\sfrac{2mp}{(p+2)}$ and $H\in (\sfrac{d}{4m} + \sfrac{d}{2mp},\sfrac{1}{2})$. Then the integral $(Y_t^{2m})_{t\geq 0}$ admits a version with paths in the space \[\mathscr{C}^{\kappa}([0,\tau];\mathscr{C}^{2m\alpha-\frac{d}{p}}(\overline{D}))\] for every $\alpha\in (\sfrac{d}{2mp},H-\sfrac{d}{4m})$ such that $2m\alpha-\sfrac{d}{p}\not\in\mathbb{N}$, every $\kappa\in [0,H-\sfrac{d}{4m}-\alpha)$, and every $\tau>0.$
\end{proposition}

\begin{proof}
The claim follows directly from the first part of \autoref{prop:continuity_in_domains} and the Sobolev embedding for the Bessel potential spaces from \cite[Theorem 4.6.1 (e)]{Trie78}.
\end{proof}

\begin{remark}
Note that it follows from \autoref{prop:space-time_continuity} that in the case of the heat equation in $d=1$ driven by a fractional process of Hurst parameter $H\in (\sfrac{1}{4},\sfrac{1}{2})$, we can always find $p$, $\alpha$, and $\kappa$ for which space-time continuity of the solution occurs.
\end{remark}

\subsection{Heat equation with regular Neumann boundary noise in \texorpdfstring{$L^p(D)$}{Lp(D)}}

Let $D\subseteq \R^d$ be a bounded open domain with smooth boundary $\partial D$ of finite surface measure. Consider the heat equation on $D$ perturbed by fractional noise through the boundary $\partial D$ that is formally given by 
	\begin{equation}
	\label{eq:heat_eq_bdry_noise}
		(\partial_t u)(t,x) = (\Delta_x u)(t,x), \quad (t,x)\in (0,\infty)\times D,
	\end{equation}
and that is subject to the initial condition 
	\begin{equation}
	\label{eq:heat_eq_bdry_noise_IC}
		u(0,x)=0, \quad x\in D, 
	\end{equation}
and the Neumann boundary noise condition 
	\begin{equation}
	\label{eq:heat_eq_bdry_noise_BC}
		(\partial_{\nu}u)(t,x) = \dot{z}_t(x), \quad (t,x)\in (0,\infty)\times \partial D.
	\end{equation}
Here, $\partial_{\nu}$ denotes the conormal derivative and $\dot{z}$ denotes a noise term that is fractional in time and uncorrelated in space. The problem \eqref{eq:heat_eq_bdry_noise} - \eqref{eq:heat_eq_bdry_noise_BC} is interpreted as an abstract Cauchy problem for a linear stochastic evolution equation driven by a cylindrical fractional process. The solution to this problem is sought in the mild form which leads to the question of existence of a stochastic convolution integral. In particular, let $H\in [\sfrac{1}{2},1)$ and let $(Z_t)_{t\geq 0}$ be an $L^2(\partial D)$-cylindrical $H$-fractional process that lives in a finite Wiener chaos. We seek $p>1$ such that the process $(Y_t)_{t\geq 0}$, still formally, given by 
	\[Y_t := \int_0^t (\lambda \mathrm{I}-\Delta_N)S_N(t-\cdot)B_N\d{Z}, \quad t\geq 0,\]
with some $\lambda>0$ is $L^p(D)$-valued and measurable. Let us fix the following notation.

\begin{notation} 
If $\mathcal{U}$ and $\mathcal{V}$ are two isomorphic Banach spaces, we write $\mathcal{U}\simeq \mathcal{V}$. 
\end{notation}

Throughout this section, $\Delta_N$ is the realisation of the Neumann Laplacian in the space $L^p(D)$ with the domain 
	\[ \Dom (\Delta_N) = W_{\partial_\nu}^{2,p}(D):= \{f\in W^{2,p}(D)\,|\, \partial_\nu f=0\mbox{ on } \partial D\},\]
the family $(S_N(t), t\geq 0)$ is the strongly continuous analytic semigroup of bounded linear operators acting on the space $L^p(D)$ generated by $\Delta_N$, and $B_N$ is the Neumann boundary operator defined by $B_Ng :=v$ where $v$ is the solution to the elliptic problem 
	\begin{alignat*}{2}
					 (\lambda \mathrm{I}- \Delta)v & =0 &&\quad  \mbox{on } D, \\
					 \partial_{\nu} v & =g &&\quad  \mbox{on }\partial D.\numberthis	\label{eq:elliptic}
	\end{alignat*}

Initially, the above convolution integral $(Y_t)_{t\geq 0}$ is rigorously defined and to this end, we begin with some preliminary observations. For $p>1$ and $\alpha\in [0,2]\setminus\{1+\sfrac{1}{p}\}$, we set
			\[ H^{\alpha,p}_{\partial_\nu}(D):= \begin{cases}
				H^{\alpha,p}(D), & \quad \alpha\in [0,\sfrac{1}{p}+1),\\
				\{f\in H^{\alpha,p}(D)\,|\, \partial_\nu f=0 \mbox{ on } \partial D\}, & \quad \alpha\in (\sfrac{1}{p}+1,2],
			\end{cases}\]
where $H^{\alpha,p}(D)$ denotes a Bessel potential space. For the purposes of the following lines, let $p\in (1,2]$.
	\begin{itemize}
		\item If $\alpha\in (\sfrac{1}{p},\sfrac{1}{p}+1)$, then there is the following Sobolev embedding
			\begin{equation*}
				\iota_1: L^2(\partial D)\hookrightarrow W^{\alpha-1-\frac{1}{p},p}(\partial D).
			\end{equation*}
		\item If $\alpha\in (\sfrac{1}{p},\sfrac{1}{p}+1)$ and $\lambda>0$, then there exists a unique weak solution $v$ in the Sobolev-Slobodeckii space $W^{\alpha,p}(D)$ (in the sense of \cite[formula (9.4)]{Ama93}) to problem \eqref{eq:elliptic} for every $g\in W^{\alpha-1-\frac{1}{p},p}(\partial D)$ by \cite[Theorem 9.2 and Remark 9.3 (e)]{Ama93}. Consequently, for the Neumann boundary map $B_N$ it holds that 
			\begin{equation*}
				B_N\in\mathscr{L}(W^{\alpha-1-\frac{1}{p}}(\partial D); W^{\alpha,p}(D)).
			\end{equation*}
		\item If $\alpha\in (\sfrac{1}{p}, \sfrac{1}{p}+1)$, then there is the following continuous embedding 
			\[\iota_2: W^{\alpha,p}(D)\hookrightarrow H^{\alpha,p}(D).\]
		Indeed, it holds by \cite[Proposition 3.2.4 (i)]{Trie83} that
			\[W^{\alpha,p}(D)\simeq F_{p,p}^\alpha(D) \hookrightarrow F_{p,2}^\alpha(D) \simeq H^{\alpha,p}(D).\]
		\item If $\alpha\in (0,2)\setminus\{1+\sfrac{1}{p}\}$, then it holds by \cite[Theorem 5.2]{Ama93} that 
			\begin{equation*}
				H_{\partial_\nu}^{\alpha,p}(D) \simeq [L^p(D);W_{\partial_\nu}^{2,p}(D)]_{\frac{\alpha}{2}}
			\end{equation*} 
		 where $[\cdot\,;\, \cdot]_\theta$ denotes the complex interpolation functor.
		\item For $\lambda>0$, the operator $(\lambda \mathrm{I}-\Delta_N)$ is positive and, consequently, its fractional powers can be defined. Moreover, by \cite[Example III.4.7.3 (d)]{Ama95}, the operator $(\lambda \mathrm{I}-\Delta_N)$ has bounded imaginary powers and, consequently, it follows by \cite[Remark 6.1 (d)]{Ama93} (see also \cite[Theorem 1.15.3]{Trie78}) that 
			\[ [L^p(D); W^{2,p}_{\partial_\nu}(D)]_{\frac{\alpha}{2}} = [\Dom (\lambda\mathrm{I}-\Delta_N)^0;\Dom (\lambda\mathrm{I}-\Delta_N)^1]_{\frac{\alpha}{2}} \simeq \Dom (\lambda\mathrm{I}-\Delta_N)^\frac{\alpha}{2}\]
		whenever $\alpha\in (0,2)$.
	\end{itemize}

With these preparatory results, we obtain the following two lemmas that lead to a rigorous definition of the integrand. The first lemma is a simple consequence of the above claims.

\begin{lemma}
\label{lem:Phi_N}
Let $p\in (1,2]$ and $\alpha\in (\sfrac{1}{p},\sfrac{1}{p}+1)$. Then the composition 
	\[\varPhi_N:\quad  L^2(\partial D)\overset{\iota_1}{\hookrightarrow} W^{\alpha-1-\frac{1}{p},p}(\partial D) \overset{B_N}{\rightarrow } W^{\alpha,p}(D) \overset{\iota_2}{\hookrightarrow }H^{\alpha,p}(D) \simeq \Dom (\lambda\mathrm{I}-\Delta_N)^\frac{\alpha}{2}\]
is a bounded linear operator. 
\end{lemma}

\begin{lemma}
Let $p\in (1,2]$ and $\lambda>0$. Let also $\varPhi_N$ be the bounded linear operator from \autoref{lem:Phi_N}. The function $[t\mapsto (\lambda\mathrm{I}-\Delta_N)S_N(t)\varPhi_N]$ defined on $(0,\infty)$ takes values in the space $\mathscr{L}(L^2(\partial D);L^p(D))$. 
\end{lemma}

\begin{proof}
Let $\alpha\in (\sfrac{1}{p},\sfrac{1}{p}+1)$ be arbitrary. By \cite[Theorem 2.6.13 (c)]{Pazy}, the following inequality that yields the claim is obtained for $t>0$:
	\begin{align*}
		\|(\lambda\mathrm{I}-\Delta_N)S_N(t) \varPhi_N\|_{\mathscr{L}(L^2(\partial D);L^p(D))} & = \|(\lambda \mathrm{I}-\Delta_N)^{1-\frac{\alpha}{2}}S_N(t)(\lambda \mathrm{I}-\Delta_N)^\frac{\alpha}{2}\varPhi_N\|_{\mathscr{L}(L^2(\partial D);L^p(D))}\\
		&\lesssim t^{\frac{\alpha}{2}-1}\|(\lambda\mathrm{I}- \Delta_N)^\frac{\alpha}{2}\varPhi_N\|_{\mathscr{L}(L^2(\partial D);L^p(D))}.
	\end{align*}
The finiteness of the right-hand side of the above inequality follows from \autoref{lem:Phi_N}.
\end{proof}

Let $p\in (1,2]$ and $\lambda>0$ be fixed for the remainder of the subsection. In agreement with the convention of \autoref{sec:stoch_conv}, the function $[t\mapsto (\lambda\mathrm{I}-\Delta_N)S_N(t)\varPhi_N]$ is understood as the map $G_N$ defined by
	\begin{equation}
	\label{eq:kernel_1}
	 G_N(\varphi,u)(t) := \langle \varphi, (\lambda\mathrm{I}-\Delta_N)S_N(t)\varPhi_Nu\rangle = \int_D\varphi(x) [(\lambda\mathrm{I}-\Delta_N)S_N(t)\varPhi_Nu](x)\d{x}, \quad t>0,
	 \end{equation}
for $\varphi\in (L^p(D))^*$ and $u\in L^2(\partial D)$. Sufficient conditions for the integral process $(Y_t)_{t\geq 0}$ where
	\begin{equation}
	\label{eq:Y_t1}
		Y_t:= \int_0^t G_N(t-\cdot)\d{Z}, \quad t\geq 0,
	\end{equation}
to be a well defined $L^p(D)$-valued measurable process are given now. To this end, denote by $g_N$ the Green kernel of the semigroup $S_N$; that is, $g_N: (0,\infty)\times D\times D\rightarrow\R$ is the function such that the action of $S_N$ on $f\in L^p(D)$ is given by 
	\[[S_N(t)f](x) = \int_D g_N(t,x,y)f(y)\d{y}, \quad t>0, \quad x\in D.\]

\begin{lemma}
\label{lem:kernel_2}
If $H>\sfrac{d}{2}-\sfrac{1}{2p}$, then 
	\begin{equation*}
		I_{H,p,g_N}(t_0):=\int_D \left[\int_0^{t_0} \left(\int_{\partial D}|g_N(s,x,y)|^2\d{S}_y\right)^\frac{1}{2H}\d{s}\right]^{pH}\d{x}<\infty
	\end{equation*}
holds for every $t_0>0$.
\end{lemma}

\begin{proof}
We split the reasoning into two cases. 

\textit{Step 1: The uncorrelated noise case $H=\sfrac{1}{2}$}. In this case, $p\in (1,2]$ needs to satisfy the inequality $1+\sfrac{1}{p}>d$ which only allows $d=1$. By using the estimate 
	\begin{equation}
	\label{eq:Eidelman}
		|g_N(s,x,y)| \leq C  s^{-\frac{d}{2}}\mathrm{e}^{-\frac{|x-y|^2}{\frac{c}{2}s}}, \quad x\in D,\quad  y\in\partial D, \quad s>0,
	\end{equation}
with some finite positive constants $C$ and $c$ from \cite[Theorem 1.1]{Eid70}, there is the estimate
	\begin{align*}
		I_{\frac{1}{2},p,g_N}(t_0) & \lesssim \int_D\left[\int_{0}^{t_0} \int_{\partial D}s^{-1}\mathrm{e}^{-\frac{|x-y|^2}{cs}}\d{S}_y\d{s}\right]^\frac{p}{2}\d{x}\\
				& \lesssim \int_D\left[\int_0^{t_0} s^{-1}\mathrm{e}^{-\frac{\rho(x)^2}{cs}}\d{s}\right]^\frac{p}{2}\d{x} \\
				& \propto \int_D \left[\int_0^{\frac{ct_0}{\rho(x)^2}}\frac{1}{r}\mathrm{e}^{-\frac{1}{r}}\d{r}\right]^\frac{p}{2}\d{x}
	\end{align*}
where 
	\[ \rho(x) := \inf_{y\in\partial D}|x-y|\]
is the distance of $x\in D$ from the boundary $\partial D$. On the set $D_{\leq } := \{x\in D\,|ct_0\leq \rho(x)^2\}$, the inner integral can be estimated by
	\begin{equation*}
		\int_0^{\frac{ct_0}{\rho(x)^2}}\frac{1}{r}\mathrm{e}^{-\frac{1}{r}}\d{r} \leq \int_0^1\frac{1}{r}\mathrm{e}^{-\frac{1}{r}}\d{r} =: c_1<\infty
	\end{equation*}
and, as a consequence, we have that 
	\begin{equation}
	\label{eq:wiener_2}
		\int_{D_{\leq} } \left[\int_0^{\frac{ct_0}{\rho(x)^2}}\frac{1}{r}\mathrm{e}^{-\frac{1}{r}}\d{r}\right]^\frac{p}{2}\d{x} <\infty.
	\end{equation}
On $D_>:=\{x\in D\,|\, ct_0 >\rho(x)^2\}$, the inner integral can be estimated as follows:
	\begin{equation*}
		\int_0^{\frac{ct_0}{\rho(x)^2}} \frac{1}{r}\mathrm{e}^{-\frac{1}{r}}\d{r} \leq \int_0^1 \frac{1}{r}\mathrm{e}^{-\frac{1}{r}}\d{r} + \int_1^\frac{ct_0}{\rho(x)^2}\frac{1}{r}\d{r} = c_1 + \log\left(\frac{ct_0}{\rho(x)^2}\right)
	\end{equation*}
Consequently, we have that
\begin{align*}
\int_{D_>} \left[c_1+\log\left(\frac{ct_0}{\rho(x)^2}\right)\right]^\frac{p}{2}\d{x} & = 2\int_0^{\sqrt{ct_0}}\left[c_1+\log\left(\frac{ct_0}{x^2}\right)\right]^\frac{p}{2}\d{x} \\
& =2\sqrt{ct_0}\int_1^\infty y^{-2}\left(c_1+2\log y\right)^\frac{p}{2}\d{y} \\
& = 2\sqrt{ct_0}\int_0^\infty (c_1+2u)^\frac{p}{2}\mathrm{e}^{-u}\d{u}
\end{align*}
which if finite and thus it is shown that
	\begin{equation}
	\label{eq:wiener_1}
		\int_{D_>}\left[\int_0^{\frac{ct_0}{\rho(x)^2}} \frac{1}{r}\mathrm{e}^{-\frac{1}{r}}\d{r}\right]^\frac{p}{2}\d{x} <\infty.
	\end{equation}
Putting \eqref{eq:wiener_1} and \eqref{eq:wiener_2} together yields $I_{\frac{1}{2},p,g_N}(t_0) <\infty$.

\textit{Step 2: The regular noise case $H\in (\sfrac{1}{2},1)$.} In this case, only $d=1$ and $d=2$ can be considered. By using estimate \eqref{eq:Eidelman}, we obtain
	\begin{equation}
	\label{eq:reg_Neu_bdd}
		I_{H,p, G_N}(t_0) \lesssim \int_D \rho(x)^{-dp + 2pH} \left(\int_0^{\frac{2cHt_0}{\rho(x)^2}}r^{-\frac{d}{2H}} \mathrm{e}^{-\frac{1}{r}}\d{r}\right)^{pH}\d{x}.
	\end{equation}
Now, if $d= 2$, we have that 
	\begin{equation*}
		\int_D\rho(x)^{-2p+2pH}\left(\int_0^\frac{2cHt_0}{\rho(x)^2}r^{-\frac{1}{H}}\mathrm{e}^{-\frac{1}{r}}\d{r}\right)^{pH}\d{x} \leq \left(\int_D\rho(x)^{-2p+2pH}\d{x}\right) \left(\int_0^\infty r^{-\frac{1}{H}}\mathrm{e}^{-\frac{1}{r}}\d{r}\right)^{pH}.
	\end{equation*}
The first integral is convergent since $-2p+2pH>-1$ (this corresponds to the condition $H>\sfrac{d}{2}- \sfrac{1}{2p}$) and the second integral is the constant $\Gamma(\sfrac{1}{H}-1)^{pH}$. If $d=1$, we split the integral \eqref{eq:reg_Neu_bdd} into two integrals. On the set $D_{\leq}:=\{x\in D\,|\, 2cHt_0\leq \rho(x)^2\}$, the inner integral can be estimated as follows:
	\begin{equation*}
		\int_0^{\frac{2cHt_0}{\rho(x)^2}}r^{-\frac{1}{2H}} \mathrm{e}^{-\frac{1}{r}}\d{r} \leq \int_0^1r^{-\frac{1}{2H}} \mathrm{e}^{-\frac{1}{r}}\d{r}=: c_2<\infty.
	\end{equation*}
This means that we have the estimate
	\begin{equation}
	\label{eq:int_1}
		\int_{D_{\leq}}\rho(x)^{-p+2pH}\left(\int_0^{\frac{2cHt_0}{\rho(x)^2}} r^{-\frac{1}{2H}}\mathrm{e}^{-\frac{1}{r}}\d{r}\right)^{pH}\d{x} \leq c_2^{pH}\, c_3\, |D_{\leq}|
	\end{equation}
where $c_3:=\left(\mathrm{diam}\,D\right)^{-p+2pH}.$ Since $D$ is a bounded set, the right-hand side of estimate \eqref{eq:int_1} is finite. On the other hand, on the set $D_{>}:=\{x\in D\,|\, 2cHt_0>\rho(x)^2\}$, the inner integral can be estimated by
	\begin{equation*}
		\int_0^{\frac{2cHt_0}{\rho(x)^2}}r^{-\frac{1}{2H}} \mathrm{e}^{-\frac{1}{r}}\d{r} \leq \int_0^1r^{-\frac{1}{2H}} \mathrm{e}^{-\frac{1}{r}}\d{r} + \int_1^{\frac{2cHt_0}{\rho(x)^2}}r^{-\frac{1}{2H}}\d{r} = c_2 + \frac{1}{1-\frac{1}{2H}}\left[\left(\frac{\sqrt{2cHt_0}}{\rho(x)}\right)^{2-\frac{1}{H}}-1\right]
	\end{equation*}
and we obtain
	\begin{align*}
		\int_{D_>}\rho(x)^{-p+2pH}\left(\int_0^\frac{2cHt_0}{\rho(x)^2} r^{-\frac{1}{2H}}\mathrm{e}^{-\frac{1}{r}}\d{r}\right)^{pH}\d{x} & \\
		& \hspace{-4cm} \leq 2\int_0^{\sqrt{2cHt_0}}x^{-p+2pH}\left\{c_2 + \frac{1}{1-\frac{1}{2H}}\left[\left(\frac{\sqrt{2cHt_0}}{x}\right)^{2-\frac{1}{H}}-1\right]\right\}^{pH}\d{x}.
	\end{align*}
Since the last integral converges, it follows that $I_{H,p,g_N}(t_0)<\infty$. 
\end{proof}

\begin{lemma}
\label{lem:kernel_1}
There is the equality 
	\[G_N(\varphi,u)(t) = \int_D\int_{\partial D}\varphi(x) g_N(t,x,y)u(y)\d{S}_y\d{x}, \quad t >0, \]
for every $\varphi\in (L^p(D))^*$ and $u\in\mathscr{C}(\partial D)$. 
\end{lemma}

\begin{proof}
Let $u\in \mathscr{C}(\partial D)$. Then it follows that $B_Nu\in \mathscr{C}^2(D)\cap\mathscr{C}(\overline{D})$. Consequently, for $t>0$ and $x\in D$, it follows first by the symmetry of the kernel $g_N$ in the spatial variables and then by the second Green formula that 
	\begin{align*}
		[(\lambda \mathrm{I}-\Delta_N)S_N(t)\varPhi_N u](x) & = \int_D [(\lambda \mathrm{I}-\Delta_{N,y})g_N](t,x,y)[B_Nu](y)\d{y}\\
			& = \int_D g_N(t,x,y)\underbrace{[(\lambda\mathrm{I}-\Delta_{N})B_Nu](y)}_{=0}\d{y}\\
				& \hspace{1cm} + \int_{\partial D}g_N(t,x,y)\underbrace{[\partial_{\nu}B_Nu](y)}_{=u(y)}\d{S}_y \\
				& \hspace{1cm} -\int_{\partial D}\underbrace{[\partial_{\nu,y}g_N](t,x,y)}_{=0}[B_Nu](y)\d{S}_y\\
				& = \int_{\partial D}g_N(t,x,y)u(y)\d{S}_y \numberthis\label{eq:kernel_2}
	\end{align*}
where $\d{S}$ is the surface measure on $\partial D$. Inserting formula \eqref{eq:kernel_2} into equality \eqref{eq:kernel_1} yields the claim.
\end{proof}

\begin{lemma}
\label{lem:kernel_4}
If $H>\sfrac{d}{2}-\sfrac{1}{2p}$, then $G_N(\varphi,u)\in\mathscr{D}^H(0,t_0)$ for every $\varphi\in (L^p(D))^*$, $u\in\mathscr{C}(\partial D)$, and $t_0>0$.
\end{lemma}

\begin{proof}
By using \autoref{prop:Lp_embeddings}, \autoref{lem:kernel_1}, the Minkowski inequality, and the H\"older inequality successively, the following estimate is obtained:
	\begin{equation*}
		\|G_N(\varphi,u)\|_{\mathscr{D}^H(0,t_0)}  \lesssim S(\partial D)^{\frac{1}{2}} \|\varphi\|_{(L^p(D))^*} \|u\|_{\mathscr{C}(\partial D)} I_{H,p,g_N}^\frac{1}{p}(t_0)
	\end{equation*}
and the claim follows by \autoref{lem:kernel_2}.
\end{proof}

\begin{lemma}
\label{lem:kernel_3}
If $H>\sfrac{d}{2}-\sfrac{1}{2p}$, then 
	\begin{equation*}
		\int_D \|a_N(x)\|_{\mathscr{D}^H(0,t_0;L^2(\partial D))}^p\d{x} <\infty
	\end{equation*}
holds for every $t_0>0$. Here, $a_N(x)(t,y):= g_N(t,x,y)$ for $t>0$, $x\in D$, and $y\in \partial D$. 
\end{lemma}

\begin{proof}
By \autoref{prop:characterisation_of_D} and \autoref{prop:Lp_embeddings}, the integral can be estimated by $I_{H,p,g_N}(t_0)$ and the claim follows from \autoref{lem:kernel_2}.
\end{proof}

\begin{lemma}
\label{lem:kernel_5}
If $H>\sfrac{d}{2}-\sfrac{1}{2p}$, then \[\overline{G}_N^*(g\otimes u)(x) = \langle a_N(x),g\otimes u\rangle_{\mathscr{D}^H(0,t_0;L^2(\partial D))}, \quad  x\in D,\] holds for every $t_0>0$, $g\in\mathscr{C}^1([0,t_0])$, and $u\in\mathscr{C}(\partial D)$.
\end{lemma}

\begin{proof}
Let $\varphi\in (L^p(D))^*$. By \autoref{lem:kernel_4}, $G_N(\varphi,u)\in\mathscr{D}^H(0,t_0)$ and by \autoref{lem:kernel_1}, it follows that 
	\begin{equation}
	\label{eq:kernel_aN_1} \langle G_N(\varphi,u),g\rangle_{\mathscr{D}^H(0,t_0)} = \int_D \varphi(x) \left\langle \int_{\partial D} g_N(\,\cdot\,, x,y)u(y)\d{S}_y, g\right\rangle_{\mathscr{D}^H(0,t_0)}\d{x}.
	\end{equation}
On the other hand, we also have that 
	\begin{equation}
	\label{eq:kernel_aN_2}
	 \langle G_N(\varphi,u),g\rangle_{\mathscr{D}^H(0,t_0)} = \langle \overline{G}\varphi,g\otimes u\rangle_{\mathscr{D}^H(0,t_0;L^2(\partial D))} = \langle \varphi,\overline{G}_N^*(g\otimes u)\rangle
	 \end{equation}
and a comparison of formulas \eqref{eq:kernel_aN_1} and \eqref{eq:kernel_aN_2} yields the claim.
\end{proof}

\begin{proposition}
\label{prop:heat_bdry_noise_measurability}
If $p\in (1,2]$ is such that $H>\sfrac{d}{2}-\sfrac{1}{2p}$, then the process $(Y_t)_{t\geq 0}$ defined by formula \eqref{eq:Y_t1} is a well defined $L^p(D)$-valued stochastic process that admits a measurable version.
\end{proposition}

\begin{proof}
By \autoref{lem:kernel_3} and \autoref{lem:kernel_5}, the assumptions of \autoref{prop:kernel} are satisfied for the operator $\overline{G}_N^*$ and therefore, it follows that $\overline{G}_N^*\in \gamma(\mathscr{D}^H(0,t_0;L^2(\partial D));L^p(D))$ for any $t_0>0$. The proof is concluded by appealing to \autoref{rem:existence_Y_alpha}.
\end{proof}

\begin{appendices}

\section{Homogeneous fractional Sobolev spaces}
\label{sec:appendix}

In this appendix, we give a self-contained review of some results on homogeneous fractional Sobolev spaces that are needed in the paper; see, e.g., \cite{BerLof76}, \cite{Trie78}, or \cite{Trie83}.

\subsection{Homogeneous fractional Sobolev spaces on the real line}

\begin{definition}
\label{app:Ws_R}
Let $d\in\mathbb{N}$ and let $s\in\R$ be such that $|s|<\sfrac{d}{2}$. The \textit{homogeneous fractional Sobolev space} $\dot{W}^{s,2}(\R^d)$ is defined as the space 
	\begin{equation*}
		\dot{W}^{s,2}(\R^d) := \left\{F\in\mathscr{S}'(\R^d;\mathbb{C})\,\left|\, \exists\, h_F\in L^2(\R^d;\mathbb{C}) : \hat{F}=|x|^{-s}h_F\right.\right\}
	\end{equation*}
equipped with the norm $\|\cdot\|_{\dot{W}^{s,2}(\R^d)}$ that is defined by $\|F\|_{\dot{W}^{s,2} (\R^d)} := \|h_F\|_{ L^2(\R^d;\mathbb{C})}$. 
\end{definition}

The bound $2s<d$ in the above definition is imposed to guarantee that $|x|^{-s}h$ is a tempered function and we assume $-d<2s$ to have the space of Schwartz functions $\mathscr{S}(\R^d)$ included in $\dot{W}^{s,2}(\R^d)$. It is well-known that the space $\dot{W}^{s,2}(\R^d)$ is a separable complex Hilbert space that contains the space $\mathscr{C}_c^\infty(\R^d;\mathbb{C})$ as a dense subset. The following result shows gives a connection between $\dot{W}^{s,2}(\R^d)$ and a certain Lebesgue space. 

\begin{proposition}
\label{prop:Lp_embeddings}
If $0\leq s<\sfrac{d}{2}$, then
	\begin{equation*}
		L^{\frac{2d}{d+2s}}(\R^d;\mathbb{C})\subseteq \dot{W}^{-s,2}(\R^d), \quad \dot{W}^{s,2}(\R^d)\subseteq L^{\frac{2d}{d-2s}}(\R^d;\mathbb{C})
	\end{equation*}
with the embeddings being continuous.
\end{proposition}

\begin{proof}
If $f\in\mathscr{S}(\R^d;\mathbb{C})$, then
	\begin{equation*}
		\|f\|_{\dot{W}^{ -s,2}(\R^d)} = \||x|^{-s}\hat{f}\|_{ L^2(\R^d;\mathbb{C})} = C_{s,d}\||x|^{s-d}*f\|_{L^2(\R^d;\mathbb{C})} \leq C_s \|f\|_{L^\frac{2d}{d+2s}(\R^d;\mathbb{C})}
	\end{equation*}
holds by Parseval's equality and the Hardy-Littlewood-Sobolev inequality with some finite positive constant $C_s$ and $C_{s,d}$. Now, $\mathscr{S}(\R^d;\mathbb{C})$ is dense in $L^\frac{2d}{d+2s} (\R^d;\mathbb{C})$ and $\dot{W}^{-s,2}(\R^d)$ is embedded in $(\mathscr{S}',w^*)$. If $f,g\in\mathscr{S}(\R^d;\mathbb{C})$, then
	\begin{equation*}
		|\langle f,g\rangle_{ L^2(\R^d;\mathbb{C})}| \leq |\langle\hat{f},\hat{g}\rangle_{ L^2(\R^d;\mathbb{C})}| \leq \||x|^s\hat{f}\|_{ L^2(\R^d;\mathbb{C})}\||x|^{-s}\hat{g}\|_{L^2(\R^d;\mathbb{C})} \leq C_{s,d} \||x|^s\hat{f}\|_{L^2(\R^d;\mathbb{C})}\|g\|_{L^\frac{2d}{d+2s}(\R^d;\mathbb{C})}
	\end{equation*}
so that 
	\begin{equation*}
		\|f\|_{L^\frac{2d}{d-2s}(\R^d;\mathbb{C})} \leq C_{s,d} \|f\|_{\dot{W}^{s,2}(\R^d)}
	\end{equation*}
and the embedding follows from the density above. 
\end{proof} 

\begin{lemma}
Let $s\in\R$ be such that $|s|<\sfrac{d}{2}$ and let $Q$ be a non-degenerate cube in $\R^d$. Then the indicator function $\bm{1}_{Q}$ belongs to the space $\dot{W}^{s,2}(\R^d)$ if and only if $-\sfrac{d}{2}<s<\sfrac{1}{2}$.
\end{lemma}

\begin{proof}
The indicator function $\bm{1}_Q$ belongs to the space $\dot{W}^{s,2}(\R^d)$ if and only if the integral
	\begin{equation*}
		\int_{\R^d} |x|^{2s}\frac{|\sin (\lambda_1 x_1)|^2|\sin (\lambda_2 x_2)|^2\cdot\ldots\cdot |\sin (\lambda_d x_d)|^2}{x_1^2x_2^2\cdot\ldots\cdot x^2_d}\d{x}
	\end{equation*}
is convergent with $\lambda_1,\lambda_2, \ldots, \lambda_d >0$. Clearly, the integral grows as $s$ grows. Assume therefore, that $0<s<\sfrac{d}{2}$. In this case, we have that there is a finite constant $C$ such that
	\begin{equation*}
		\sum_{j=1}^d \int_{\R^d} |x_j|^{2s} \frac{|\sin (\lambda_1 x_1)|^2|\sin (\lambda_2 x_2)|^2\cdot\ldots\cdot |\sin (\lambda_d x_d)|^2}{x_1^2x_2^2\cdot\ldots\cdot x^2_d} \d{x} = C \int_0^\infty x^{2s-2}|\sin x|^2\d{x},
	\end{equation*}
holds. The last integral is convergent if and only if $s<\sfrac{1}{2}$. 
\end{proof}

For a non-degenerate cube $Q$ in $\mathbb{R}^d$ and $F\in \bigoplus_{|s|<\frac{1}{2}}\dot{W}^{s,2}(\R^d)$, define the average $J_QF$ by
	\begin{equation}
	\label{eq:J_Q}
		J_QF := \frac{1}{|Q|} \int_{\R^d}\widehat{F}(x)\overline{\widehat{\bm{1}_Q}(x)}\d{x}.
	\end{equation}

\begin{lemma}
For every $F\in \bigoplus_{|s|<\frac{1}{2}} \dot{W}^{s,2}(\R^d)$ and a non-degenerate cube $Q$ in $\R^d$ it holds that 
	\begin{equation*}
		\lim_{x\rightarrow\infty} J_{x+Q}F=0.
	\end{equation*}
Moreover, if $F\in \bigoplus_{|s|<\frac{1}{2}} \dot{W}^{s,2}(\R^d)$ can be represented as a tempered function, then 
	\begin{equation*}
		J_QF = \frac{1}{|Q|}\int_Q F(x)\d{x}.
	\end{equation*}
\end{lemma}

\begin{proof}
The first assertion is standard. As for the second assertion, let $\{\varphi_n\}_{n}$ be a mollification of $\frac{1}{|Q|}\bm{1}_Q$. Then it follows that 
	\begin{equation*}
		\lim_{n\rightarrow\infty} \left\|\frac{1}{|Q|} \bm{1}_Q - \varphi_n\right\|_{\dot{W}^{s,2}(\R)} = 0
	\end{equation*}
for every every $|s|<\sfrac{1}{2}$. Hence, 
	\begin{equation*}
		J_QF = \lim_{n\rightarrow \infty}\int_{\R^d} \widehat{F}(x)\overline{\widehat{\varphi}(x)} \d{x} = \lim_{n\rightarrow\infty} \int_{\R^d} F(x)\varphi_n(x)\d{x} = \frac{1}{|Q|} \int_Q F(x)\d{x}.
	\end{equation*}
\end{proof}

For $a\in\R^d$ and $r>0$, let $\{Q_k^{a,r}\}_{k\in\mathbb{Z}^d}$ be a decomposition of $\R^d$ to cubes, i.e. $Q^{a,r}_k$ is defined by \[Q_k^{a,r}:=a+r(k+[0,1)^d), \quad k\in\mathbb{Z}^d.\] For $F\in\bigoplus_{|s|<\frac{1}{2}}\dot{W}^{s,2}(\R^d)$, define a step function $M_{a,r}F$, that vanishes at infinity, by
	\begin{equation*}
		M_{a,r}F := \sum_{k\in\mathbb{Z}^d} (J_{Q_k^{a,r}}F)\bm{1}_{Q_k^{a,r}}
	\end{equation*}
where $J_{Q_k^{a,r}}F$ is the average of $F$ over the cube $Q_k^{a,r}$ defined by formula \eqref{eq:J_Q}. 

\begin{lemma}
\label{lem:approx_1}
Let $p\geq 1$ and $s\in\R$ be such that $0<s<\sfrac{1}{p}$. Then there is a finite positive constant $C_{s,d,p}$ that does not depend on the decomposition $\{Q_k^{a,r}\}_{k\in\mathbb{Z}^d}$ such that the inequality
	\begin{equation*}
		\left(\int_{\R^d}\int_{\R^d} \frac{|(M_{a,r}f)(x) - (M_{a,r}f)(y)|^p}{|x-y|^{d+sp}}\d{x}\d{y}\right)^\frac{1}{p} \leq C_{s,d,p} \left(\int_{\R^d}\int_{\R^d} \frac{|f(x)-f(y)|^p}{|x-y|^{d+sp}}\d{x}\d{y}\right)^\frac{1}{p}
	\end{equation*}
holds for every $f\in L^1_{\mathrm{loc}}(\R^d)$. 
\end{lemma}

\begin{proof}
Set $I_k:= k+ I$ for $k\in\mathbb{Z}^d$ where $I:=[0,1)^d$ and let $\beta >-d-1$. Then there exists a finite positive constant $C_{\beta,d}$ such that 
	\begin{equation*}
		\int_{I_k}\int_{I_l}|x-y|^\beta \d{x}\d{y} \leq C_{\beta,d} \|k-l\|_{\ell^\infty}^\beta
	\end{equation*}
holds for every $k,l\in\mathbb{Z}^d$. Indeed, the claim is obvious if either $\|k-l\|_{\ell^\infty} = 0$ or $\|k-l\|_{\ell^\infty}\geq 2$. Suppose therefore, that $\|k-l\|_{\ell^\infty}=1$ and denote $\mathfrak{J}:= \{j\in \{1,2,\ldots,d\}\,|\, |k_j-l_j|=1\}$ and $m:=|\mathfrak{J}|$. Clearly, we have that $m\geq 1$ and  
	\begin{equation*}
		\int_{I_k}\int_{I_l} |x-y|^\beta\d{x}\d{y} = \int_I\int_I \left(\sum_{j=1}^d\left|x_j-y_j+ |k_j-l_j|\right|^2\right)^\frac{\beta}{2}\d{x}\d{y} \leq 2^d \int_{2I} |z|^{m+\beta}\d{z}
	\end{equation*}
by replacing $y_j$ by $1-y_j$ for every $j\in\mathfrak{J}$ and by using the fact that for every non-negative function $h\in L^1(-1,2)$ it holds that 
	\begin{equation*}
		\frac{1}{2}\int_{-\frac{1}{2}}^{\frac{1}{2}} h(z)\d{z} \leq \int_0^1\int_0^1 h(x-y)\d{x}\d{y}\leq \int_{-1}^{1}h(z)\d{z}
	\end{equation*}
and 
	\begin{equation*}
		\int_0^1 zh(z)\d{z} \leq \int_0^1\int_0^1 h(x+y)\d{x}\d{y} \leq \int_0^2zh(z)\d{z}.
	\end{equation*}
In particular, we have for every $F\in L_{\mathrm{loc}}^1(\R^d)$ and $k\neq l$ that  
	\begin{align*}
		\int_{I_k}\int_{I_l} \frac{|(M_{0,1}F)(x)-(M_{ 0,1}F)(y)|^p}{|x-y|^{d+ps}}\d{x}\d{y} & \leq C_{d,s,p} \|k-l\|_{\ell^\infty}^{-d-sp} \left(\int_{I_k}\int_{I_l} |F(x)-F(y)|\d{x}\d{y}\right)^p \\
		& \leq \tilde{C}_{d,s,p} \int_{I_k}\int_{I_l} \frac{|F(x)-F(y)|^p}{|x-y|^{d+sp}}\d{x}\d{y}
	\end{align*}
holds with some finite positive constants $C_{d,s,p}$ and $\tilde{C}_{d,s,p}$ by H\"older's inequality. Finally, the general case is inferred for $f\in L^1_{\mathrm{loc}}(\R^d)$ by noting that 
	\begin{equation*}
		(M_{a,r}f)(x) = [M_{0,1}f(a+r\,\cdot)]\left(\frac{x-a}{r}\right), \quad x\in\R^d,
	\end{equation*}
holds for every $a\in\R^d$ and $r>0$. 
\end{proof}

\begin{lemma}
\label{lem:approx_2}
If $0\leq s<\sfrac{1}{2}$, then there is the convergence
	\begin{equation*}
		\lim_{r\rightarrow\infty} \sup_{a\in\mathbb{R}^d} \|M_{a,r}f - f\|_{\dot{W}^{s,2}(\R^d)}=0
	\end{equation*}
for every $f\in \dot{W}^{s,2}(\R^d)$. 
\end{lemma}

\begin{proof}
Observe that for every $t\in (s,\sfrac{1}{2})$ and every $f\in\dot{W}^{s,2}(\R^d)$, the estimate
	\begin{equation*}
		\|f\|_{\dot{W}^{s,2}(\R)}^2 = \int_{\left\{\left.x\in\R^d\right|\, |x|<R\right\}}|x|^{2s}|\hat{f}|^2\d{x}  + \int_{\left\{\left. x\in\R^d\right|\, |x|\geq R\right\}}|x|^{2s}|\hat{f}|^2\d{x} \leq R^{2s}\|f\|_{ L^2(\R;\mathbb{C})}^2 + \frac{1}{R^{2(t-s)}} \|f\|_{\dot{W}^{t,2}(\R)}^2
	\end{equation*}
holds for every $R>0$. Consequently, it is obtained by \autoref{lem:approx_1} that there exists a finite positive constant $C_{t,d}$ such that 
	\begin{equation*}
		\|M_{a,r}\varphi - \varphi\|_{\dot{W}^{s,2}(\R)}^2 \leq R^{2s} \|M_{a,r}\varphi - \varphi\|_{ L^2(\R;\mathbb{C})}^2 + \frac{C_{t,d}}{R^{2(t-s)}} \|\varphi\|_{\dot{W}^{t,2}(\R)}^2
	\end{equation*}
holds for every $R>0$ and every $\varphi\in\mathscr{C}_c^\infty(\R^d;\mathbb{C})$. Density of the space $\mathscr{C}_c^\infty(\R^d;\mathbb{C})$ in the space $\dot{W}^{s,2}(\mathbb{R}^d)$ concludes the proof. 
\end{proof}

\begin{proposition}
\label{thm:approx_by_step_functions}
If $|s|<\sfrac{1}{2}$, then there exists a finite positive constant $C_{s,d}$ such that the inequality
	\begin{equation*}
		\|M_{a,r}F\|_{\dot{W}^{s,2}(\R^d)} \leq C_{s,d} \|F\|_{\dot{W}^{s,2}(\R^d)}
	\end{equation*}
is satisfied for every $F\in\dot{W}^{s,2}(\R^d)$. Moreover, there is the convergence
	\begin{equation*}
		\lim_{r\rightarrow 0}\sup_{a\in\R^d} \|M_{a,r}F-F\|_{\dot{W}^{s,2}(\R^d)}=0
	\end{equation*}
for every $F\in \dot{W}^{s,2}(\R^d)$. 
\end{proposition}

\begin{proof}
The case $0\leq s<\sfrac{1}{2}$ is covered by \autoref{lem:approx_1} and \autoref{lem:approx_2}. If $\varphi,\psi\in \mathscr{S}(\R^d ;\mathbb{C})$, then we have that 
	\begin{equation*}
		\langle M_{a,r}\varphi,\psi\rangle_{L^2(\R;\mathbb{C})} = \langle \varphi, M_{a,r}\psi\rangle_{ L^2(\R;\mathbb{C})}.
	\end{equation*}
Hence, there exists a finite positive constant $C_{s,d}$ such that the inequality
	\begin{equation*}
		\|M_{a,r}\varphi\|_{\dot{W}^{-s,2}(\R^d)} \leq C_{s,d} \|\varphi\|_{\dot{W}^{-s,2}(\R^d)}
	\end{equation*}
is satisfied for every $\varphi\in\mathscr{C}^{\infty}_c(\R^d;\mathbb{C})$ by duality. Since the space $\mathscr{C}_c^\infty(\R^d;\mathbb{C})$ is dense in $\dot{W}^{-s,2}(\R^d)$, boundedness of $M_{a,r}$ follows. The approximation property then follows from an interpolation argument similar to the one in \autoref{lem:approx_2}.
\end{proof}

\begin{lemma}
\label{lem:localization}
For every $|s|<\sfrac{1}{2}$ there exists a finite positive constant $C_s$ such that the inequality
	\begin{equation*}
		 \sup\{\|\bm{1}_JF\|_{\dot{W}^{s,2}(\R)}\,|\, J \mbox{ interval in }\R\} \leq C_s \|F\|_{\dot{W}^{s,2}(\R)}
	\end{equation*}
is satisfied for every $F\in \dot{W}^{s,2}(\R)$. Moreover, for every $F\in\dot{W}^{s,2}(\R)$ and $x\in\R$, there are the convergences
	\begin{align*}
		& \lim_{y\rightarrow\infty}\|\bm{1}_{(y,\infty)}F\|_{\dot{W}^{s,2}(\R)} = 0,\\
		& \lim_{y\rightarrow x} \|\bm{1}_{(x,y)}F\|_{\dot{W}^{s,2}(\R)} = 0.
	\end{align*}
\end{lemma}

\begin{proof}
For $s=0$, the claim is clear. For $0<s<\sfrac{1}{2}$, the claim is first proved for $f\in\mathscr{C}_c^\infty(\R;\mathbb{C})$ by using the fractional Hardy's inequality (see, e.g., \cite{KruMalPer00}) and then extended to $f\in\dot{W}^{s,2}(\R)$ by a standard density argument. For $-\sfrac{1}{2}<s<0$, the claim follows from duality, interpolation, and \autoref{prop:Lp_embeddings}.
\end{proof}

\subsection{Homogeneous fractional Sobolev spaces on intervals}

\begin{definition}
\label{app:Ws_I}
Let $T$ be an interval in $\R$ and let $s\in\R$ be such that $|s|<\sfrac{1}{2}$. The \textit{homogeneous fractional Sobolev space} $\dot{W}^{s,2}(T)$ is defined as the space
	\begin{equation*}
		\dot{W}^{s,2}(T) := \left\{F: \mathscr{C}^\infty_c(T)\rightarrow \R \left| \,\exists \overline{F} \in\dot{W}^{s,2}(\R):  \overline{F}|_{\mathscr{C}^\infty_c(T)}=F\mbox{ and } \overline{F}|_{\mathscr{C}^\infty_c(\R\setminus T)} = 0\right.\right\}
	\end{equation*}
equipped with the norm $\|\cdot\|_{\dot{W}^{s,2}(T)}$ that is defined by $\|F\|_{\dot{W}^{s,2}(T)}:= \|\overline{F}\|_{\dot{W}^{s,2}(\R)}$. 
\end{definition}

It should be noted that it follows from \autoref{lem:localization} that if $F\in \dot{W}^{s,2}(T)$, its extension by zero outside $T$, denoted by $\overline{F}$ in \autoref{app:Ws_I}, is unique. The following two propositions and the subsequent remark provide an alternative way of computing the norms of functions in the space $\dot{W}^{s,2}(T)$.

\begin{proposition}
Let $T$ be an unbounded interval in $\R$ and let $0<s<\sfrac{1}{2}$. It holds that 
	\begin{equation*}
		\dot{W}^{s,2}(T) = \left\{f\in L^{\frac{2}{1-2s}}(T;\mathbb{C})\,\left|\, \vertiii{f}_{\dot{W}^{s,2}(T)}<\infty \right.\right\}
	\end{equation*}
where $\vertiii{\,\cdot\,}_{\dot{W}^{s,2}(T)}$ is defined by 
	\begin{equation*}
		\vertiii{f}_{\dot{W}^{s,2}(T)} := \left(\int_T\int_T\frac{|f(x)-f(y)|^2}{|x-y|^{1+2s}}\d{x}\d{y}\right)^\frac{1}{2}.
	\end{equation*}
Moreover, the norm $\vertiii{\,\cdot\,}_{\dot{W}^{s,2}(T)}$ is equivalent to the norm $\|\cdot\|_{\dot{W}^{s,2}(T)}$. 
\end{proposition}

\begin{proposition}
\label{prop:equiv_norm_1}
Let $T$ be a bounded interval in $\R$ and let $0<s<\sfrac{1}{2}$. It holds that 
	\begin{equation*}
		\dot{W}^{s,2}(T) = \left\{ f\in L^2(T;\mathbb{C})\,\left| \vertiii{f}_{\dot{W}^{s,2}(T)}<\infty\right.\right\}
	\end{equation*}
where the norm $\vertiii{\,\cdot\,}_{\dot{W}^{s,2}(T)}$ is in the case of $T$ being a bounded interval defined by 
\begin{equation*}
		\vertiii{f}_{\dot{W}^{s,2}(T)} := \left(\int_T |f(x)|^2\d{x}\right)^\frac{1}{2} + \left(\int_T\int_T \frac{|f(x)-f(y)|^2}{|x-y|^{1+2s}}\d{x}\d{y}\right)^\frac{1}{2}.
	\end{equation*}
Moreover, the norm $\vertiii{\,\cdot\,}_{\dot{W}^{s,2}(T)}$ is equivalent to the norm $\|\cdot\|_{\dot{W}^{s,2}(T)}$. 
\end{proposition}

\begin{remark}
\label{rem:equiv_norm}
There is also a description of the space $\dot{W}^{s,2}(T)$ with $-\sfrac{1}{2}<s<0$ in terms of duality. However, for the purposes in the present paper, it is only remarked that if $f\in L^\frac{2}{1-2s}(T)$ with $-\sfrac{1}{2}<s<0$ and $T$ a (bounded or unbounded) interval in $\R$, then $f\in\dot{W}^{s,2}(T)$ and it holds that
	\begin{equation*}
		\|f\|_{\dot{W}^{s,2}(T)}\eqsim \left( \int_T\int_T \frac{f(x)f(y)}{|x-y|^{1+2s}}\d{x}\d{y}\right)^\frac{1}{2}.
	\end{equation*}
\end{remark}

\subsection{Additional lemmas}
\label{sec:additional_lemmas}

In the first lemma, the norms of affine transformations of elements of homogeneous fractional Sobolev spaces are computed.

\begin{lemma}
\label{lem:affine_transform}
Let $a,b\in\R$, $a\neq 0$, and let $T\subseteq\R$ be an interval. Let $\tilde{T}\subseteq\R$ be the interval for which the map $g:\tilde{T}\rightarrow T$ given by $g(x)=ax+b$ is a bijection. Let $s\in\R$ be such that $|s|<\sfrac{1}{2}$, let $F\in \dot{W}^{s,2}(T)$ and define
	\begin{equation*}
		(F)_{a,b}(\varphi):= F\left(\frac{1}{|a|}\varphi\left(\frac{\cdot-b}{a}\right)\right), \quad \varphi\in \mathscr{C}_c^\infty(\tilde{T}).
	\end{equation*}
Then it holds that 
	\begin{equation*}
		\|(F)_{a,b}\|_{\dot{W}^{s,2}(\tilde{T})} = |a|^{s-\frac{1}{2}}\|F\|_{\dot{W}^{s,2}(T)}
	\end{equation*}
\end{lemma}

\begin{proof}
\textit{Step 1.} Initially note that for a tempered distribution $F\in\mathscr{S}'(\R)$, $(F)_{a,b}$ is defined by
	\begin{equation*}
		(F)_{a,b}(\varphi) := F\left(\frac{1}{|a|}\varphi\left(\frac{\cdot-b}{a}\right)\right), \quad \varphi\in\mathscr{S}(\R).
	\end{equation*}
In this case, it follows for $\varphi\in\mathscr{S}(\R)$ that the equality
	\begin{equation*}
		\widehat{(F)_{a,b}}(\varphi) = (F)_{a,b}(\hat{\varphi})=F\left(\frac{1}{|a|}\hat{\varphi}\left(\frac{\cdot-b}{a}\right)\right) = F(\widehat{\varPhi_{a,b}}) = \hat{F}(\varPhi_{a,b})
	\end{equation*}
is satisfied with $\varPhi$ being a Schwarz function that is given by
	\begin{equation*}
		\varPhi_{a,b}(x) := \mathrm{e}^{\mathrm{i}bx}\varphi(ax), \quad x\in\R. 
	\end{equation*}
\textit{Step 2.} The claim of the lemma is proved for $T=\R$ now. To this end, let $F\in\dot{W}^{s,2}(\R)$. Then its Fourier transform $\hat{F}$ is a tempered function and we have that 
	\begin{equation*}
		\int_{-\infty}^\infty \widehat{(F)_{a,b}}(x)\varphi(x)\d{x} = \widehat{(F)_{a,b}}(\varphi) = \hat{F}(\varPhi_{a,b}) = \int_{-\infty}^\infty\frac{1}{|a|}\mathrm{e}^{\mathrm{i}b\frac{x}{a}}\hat{F}\left(\frac{x}{a}\right)\varphi(x)\d{x}
	\end{equation*}
for $\varphi\in\mathscr{S}(\R)$ by \textit{Step 1} so that
	\begin{equation*}
		 \widehat{(F)_{a,b}}(x)  = \frac{1}{|a|}\mathrm{e}^{\mathrm{i}b\frac{x}{a}}\hat{F}\left(\frac{x}{a}\right), \quad x\in\R. 
	\end{equation*} 
Consequently, we have that 
	\begin{equation*}
		\|(F)_{a,b}\|_{\dot{W}^{s,2}(\R)}^2 = \int_{-\infty}^\infty |x|^{2s}\left|
	\frac{1}{|a|}\mathrm{e}^{\mathrm{i}b\frac{x}{a}}\hat{F}\left(\frac{x}{a}\right)\right|^2\d{x} = |a|^{2s-1}\int_{-\infty}^\infty|x|^{2s} |\hat{F}(x)|^2\d{x} = |a|^{2s-1}\|F\|_{\dot{W}^{s,2}(\R)}^2.
	\end{equation*}
\textit{Step 3:}  Let $T\subset\R$ be an interval and let $\tilde{T}\subset\R$ be the interval for which $g: \tilde{T}\rightarrow T$ is a bijection. Let $F\in\dot{W}^{s,2}(T)$ and let $\overline{F}$ be its extension by zero outside $T$. Define
	\begin{equation*}
		(F)_{*,a,b}(\varphi) := F\left(\frac{1}{|a|}\varphi \left( \frac{\cdot-b}{a}\right)\right), \quad \varphi\in \mathscr{C}_c^\infty(\tilde{T}).
	\end{equation*}
Then we have for every $\varphi\in\mathscr{C}_c^\infty(\tilde{T})$ that 
	\begin{equation*}
		(F)_{*,a,b}(\varphi) = F\left(\frac{1}{|a|}\varphi\left(\frac{\cdot-b}{a}\right)\right) = \overline{F}\left(\frac{1}{|a|}\varphi\left(\frac{\cdot-b}{a}\right)\right) = (\overline{F})_{a,b}(\varphi)
	\end{equation*}
which implies that $\overline{(F)_{*,a,b}} = (\overline{F})_{a,b}$. By using this as well as the result of \textit{Step 2}, we have that the chain of equalities 
	\begin{equation*}
		\|(F)_{*,a,b}\|_{\dot{W}^{s,2}(\tilde{T})}^2 = \|\overline{(F)_{*,a,b}}\|_{\dot{W}^{s,2}(\R)}^2 = \|(\overline{F})_{a,b}\|_{\dot{W}^{s,2}(\R)}^2 = |a|^{2s-1} \|\overline{F}\|_{\dot{W}^{s,2}(\R)}^2 = |a|^{2s-1}\|F\|_{\dot{W}^{s,2}(T)}^2
	\end{equation*}
holds and the claim is proved. 
\end{proof}

Finally, the following lemmas concern some properties of a restriction to a subinterval.

\begin{lemma}
\label{lem:restriction}
Let $T$ and $\tilde{T}$ be two intervals such that $\tilde{T}\subseteq T\subseteq \R$. Let $s\in\R$ be such that $|s|<\sfrac{1}{2}$. If $F\in\dot{W}^{s,2}(T)$, then the restriction of $F$ to the interval $\tilde{T}$, $F|_{\tilde{T}}$, belongs to the space $\dot{W}^{s,2}(\tilde{T})$ and there exists a finite positive constant $C_s$ such that the following inequality is satisfied:
	\begin{equation*}
		\|F|_{\tilde{T}}\|_{\dot{W}^{s,2}(\tilde{T})} \leq C_s\|F\|_{\dot{W}^{s,2}(T)}.
	\end{equation*}
\end{lemma}

\begin{proof}
By \autoref{lem:localization}, it follows that 
	\begin{equation*}
		\|F|_{\tilde{T}}\|_{\dot{W}^{s,2}(\tilde{T})} = \|\overline{F|_{\tilde{T}}}\|_{\dot{W}^{s,2}(\R)} = \|\bm{1}_{\tilde{T}} \overline{F}\|_{\dot{W}^{s,2}(\R)} \leq C_s \|\overline{F}\|_{\dot{W}^{s,2}(\R)} = \|F\|_{\dot{W}^{s,2}(T)}.
	\end{equation*}
\end{proof}

\begin{lemma}
\label{lem:norm_tends_to_zero}
Let $T\subseteq\R$ be an interval and let $s\in\R$ be such that $|s|<\sfrac{1}{2}$. Then for every $x\in\mathrm{Int}\, T$ and every $F\in\dot{W}^{s,2}(T)$ it holds that 
	\begin{equation*}
		\lim_{y\rightarrow x+} \|F|_{[x,y]}\|_{\dot{W}^{s,2}(x,y)}=0.
	\end{equation*}
\end{lemma}

\begin{proof}
It follows for $y\in T$, $y>x$, that
	\begin{equation*}
		\|F|_{[x,y]}\|_{\dot{W}^{s,2}(x,y)} = \|\overline{F|_{[x,y]}}\|_{\dot{W}^{s,2}(\R)} = \|\bm{1}_{[x,y]}\overline{F}\|_{\dot{W}^{s,2}(\R)}
	\end{equation*}
is satisfied and the claim follows by \autoref{lem:localization}.
\end{proof}

\end{appendices}

\section*{Acknowledgement}
\addcontentsline{toc}{section}{Acknowledgement}

This research was supported by the Czech Science Foundation through project No. 19-07140S.


\addcontentsline{toc}{section}{References}

{\small 
\begin{thebibliography}{58}
\providecommand{\natexlab}[1]{#1}
\providecommand{\url}[1]{\texttt{#1}}
\providecommand{\urlprefix}{URL }
\expandafter\ifx\csname urlstyle\endcsname\relax
  \providecommand{\doi}[1]{doi:\discretionary{}{}{}#1}\else
  \providecommand{\doi}[1]{doi:\discretionary{}{}{}\begingroup
  \urlstyle{rm}\url{#1}\endgroup}\fi
\providecommand{\bibinfo}[2]{#2}

\bibitem[{Al\`{o}s and Bonaccorsi(2002)}]{AlosBon02}
\bibinfo{author}{E.~Al\`{o}s}, \bibinfo{author}{S.~Bonaccorsi},
  \bibinfo{title}{Stochastic partial differential equations with {D}irichlet
  white-noise boundary conditions}, \bibinfo{journal}{Ann. Inst. H. Poincar\'e
  Probab. Statist.} \bibinfo{volume}{38} (\bibinfo{year}{2002})
  \bibinfo{pages}{125--154}.

\bibitem[{Al\`{o}s et~al.(2001)Al\`{o}s, Mazet, and Nualart}]{AlosMazNua01}
\bibinfo{author}{E.~Al\`{o}s}, \bibinfo{author}{O.~Mazet},
  \bibinfo{author}{D.~Nualart}, \bibinfo{title}{Stochastic calculus with
  respect to {G}aussian processes}, \bibinfo{journal}{Ann. Probab.}
  \bibinfo{volume}{29}~(\bibinfo{number}{2}) (\bibinfo{year}{2001})
  \bibinfo{pages}{766--801}.

\bibitem[{Al\`{o}s and Nualart(2003)}]{AlosNua03}
\bibinfo{author}{E.~Al\`{o}s}, \bibinfo{author}{D.~Nualart},
  \bibinfo{title}{Stochastic integration with respect to the fractional
  {B}rownian motion}, \bibinfo{journal}{Stoch. Stoch. Rep.}
  \bibinfo{volume}{75}~(\bibinfo{number}{3}) (\bibinfo{year}{2003})
  \bibinfo{pages}{129--152}.

\bibitem[{Amann(1993)}]{Ama93}
\bibinfo{author}{H.~Amann}, \bibinfo{title}{Nonhomogeneous Linear and
  Quasilinear Elliptic and Parabolic Boundary Value Problems}, in:
  \bibinfo{booktitle}{Function spaces, differential operators and nonlinear
  analysis}, vol. \bibinfo{volume}{133} of \emph{\bibinfo{series}{Teubner-Texte
  Math.}}, \bibinfo{publisher}{Teubner}, \bibinfo{address}{Stuttgart},
  \bibinfo{pages}{9--126}, \bibinfo{year}{1993}.

\bibitem[{Amann(1995)}]{Ama95}
\bibinfo{author}{H.~Amann}, \bibinfo{title}{Linear and quasilinear parabolic
  problems. {V}ol {I} Abstract linear theory}, vol.~\bibinfo{volume}{89} of
  \emph{\bibinfo{series}{Monographs in Mathematics}},
  \bibinfo{publisher}{Birkh\"auser Boston Inc.}, \bibinfo{address}{Boston},
  \bibinfo{year}{1995}.

\bibitem[{Bai and Taqqu(2014)}]{BaiTaq14}
\bibinfo{author}{S.~Bai}, \bibinfo{author}{M.~S. Taqqu},
  \bibinfo{title}{Generalized {H}ermite processes, discrete chaos and limit
  theorems}, \bibinfo{journal}{Stoch. Proc. Appl.}
  \bibinfo{volume}{124}~(\bibinfo{number}{4}) (\bibinfo{year}{2014})
  \bibinfo{pages}{1710--1739}.

\bibitem[{Bauer(1996)}]{Bau96}
\bibinfo{author}{H.~Bauer}, \bibinfo{title}{Probability Theory},
  vol.~\bibinfo{volume}{23} of \emph{\bibinfo{series}{De Gruyter studies in
  mathematics}}, \bibinfo{publisher}{Walter de Gruyter}, \bibinfo{year}{1996}.

\bibitem[{Benedek and Panzone(1961)}]{BenPan61}
\bibinfo{author}{A.~Benedek}, \bibinfo{author}{R.~Panzone}, \bibinfo{title}{The
  space {$L^P$}, with mixed norm}, \bibinfo{journal}{Duke Math. J.}
  \bibinfo{volume}{28} (\bibinfo{year}{1961}) \bibinfo{pages}{301--324}.

\bibitem[{Bergh and L{\"o}fstr{\"o}m(1976)}]{BerLof76}
\bibinfo{author}{J.~Bergh}, \bibinfo{author}{J.~L{\"o}fstr{\"o}m},
  \bibinfo{title}{Interpolation spaces}, \bibinfo{publisher}{Springer-Verlag
  Berlin Heidelberg New York}, \bibinfo{year}{1976}.

\bibitem[{Bonaccorsi and Tudor(2011)}]{BonTud11}
\bibinfo{author}{S.~Bonaccorsi}, \bibinfo{author}{C.~Tudor},
  \bibinfo{title}{Dissipative stochastic evolution equations driven by general
  {G}aussian and non-{G}aussian noise}, \bibinfo{journal}{J. Dyn. Diff. Equat.}
  \bibinfo{volume}{23} (\bibinfo{year}{2011}) \bibinfo{pages}{791--816}.

\bibitem[{Brze\'zniak et~al.(2015)Brze\'zniak, Goldys, Peszat, and
  Russo}]{BrzGolPesRus15}
\bibinfo{author}{Z.~Brze\'zniak}, \bibinfo{author}{B.~Goldys},
  \bibinfo{author}{S.~Peszat}, \bibinfo{author}{F.~Russo},
  \bibinfo{title}{Second order {PDE}s with {D}irichlet white noie boundary
  conditions}, \bibinfo{journal}{J.~Evol. Equ.} \bibinfo{volume}{15}
  (\bibinfo{year}{2015}) \bibinfo{pages}{1--26}.

\bibitem[{Brze\'zniak and van Neerven(2003)}]{BrzNee03}
\bibinfo{author}{Z.~Brze\'zniak}, \bibinfo{author}{J.~M. A.~M. van Neerven},
  \bibinfo{title}{Space-time regularity for linear stochastic evolution
  equations driven by spatially homogeneous noise}, \bibinfo{journal}{J. Math.
  Kyoto Univ.} \bibinfo{volume}{43}~(\bibinfo{number}{2})
  (\bibinfo{year}{2003}) \bibinfo{pages}{261--303}.

\bibitem[{Brze\'zniak et~al.(2012)Brze\'zniak, van Neerven, and
  Salopek}]{BrzNeeSal12}
\bibinfo{author}{Z.~Brze\'zniak}, \bibinfo{author}{J.~M. A.~M. van Neerven},
  \bibinfo{author}{D.~Salopek}, \bibinfo{title}{Stochastic evolution equations
  driven by {L}iouville fractional {B}rownian motion}, \bibinfo{journal}{Czech.
  Math. J.} \bibinfo{volume}{62} (\bibinfo{year}{2012}) \bibinfo{pages}{1--27}.

\bibitem[{Burkholder(2001)}]{Bur01}
\bibinfo{author}{D.~L. Burkholder}, \bibinfo{title}{Martingales and singular
  integrals in {B}anach spaces}, in: \bibinfo{booktitle}{Handbook of the
  Geometry of {B}anach spaces}, \bibinfo{publisher}{North Holland},
  \bibinfo{address}{Amsterdam}, \bibinfo{pages}{233--269},
  \bibinfo{year}{2001}.

\bibitem[{Cioica-Licht et~al.(2018)Cioica-Licht, Cox, and Veraar}]{CioCoxVer18}
\bibinfo{author}{P.~A. Cioica-Licht}, \bibinfo{author}{S.~G. Cox},
  \bibinfo{author}{M.~C. Veraar}, \bibinfo{title}{Stochastic integration in
  quasi-{B}anach spaces}, \bibinfo{note}{{p}reprint, arXiv: 1804.08947}.

\bibitem[{{\v{C}}oupek and Maslowski(2017)}]{CouMas17}
\bibinfo{author}{P.~{\v{C}}oupek}, \bibinfo{author}{B.~Maslowski},
  \bibinfo{title}{Stochastic evolution equations with {V}olterra noise},
  \bibinfo{journal}{Stoch. Proc. Appl.} \bibinfo{volume}{127}
  (\bibinfo{year}{2017}) \bibinfo{pages}{877--900}.

\bibitem[{{\v{C}}oupek et~al.(2018){\v{C}}oupek, Maslowski, and
  Ondrej\'{a}t}]{CouMasOnd18}
\bibinfo{author}{P.~{\v{C}}oupek}, \bibinfo{author}{B.~Maslowski},
  \bibinfo{author}{M.~Ondrej\'{a}t}, \bibinfo{title}{${L}^p$-valued stochastic
  convolution integral driven by {V}olterra noise}, \bibinfo{journal}{Stoch.
  Dyn.} \bibinfo{volume}{18}~(\bibinfo{number}{6}) (\bibinfo{year}{2018})
  \bibinfo{pages}{1850048 (22 p.)}.

\bibitem[{Da~Prato et~al.(1987)Da~Prato, Kwapie\'{n}, and
  Zabczyk}]{DaPKwaZab87}
\bibinfo{author}{G.~Da~Prato}, \bibinfo{author}{S.~Kwapie\'{n}},
  \bibinfo{author}{J.~Zabczyk}, \bibinfo{title}{Regularity of solutions of
  linear stochastic equations in {H}ilbert spaces},
  \bibinfo{journal}{Stochastics} \bibinfo{volume}{23} (\bibinfo{year}{1987})
  \bibinfo{pages}{1--23}.

\bibitem[{De~La Pe\~na and Gin\'e(1999)}]{PG}
\bibinfo{author}{V.~H. De~La Pe\~na}, \bibinfo{author}{E.~Gin\'e},
  \bibinfo{title}{Decoupling}, \bibinfo{publisher}{Springer-Verlag New York}, 	\bibinfo{year}{1999}.

\bibitem[{Decreusefond and \"{U}st\"{u}nel(1999)}]{DecrUstu99}
\bibinfo{author}{L.~Decreusefond}, \bibinfo{author}{A.~S. \"{U}st\"{u}nel},
  \bibinfo{title}{Stochastic analysis of the fractional {B}rownian motion},
  \bibinfo{journal}{Potential Anal.} \bibinfo{volume}{10}~(\bibinfo{number}{2})
  (\bibinfo{year}{1999}) \bibinfo{pages}{177--214}.

\bibitem[{Doob(1990)}]{Doob90}
\bibinfo{author}{J.~Doob}, \bibinfo{title}{Stochastic processes},
  \bibinfo{publisher}{Wiley}, \bibinfo{year}{1990}.

\bibitem[{Duncan et~al.(2002)Duncan, Maslowski, and Pasik-Duncan}]{DunMasDun02}
\bibinfo{author}{T.~Duncan}, \bibinfo{author}{B.~Maslowski},
  \bibinfo{author}{B.~Pasik-Duncan}, \bibinfo{title}{Fractional {B}rownian
  motion and stochastic equations in {H}ilbert spaces},
  \bibinfo{journal}{Stoch. Dyn.} \bibinfo{volume}{2}~(\bibinfo{number}{2})
  (\bibinfo{year}{2002}) \bibinfo{pages}{225--250}.

\bibitem[{Eidel'man and Ivasishen(1970)}]{Eid70}
\bibinfo{author}{S.~Eidel'man}, \bibinfo{author}{S.~Ivasishen},
  \bibinfo{title}{Investigations of the Green matrix for a homogeneous
  parabolic boundary value problem}, \bibinfo{journal}{Trans. Moscow Math.
  Soc.} \bibinfo{volume}{23} (\bibinfo{year}{1970}) \bibinfo{pages}{179--234}.

\bibitem[{Fabbri and Goldys(2009)}]{FabGol09}
\bibinfo{author}{G.~Fabbri}, \bibinfo{author}{B.~Goldys}, \bibinfo{title}{An
  {LQ} problem for the heat equation on the halfline with {D}irichlet boundary
  control and noise}, \bibinfo{journal}{SIAM J. Control Optim.}
  \bibinfo{volume}{48}~(\bibinfo{number}{3}) (\bibinfo{year}{2009})
  \bibinfo{pages}{1473--1488}.

\bibitem[{Issoglio and Riedle(2014)}]{IssRie14}
\bibinfo{author}{E.~Issoglio}, \bibinfo{author}{M.~Riedle},
  \bibinfo{title}{Cylindrical fractional {B}rownian motion in {B}anach spaces},
  \bibinfo{journal}{Stoch. Proc. Appl.}
  \bibinfo{volume}{124}~(\bibinfo{number}{11}) (\bibinfo{year}{2014})
  \bibinfo{pages}{3507--3534}.

\bibitem[{Jolis(2007)}]{Jol07}
\bibinfo{author}{M.~Jolis}, \bibinfo{title}{On the {W}iener integral with
  respect to the fractional {B}rownian motion on an interval},
  \bibinfo{journal}{J. Math. Anal. Appl.} \bibinfo{volume}{330}
  (\bibinfo{year}{2007}) \bibinfo{pages}{1115--1127}.

\bibitem[{Krugljak et~al.(2000)Krugljak, Malingranda, and
  Persson}]{KruMalPer00}
\bibinfo{author}{N.~Krugljak}, \bibinfo{author}{L.~Malingranda},
  \bibinfo{author}{L.~E. Persson}, \bibinfo{title}{On an elementary approach to
  the fractional {H}ardy inequality}, \bibinfo{journal}{Proc. Amer. Math. Soc}
  \bibinfo{volume}{128}~(\bibinfo{number}{3}) (\bibinfo{year}{2000})
  \bibinfo{pages}{727--734}.

\bibitem[{Kwapien et~al.(2016)Kwapien, Veraar, and Weis}]{KwaVerWei16}
\bibinfo{author}{S.~Kwapien}, \bibinfo{author}{M.~C. Veraar},
  \bibinfo{author}{L.~W. Weis}, \bibinfo{title}{${R}$-boundednes versus
  $\gamma$-boundedness}, \bibinfo{journal}{Arkiv f\"or Matematik}
  \bibinfo{volume}{54}~(\bibinfo{number}{1}) (\bibinfo{year}{2016})
  \bibinfo{pages}{125--145}.

\bibitem[{Lindemulder and Veraar(2020)}]{LinVer20}
\bibinfo{author}{N.~Lindemulder}, \bibinfo{author}{M.~Veraar},
  \bibinfo{title}{The heat equation with rough boundary conditions and
  holomorphic functional calculus}, \bibinfo{journal}{J. Diff. Equ.}
  \bibinfo{volume}{269}~(\bibinfo{number}{7}) (\bibinfo{year}{2020})
  \bibinfo{pages}{5832--5899}.

\bibitem[{Lindenstrauss and Tzafriri(1977)}]{LinTza77}
\bibinfo{author}{J.~Lindenstrauss}, \bibinfo{author}{L.~Tzafriri},
  \bibinfo{title}{Classical {B}anach Spaces {I}: Sequence Spaces},
  \bibinfo{publisher}{Springer-Verlag Berlin Heidelberg New York},
  \bibinfo{year}{1977}.

\bibitem[{Maejima and Tudor(2007)}]{MaeTud07}
\bibinfo{author}{M.~Maejima}, \bibinfo{author}{C.~A. Tudor},
  \bibinfo{title}{Wiener integrals with respect to the {H}ermite process and
  non central limit theorem}, \bibinfo{journal}{Stoch. Anal. Appl.}
  \bibinfo{volume}{25} (\bibinfo{year}{2007}) \bibinfo{pages}{1043--1056}.

\bibitem[{Maejima and Tudor(2012)}]{MaeTud12}
\bibinfo{author}{M.~Maejima}, \bibinfo{author}{C.~A. Tudor},
  \bibinfo{title}{Selfsimilar processes with stationary increments in the
  second {W}iener chaos}, \bibinfo{journal}{Prob. Math. Stat.}
  \bibinfo{volume}{32}~(\bibinfo{number}{1}) (\bibinfo{year}{2012})
  \bibinfo{pages}{167--186}.

\bibitem[{Mandelbrot and van Ness(1968)}]{MvN68}
\bibinfo{author}{B.~Mandelbrot}, \bibinfo{author}{J.~van Ness},
  \bibinfo{title}{Fractional {B}rownian motions, fractional noises and
  applications}, \bibinfo{journal}{SIAM Review}
  \bibinfo{volume}{10}~(\bibinfo{number}{4}) (\bibinfo{year}{1968})
  \bibinfo{pages}{422--437}.

\bibitem[{Maslowski(1995)}]{Mas95}
\bibinfo{author}{B.~Maslowski}, \bibinfo{title}{Stability of semilinear
  equations with boundary and pointwise noise}, \bibinfo{journal}{Ann. Scuola
  Norm. Sup. Pisa Cl. Sci.} \bibinfo{volume}{22}~(\bibinfo{number}{1})
  (\bibinfo{year}{1995}) \bibinfo{pages}{55--93}.

\bibitem[{Nourdin and Peccati(2012)}]{NouPec12}
\bibinfo{author}{I.~Nourdin}, \bibinfo{author}{G.~Peccati},
  \bibinfo{title}{Normal Approximations with {M}alliavin Calculus: From Stein's
  Method to Universality}, vol. \bibinfo{volume}{129} of
  \emph{\bibinfo{series}{Cambridge Tracts in Mathematics}},
  \bibinfo{publisher}{Cambridge University Press}, \bibinfo{year}{2012}.

\bibitem[{Nualart(2006)}]{Nua06}
\bibinfo{author}{D.~Nualart}, \bibinfo{title}{The {M}alliavin Calculus and
  Related Topics}, \bibinfo{publisher}{Springer - Verlag Berlin Heidelberg},
  \bibinfo{year}{2006}.

\bibitem[{Pazy(1983)}]{Pazy}
\bibinfo{author}{A.~Pazy}, \bibinfo{title}{Semigroups of Linear Operators and
  Applications to Partial Differential Equations}, Applied mathematical
  sciences, \bibinfo{publisher}{Springer}, \bibinfo{address}{New York},
  \bibinfo{year}{1983}.

\bibitem[{Pipiras and Taqqu(2000)}]{PipTaqq00}
\bibinfo{author}{V.~Pipiras}, \bibinfo{author}{M.~Taqqu},
  \bibinfo{title}{Integration questions related to fractional {B}rownian
  mtion}, \bibinfo{journal}{Probab. Theory Relat. Fields} \bibinfo{volume}{118}
  (\bibinfo{year}{2000}) \bibinfo{pages}{251--291}.

\bibitem[{Pipiras and Taqqu(2001)}]{PipTaqq01}
\bibinfo{author}{V.~Pipiras}, \bibinfo{author}{M.~Taqqu}, \bibinfo{title}{Are
  classes of deterministic integrands for fractional {B}rownian motion on an
  interval complete?}, \bibinfo{journal}{Bernoulli} \bibinfo{volume}{7}
  (\bibinfo{year}{2001}) \bibinfo{pages}{873--897}.

\bibitem[{Portal and Veraar(2019)}]{PorVer19}
\bibinfo{author}{P.~Portal}, \bibinfo{author}{M.~C. Veraar},
  \bibinfo{title}{Stochastic maximal regularity for rough time-dependent
  problems}, \bibinfo{journal}{Stoch. Partial. Differ. Equ.}
  \bibinfo{volume}{7} (\bibinfo{year}{2019}) \bibinfo{pages}{541--597}.

\bibitem[{Samorodintsky and Taqqu(1994)}]{SamTaq94}
\bibinfo{author}{G.~Samorodintsky}, \bibinfo{author}{M.~S. Taqqu},
  \bibinfo{title}{Stable non-{G}aussian random processes},
  \bibinfo{publisher}{Chapman \& Hall}, \bibinfo{year}{1994}.

\bibitem[{Schnaubelt and Veraar(2011)}]{SchVer11}
\bibinfo{author}{R.~Schnaubelt}, \bibinfo{author}{M.~Veraar},
  \bibinfo{title}{Stochastic Equations with Boundary Noise}, in:
  \bibinfo{editor}{J.~E. et~al.} (Ed.), \bibinfo{booktitle}{Parabolic Problems.
  Progress in Nonlinear Differential Equations and Their Applications},
  vol.~\bibinfo{volume}{80}, \bibinfo{publisher}{Springer, Basel},
  \bibinfo{pages}{609--629}, \bibinfo{year}{2011}.

\bibitem[{Seeley(1971)}]{See71}
\bibinfo{author}{R.~Seeley}, \bibinfo{title}{Norms and Domains of the Complex
  Powers {ABz}}, \bibinfo{journal}{Amer. J. Math.}
  \bibinfo{volume}{93}~(\bibinfo{number}{2}) (\bibinfo{year}{1971})
  \bibinfo{pages}{299--309}.

\bibitem[{Taqqu(2011)}]{Taqqu11}
\bibinfo{author}{M.~S. Taqqu}, \bibinfo{title}{The {R}osenblatt Process}, in:
  \bibinfo{editor}{R.~A. Davis}, \bibinfo{editor}{K.-S. Lii},
  \bibinfo{editor}{D.~N. Politis} (Eds.), \bibinfo{booktitle}{Selected Works of
  {M}urray {R}osenblatt}, \bibinfo{publisher}{Springer New York},
  \bibinfo{pages}{29--45}, \bibinfo{year}{2011}.

\bibitem[{Tindel et~al.(2003)Tindel, Tudor, and Viens}]{TinTudVie03}
\bibinfo{author}{S.~Tindel}, \bibinfo{author}{C.~A. Tudor},
  \bibinfo{author}{F.~Viens}, \bibinfo{title}{Stochastic evolution equations
  with fractional {B}rownian motion}, \bibinfo{journal}{Probab. Theory Relat.
  Fields} \bibinfo{volume}{127} (\bibinfo{year}{2003})
  \bibinfo{pages}{186--204}.

\bibitem[{Triebel(1978)}]{Trie78}
\bibinfo{author}{H.~Triebel}, \bibinfo{title}{Interpolation theory, function
  spaces, differential operators}, \bibinfo{publisher}{North Holland},
  \bibinfo{address}{Amsterdam}, \bibinfo{year}{1978}.

\bibitem[{Triebel(1983)}]{Trie83}
\bibinfo{author}{H.~Triebel}, \bibinfo{title}{Theory of Function Spaces},
  \bibinfo{publisher}{Birkh\"auser Basel}, \bibinfo{year}{1983}.

\bibitem[{Tudor(2008)}]{Tud08}
\bibinfo{author}{C.~A. Tudor}, \bibinfo{title}{Analysis of the {R}osenblatt
  process}, \bibinfo{journal}{ESAIM: Prob. Stat.} \bibinfo{volume}{12}
  (\bibinfo{year}{2008}) \bibinfo{pages}{230--257}.

\bibitem[{Tudor(2013)}]{Tud13}
\bibinfo{author}{C.~A. Tudor}, \bibinfo{title}{Analysis of Variations for
  Self-similar Processes}, \bibinfo{publisher}{Springer International
  Publishing}, \bibinfo{year}{2013}.

\bibitem[{van Neerven(2010)}]{Nee10}
\bibinfo{author}{J.~van Neerven}, \bibinfo{title}{$\gamma$-{R}adonifying
  operators: A Survey}, in: \bibinfo{booktitle}{The AMSI--ANU Workshop on
  Spectral Theory and Harmonic Analysis}, \bibinfo{publisher}{Centre for
  Mathematics and its Applications, Mathematical Sciences Institute, The
  Australian National University}, \bibinfo{address}{Canberra AUS},
  \bibinfo{pages}{1--61}, \bibinfo{year}{2010}.

\bibitem[{van Neerven et~al.(2015{\natexlab{b}})van Neerven, Veraar, and
  Weis}]{NeeVerWei15b}
\bibinfo{author}{J.~van Neerven}, \bibinfo{author}{M.~Veraar},
  \bibinfo{author}{L.~Weis}, \bibinfo{title}{Maximal regularity in
  $\gamma$-spaces}, \bibinfo{journal}{J. Evol. Equ.}
  \bibinfo{volume}{15}~(\bibinfo{number}{2})
  (\bibinfo{year}{2015}{\natexlab{b}}) \bibinfo{pages}{361--402}.

\bibitem[{van Neerven et~al.(2015{\natexlab{a}})van Neerven, Veraar, and
  Weis}]{NeeVerWei15}
\bibinfo{author}{J.~M. A.~M. van Neerven}, \bibinfo{author}{M.~Veraar},
  \bibinfo{author}{L.~W. Weis}, \bibinfo{title}{Stochastic integration in
  {B}anach spaces - a survey}, in: \bibinfo{booktitle}{Stochastic Analysis: A
  Series of Lectures}, vol.~\bibinfo{volume}{68} of
  \emph{\bibinfo{series}{Progress in Probability}},
  \bibinfo{publisher}{Birkh\"auser Basel}, \bibinfo{pages}{297--332},
  \bibinfo{year}{2015}{\natexlab{a}}.

\bibitem[{van Neerven and Veraar(2020)}]{NeeVer20}
\bibinfo{author}{J.~M. A.~M. van Neerven}, \bibinfo{author}{M.~C. Veraar},
  \bibinfo{title}{Maximal estimates for stochastic convolutions in $2$-smooth
  {B}anach spaces and applications to stochastic evolution equations},
  \bibinfo{note}{{p}reprint, arXiv:2006.08325}.

\bibitem[{van Neerven et~al.(2012)van Neerven, Veraar, and Weis}]{NeeVerWei12}
\bibinfo{author}{J.~M. A.~M. van Neerven}, \bibinfo{author}{M.~C. Veraar},
  \bibinfo{author}{L.~W. Weis}, \bibinfo{title}{Maximal ${L}^p$-regularity for
  stochastic evolution equations}, \bibinfo{journal}{SIAM J. Math. Anal.}
  \bibinfo{volume}{44}~(\bibinfo{number}{3}) (\bibinfo{year}{2012})
  \bibinfo{pages}{1372--1414}.

\bibitem[{van Neerven et~al.(2015{\natexlab{c}})van Neerven, Veraar, and
  Weis}]{NeeVerWei15c}
\bibinfo{author}{J.~M. A.~M. van Neerven}, \bibinfo{author}{M.~C. Veraar},
  \bibinfo{author}{L.~W. Weis}, \bibinfo{title}{On the ${R}$-boundedness of
  stochastic convolution operators}, \bibinfo{journal}{Positivity}
  \bibinfo{volume}{19} (\bibinfo{year}{2015}{\natexlab{c}})
  \bibinfo{pages}{355--384}.

\bibitem[{van Neerven and Zhu(2011)}]{NeeZhu11}
\bibinfo{author}{J.~M. A.~M. van Neerven}, \bibinfo{author}{J.~Zhu},
  \bibinfo{title}{A maximal inequality for stochastic convolutions in
  $2$-smooth {B}anach spaces}, \bibinfo{journal}{Electron. J. Probab.}
  \bibinfo{volume}{16} (\bibinfo{year}{2011}) \bibinfo{pages}{689--705}.

\bibitem[{Veraar and Weis(2011)}]{VerWei11}
\bibinfo{author}{M.~Veraar}, \bibinfo{author}{L.~Weis}, \bibinfo{title}{A note
  on maximal estimates for stochastic convolutions},
  \bibinfo{journal}{Czechoslovak Math. J.}
  \bibinfo{volume}{61}~(\bibinfo{number}{136}) (\bibinfo{year}{2011})
  \bibinfo{pages}{743--758}.

\bibitem[{Veraar and Yaroslavtsev(2016)}]{VerYar16}
\bibinfo{author}{M.~C. Veraar}, \bibinfo{author}{I.~Yaroslavtsev},
  \bibinfo{title}{Cylindrical continuous martingales and stochastic integration
  in infinite dimensions}, \bibinfo{journal}{Electron. J. Probab.}
  \bibinfo{volume}{21}~(\bibinfo{number}{59}) (\bibinfo{year}{2016})
  \bibinfo{pages}{1--53}.

\end{thebibliography}
}
\end{document}